\documentclass[ejsv2,preprint,noshowframe]{imsart}
\RequirePackage[numbers]{natbib}
\RequirePackage[colorlinks,citecolor=blue,urlcolor=blue]{hyperref}
\RequirePackage{graphicx}
\usepackage{tabularx}
\usepackage[table]{xcolor}
\usepackage{newtxtext,newtxmath}
\usepackage{pifont}
\newcommand{\cmark}{\ding{51}}%
\newcommand{\xmark}{\ding{55}}%
\definecolor{tableheader}{RGB}{225,235,245}
\startlocaldefs
\theoremstyle{plain}

\newtheorem{theorem}{Theorem}[section]
\newtheorem{lemma}[theorem]{Lemma}
\newtheorem{remark}[theorem]{Remark}
\newtheorem{proposition}[theorem]{Proposition}
\theoremstyle{definition}

\theoremstyle{remark}

\endlocaldefs
\usepackage{mathtools}
\usepackage{booktabs}
\usepackage{amsthm}
\usepackage{bbm}
\begin{document}
\begin{frontmatter}
\title{Minimax Rates for Estimating\\ Powered Sliced Wasserstein Functionals}
%\title{A sample article title with some additional note\thanksref{t1}}
\runtitle{Minimax Rates for Estimating Powered Sliced Wasserstein Functionals}
%\thankstext{T1}{A sample additional note to the title.}

\begin{aug}
%%%%%%%%%%%%%%%%%%%%%%%%%%%%%%%%%%%%%%%%%%%%%%%
%% Only one address is permitted per author. %%
%% Only division, organization and e-mail is %%
%% included in the address.                  %%
%% Additional information can be included in %%
%% the Acknowledgments section if necessary. %%
%% ORCID can be inserted by command:         %%
%% \orcid{0000-0000-0000-0000}               %%
%%%%%%%%%%%%%%%%%%%%%%%%%%%%%%%%%%%%%%%%%%%%%%%
\author[A]{\fnms{Binh Thuan Tran}\ead[label=e1]{binh-thuan.tran3@univ-eiffel.fr}}
% \author[B]{\fnms{Second}~\snm{Author}\ead[label=e2]{second@somewhere.com}\orcid{0000-0000-0000-0000}}
% \and
% \author[B]{\fnms{Third}~\snm{Author}\ead[label=e3]{third@somewhere.com}}
%%%%%%%%%%%%%%%%%%%%%%%%%%%%%%%%%%%%%%%%%%%%%%
%% Addresses                                %%
%%%%%%%%%%%%%%%%%%%%%%%%%%%%%%%%%%%%%%%%%%%%%%
\address[A]{LAMA, LIGM, Université Gustave Eiffel\printead[presep={,\ }]{e1}}

% \address[B]{Department,
% University or Company Name\printead[presep={,\ }]{e2,e3}}
\runauthor{Binh Thuan Tran}
\end{aug}

\begin{abstract}
We study the two-sample estimation of the powered Sliced Wasserstein functional under projected $r$-th moment assumptions. Let $N=n\wedge m$, where $n$ and $m$ are two sample sizes. We establish the minimax rates in the noncritical regimes. When $p<r<2p$, the minimax risk is of order $N^{-(r-p)/r}$, while for $r>2p$ it is of order $N^{-1/2}$. In both cases, the natural plug-in estimator is minimax rate-optimal. At the endpoint $r=p$, the minimax risk is bounded away from zero, so uniformly consistent estimation is impossible. At the critical order $r=2p$, we obtain the exact worst-case rate of the plug-in estimator and a lower bound for the minimax risk, leaving a logarithmic gap between the lower and upper bounds. The exact minimax rate in this critical case remains open.
\end{abstract}

% \begin{keyword}[class=MSC]
% \kwdgroup[type=primary]{\kwd{00X00}
% \kwd{00X00}}
% \kwdgroup[type=secondary]{\kwd{00X00}}
% \end{keyword}


\end{frontmatter}
%%%%%%%%%%%%%%%%%%%%%%%%%%%%%%%%%%%%%%%%%%%%%%
%% Please use \tableofcontents for articles %%
%% with 50 pages and more                   %%
%%%%%%%%%%%%%%%%%%%%%%%%%%%%%%%%%%%%%%%%%%%%%%
%\tableofcontents
\section{Introduction}
\subsection{Motivation and statistical question}
Optimal transport provides a geometrically meaningful framework for comparing probability distributions and has become an important tool in probability, statistics, and machine learning; see, for example, \cite{villani2008optimal,peyre2019computational}. Among the resulting discrepancies, the Wasserstein distance is particularly attractive because it is well defined under mild moment assumptions, does not require the distributions to admit densities with respect to a common dominating measure, and accounts for the geometry of the sample space. These properties have motivated applications in generative modeling, autoencoders, and domain adaptation \citep{arjovsky2017wasserstein,tolstikhin2017wasserstein,courty2016optimal}, as well as transport-inspired flow-matching methods \citep{lipman2022flow}.

Despite these advantages, Wasserstein distances face important computational and statistical limitations. For discrete measures supported on at most $K$ points, exact computation requires solving a large-scale linear program, and standard exact algorithms can have a worst-case complexity of order $O(K^3\log K)$ \citep{peyre2019computational}. Moreover, for probability measures on $\mathbb{R}^d$, where $d$ denotes the ambient dimension, empirical Wasserstein distances based on $n$ observations suffer from the curse of dimensionality, with convergence rates that can be as slow as $n^{-1/d}$ \citep{fournier2015rate,weed2019sharp}. These limitations can make their direct use impractical in large-scale and high-dimensional statistical problems.

These difficulties have motivated the development of more tractable alternatives based on low-dimensional projections. A prominent example is the Sliced Wasserstein distance, introduced by \citet{rabin2011sw}, which compares two multivariate probability measures by projecting them onto one-dimensional subspaces, computing the corresponding one-dimensional Wasserstein distances, and averaging over the unit sphere. By reducing a multivariate transport problem to a collection of one-dimensional ones, the Sliced Wasserstein distance benefits from the explicit quantile representation and computational simplicity of one-dimensional optimal transport while retaining relevant geometric information; see \cite{nadjahi2021sliced}. These properties have led to numerous applications; we refer to \citet[Chapter~6]{nguyen2025introduction} for a broader overview. 

In statistical applications, however, the underlying distributions are unknown and only independent observations from them are available. This raises the question of how accurately the population Sliced Wasserstein functional can be approximated by its empirical counterpart. 

In this work, we focus on the powered Sliced Wasserstein functional. The powered quantity is itself of interest in several applications. In particular, it has been used directly in  Sliced Wasserstein barycenter problems, generative modeling, discrete optimization and two-sample testing \citep{bonneel2015sliced,deshpande2018generative,tanguy2025properties, tran2025minimax}. It is therefore natural to study the statistical accuracy with which this population quantity can be estimated from samples.

Existing results provide upper bounds for the estimation of Sliced Wasserstein functionals under different assumptions, including compact-support and moment conditions (see Table~\ref{tab:overview} for a summary and Section~\ref{sec:relatedwork} for a detailed discussion). For the powered functional under projected moment assumptions, however, the optimal estimation rates are not fully understood. In particular, it remains unclear how the estimation rate changes as the available projected moment order varies. This leads to the main question considered in this paper: how does the optimal rate for estimating the powered Sliced Wasserstein functional depend on the available projected moment order? In the next subsection, we summarize our main contributions to the two-sample estimation of the powered Sliced Wasserstein functional, whose formal statistical setting is introduced in Section~\ref{sec:problemsetting}, and briefly compare our results with existing ones.
\subsection{Overview and main contributions}\label{sec:overview-maincontributions}
Throughout this subsection, $p$ denotes the order of the powered Sliced Wasserstein functional, $r$ the projected moment order, and $N\coloneqq n \wedge m$, where $n$ and $m$ are the two sample sizes. The precise statistical setting and the corresponding risks are formally defined in Section~\ref{sec:problemsetting}. Our main contributions are summarized as follows:
\begin{itemize}
\item \textbf{Minimax rates in the noncritical regimes.} In the regime $r>2p$, we remove the logarithmic factor appearing in the upper bound of \citet[Theorem~2 and Eq.~(19)]{manole2022minimax} and obtain the parametric rate $N^{-1/2}$. We also establish a matching minimax lower bound, showing that the natural plug-in estimator is minimax rate-optimal in this regime.

In the intermediate regime $p<r<2p$, we establish matching upper and lower bounds of order $N^{-(r-p)/r}$. This identifies the minimax rate and proves the minimax rate-optimality of the plug-in estimator in both noncritical regimes. Moreover, the two-distribution pairs used in our lower-bound construction for $p<r<2p$ may both be chosen off the diagonal, that is, $\mu_0 \ne \nu_0$ and $\mu_1 \ne \nu_1$. Hence, the same minimax rate persists even when the estimation problem is restricted to pairs of distinct distributions, confirming the expectation stated by \citet{manole2022minimax} that excluding the diagonal does not restore the parametric rate; see Theorem~\ref{thm:minimax-noncritical} and Remark~\ref{rem:distinct-distributions}.
\item \textbf{Boundary moment regimes and an improved critical lower bound.} At the minimal moment order $r=p$, we show that the minimax risk remains bounded away from zero. Consequently, no estimator is uniformly consistent under this projected moment assumption; see Remark~\ref{rem:minimal-moment-order}.

At the critical moment order $r=2p$, we establish matching upper and lower bounds for the worst-case risk of the plug-in estimator, showing that it is of order $\sqrt{\log(N+1)/{N}}$. For the minimax risk, we obtain a lower bound of order $[\log(N+1)]^{1/4}/{\sqrt{N}}$. Together with the upper bound provided by the plug-in estimator, this leaves a gap of order at most $[\log(N+1)]^{1/4}$. Consequently, the exact minimax rate and the minimax rate-optimality of the plug-in estimator remain open at the critical moment order; see Proposition~\ref{prop:r=2p}.

\item As a by-product, in the critical regime $r=2p$, Lemma~\ref{lem:projected-empirical-ratev2} improves the one-dimensional bound of \citet[Corollary~7.18]{bobkov2019one} for the expected powered $p$-Wasserstein error of the empirical measure, from $\log(n+1)/\sqrt{n}$ to $\sqrt{\log(n+1)/n}$. We also note that, combined with Lemma~\ref{lem:onedimensionalcaser=2p}, this upper bound is sharp, yielding the exact worst-case rate $\sqrt{\log(n+1)/n}$ for the expected powered $p$- Wasserstein error of the empirical measure over distributions with bounded $2p$-th moment.
\end{itemize}
Table~\ref{tab:overview} summarizes the closest existing two-sample estimation results for Sliced Wasserstein functionals and places our contributions in this context.
\setlength{\intextsep}{8pt plus 2pt minus 2pt}
\begin{table}[!htbp]
\centering
\caption{Overview of existing results for the two-sample estimation of Sliced Wasserstein (SW) functionals.}
\label{tab:overview}
\renewcommand{\arraystretch}{1.25}
\setlength{\tabcolsep}{4.5pt}
\begin{tabularx}{\textwidth}{>{\centering\arraybackslash}p{0.22\textwidth}
>{\centering\arraybackslash}p{0.13\textwidth}>{\centering\arraybackslash}X>{\centering\arraybackslash}p{0.13\textwidth}>{\centering\arraybackslash}p{0.14\textwidth}}
\toprule
\rowcolor{tableheader}
\textbf{Reference}&\textbf{Estimand}&\textbf{Assumption}&\textbf{Upper bound}&\textbf{Lower bound/ Optimality}
\\
\midrule
\citet[Corollary~2]{nadjahi2020statistical}
&$\operatorname{SW}_p$&
Finite $r$-th moments, $r>p$&\cmark&\xmark
\\
\citet[Theorem~1(i),Proposition~4]{manole2022minimax},
&$\mathrm{SW}_p$&$\mathcal{K}_{p,r}(b)$, $r>2p$;
see \eqref{sec1:eqclass}&\cmark&\cmark
\\
\citet[Proposition~1]{nietert2022statistical}
&$\mathrm{SW}_p^p$&Compact support
&\cmark&\xmark
\\
\citet[Theorem~2]{manole2022minimax}
&$\mathrm{SW}_p^p$&$\mathcal{K}_{p,r}(b)$, $r>2p$;see \eqref{sec1:eqclass}
&\cmark&\xmark
\\
\textbf{This work}&$\operatorname{SW}^p_p$
&$\mathcal{K}_{p,r}(b)$; see \eqref{sec1:eqclass}
&\cmark&\textbf{\cmark}
\\
\bottomrule
\end{tabularx}
\end{table}
We conclude the introduction with some notation and a description of the organization of the paper.

\textbf{Notation.} Throughout the paper, $n$ and $m$ denote the two sample sizes, and we
assume without loss of generality that $1\le n\le m$. We write $A\lesssim B$ if $A\le CB$ for some constant $C>0$, and $A\gtrsim B$ if $B\lesssim A$. We write $A\asymp B$ when both $A\lesssim B$ and
$A\gtrsim B$ hold. Subscripts indicate the parameters on which the implicit constants may depend; for example, $A\lesssim_{p,r}B$ means that the corresponding constant depends only on $p$ and $r$. We also use $C_p$, $C_{p,r}$, and similar notation for positive constants depending only on the indicated parameters. Their values may change from line to line. For $x\in\mathbb R$, we write $[x]_+ \coloneqq \max\{x,0\}$ for its positive part and $\lceil x\rceil$ for the ceiling of $x$. Finally, we denote by $\mathbb{S}^{d-1}$ the unit sphere in $\mathbb{R}^d$, by $\sigma$ the uniform probability measure on $\mathbb{S}^{d-1}$. For a probability measure $\xi$ on $\mathbb{R}^d$, a number $q\ge 1$, and $\theta \in \mathbb{S}^{d-1}$, we denote by $M_{\xi,r}(\theta)\coloneqq \mathbb{E}_{X \sim \xi}\left|\langle X,\theta \rangle\right|^r$.
% We also use the projected-moment notation $M_{Q,q}(\theta) \coloneqq \mathbb E_{X\sim Q} |\langle X,\theta\rangle|^q$ for $q\ge 1$, $Q\in\mathcal P(\mathbb R^d)$, and a direction $\theta$ on the unit sphere $\mathbb{S}^{d-1}$.

\textbf{Organization.} The remainder of the paper is organized as follows. Section~\ref{sec:problemsetting} formally introduces the problem setting and presents the necessary preliminaries. Section~\ref{sec:relatedworkandmainresults} reviews the related work and presents our main results. Section~\ref{sec:proofsupperbound} develops the upper-bound analysis for the plug-in estimator in the noncritical and critical moment regimes, while Section~\ref{sec:prooflowerbound} establishes the corresponding lower bounds for the minimax and plug-in risks. The omitted proofs are deferred to the appendices.
\section{Problem setting}\label{sec:problemsetting}
\subsection{Wasserstein and Sliced Wasserstein distances}
We denote by $\mathcal P(\mathbb R^d)$ the set of Borel probability measures on $\mathbb R^d$. For $p\geq1$, we further denote by $\mathcal P_p(\mathbb R^d)$ the subset of probability measures with finite $p$-th moment $$\mathcal P_p(\mathbb R^d) \coloneqq \left\{ \mu \in \mathcal P(\mathbb R^d): \mathbb E_{X\sim \mu}\|X\|^p <+\infty \right\},$$ where $\|\cdot\|$ denotes the Euclidean norm on $\mathbb R^d$.

For $p\geq 1$, the $p$-Wasserstein distance~\citep[Section~6]{villani2008optimal} between two probability measures $\mu,\nu\in\mathcal P_p(\mathbb R^d)$ is defined by
\begin{equation*}
\operatorname{W}_p(\mu,\nu) \coloneqq \left(\inf_{\pi\in\Gamma(\mu,\nu)}\int_{\mathbb R^d \times \mathbb R^d}\|x-y\|^pd\pi(x,y)\right)^{1/p},
\end{equation*}
where $\Gamma(\mu,\nu)$ denotes the set of all couplings between $\mu$ and $\nu$, that is, the set of probability measures on $\mathbb R^d\times\mathbb R^d$ whose first and second marginals are $\mu$ and $\nu$, respectively.

In general, the $p$-Wasserstein distance does not admit a closed-form expression, except in some special cases and in the one-dimensional setting; see, for instance, \citet[Section~2.4.1]{nadjahi2021sliced}. When $d=1$, the $p$-Wasserstein distance admits an explicit representation in terms of quantile functions. More precisely, for $\mu,\nu\in\mathcal P_p(\mathbb R)$, let $F_\mu$ and $F_\nu$ denote their cumulative distribution functions. We define their respective quantile functions by
\begin{equation*}
F_\mu^{-1}(u) \coloneqq \inf\left\{x\in\mathbb R: F_\mu(x)\ge u\right\},\qquad u\in(0,1),
\end{equation*}
and similarly for $F_\nu^{-1}$. Then,
$$
\operatorname{W}_p^p(\mu,\nu)=\int_0^1 |F_\mu^{-1}(u)-F_\nu^{-1}(u)|^pdu.
$$

The Sliced Wasserstein distance is defined by averaging Wasserstein distances between one-dimensional projections of the distributions. More precisely, let
$\mathbb S^{d-1}\coloneqq \left\{\theta\in\mathbb R^d:\|\theta\|=1\right\}$
denote the unit sphere in $\mathbb R^d$. For each $\theta\in\mathbb S^{d-1}$, let $\Pi^\theta:\mathbb R^d\to\mathbb R$ be the linear projection defined as $\Pi^\theta(x)\coloneqq \langle \theta,x\rangle.$
We denote by $\Pi^{\theta}_{\#}\mu$ the pushforward of $\mu$ by $\Pi^\theta$, that is, the distribution of $\langle \theta,X\rangle$ when $X\sim\mu$. Then, for $\mu,\nu\in\mathcal P_p(\mathbb R^d)$, the Sliced Wasserstein distance is defined as
$$\operatorname{SW}_p(\mu,\nu)\coloneqq \left(\int_{\mathbb{S}^{d-1}}\operatorname{W}^p_p\left(\Pi^{\theta}_{\#}\mu,\Pi^{\theta}_{\#}\nu\right)d\sigma(\theta)\right)^{1/p},
$$
where $\sigma$ denotes the uniform probability measure on $\mathbb{S}^{d-1}$. Accordingly, its $p$-th power is given by
$$
\operatorname{SW}_p^p(\mu,\nu) = \int_{\mathbb S^{d-1}} \operatorname{W}_p^p\left(\Pi^\theta_{\#}\mu, \Pi^\theta_{\#}\nu
\right)d\sigma(\theta).
$$

\subsection{Two-sample estimation problem and risks}
Let $X_1,\ldots,X_n$ and $Y_1,\ldots,Y_m$ be two independent samples, where $X_1,\ldots,X_n$ are i.i.d.\ with distribution $\mu$ and $Y_1,\ldots,Y_m$ are i.i.d.\ with distribution $\nu$. We denote the corresponding empirical measures by
$$\widehat{\mu}_n \coloneqq\frac{1}{n}\sum_{i=1}^n \delta_{X_i}, \qquad \widehat{\nu}_m\coloneqq\frac{1}{m}\sum_{j=1}^m \delta_{Y_j}.
$$
The corresponding plug-in estimator of $\operatorname{SW}_p^p(\mu,\nu)$ is
$\operatorname{SW}_p^p\left( \widehat{\mu}_n,\widehat{\nu}_m\right).$
We measure its error under absolute loss and define its worst-case risk over the class $\mathcal{C} \subset \mathcal{P}_p\left(\mathbb{R}^d\right)$ by
\begin{equation*}
\mathcal R_{n,m}^{\mathrm{plug}} \left(\mathcal C\right)
\coloneqq\sup_{\mu,\nu\in\mathcal C}\mathbb E_{\mu,\nu}\left|
\operatorname{SW}_p^p\left(\widehat{\mu}_n,\widehat{\nu}_m\right)-\operatorname{SW}_p^p(\mu,\nu)\right|,
\end{equation*}
where $\mathbb E_{\mu,\nu}$ denotes expectation under the joint distribution
$\mu^{\otimes n}\otimes\nu^{\otimes m}$, with the shorthand $\mathbb E$ used thereafter whenever the underlying sampling law is clear from context. This worst-case risk quantifies the uniform estimation error of the plug-in estimator over the considered class, rather than its pointwise error for a fixed pair $(\mu,\nu)$.

To compare the plug-in estimator with all possible estimation procedures, we adopt the standard minimax framework \citep{tsybakov2008introduction} and define the minimax risk by
\begin{equation*}
\mathcal R_{n,m}^{\star}\left(\mathcal C
\right)\coloneqq\inf_{\widehat T_{n,m}} \sup_{\mu,\nu\in\mathcal C}
\mathbb E_{\mu,\nu}\left|\widehat T_{n,m}-\operatorname{SW}_p^p(\mu,\nu)
\right|,
\end{equation*}
where the infimum is taken over all estimators $\widehat T_{n,m}$ based on the two samples. The minimax risk represents the best possible worst-case error over $\mathcal C$ and therefore provides the benchmark for evaluating the plug-in estimator. We say that the plug-in estimator is minimax rate-optimal over $\mathcal{C}$ if
\begin{equation*}
\mathcal R_{n,m}^{\mathrm{plug}}
\left(\mathcal C\right)\lesssim \mathcal R_{n,m}^{\star} \left(\mathcal C\right),
\end{equation*}
where the implicit constant is independent of $n$ and $m$. We next introduce the class of probability measures over which the preceding estimation problem is studied.
% Beyond the estimation problem considered here, it is interesting to note that the bounds developed over projected moment classes may also help extend the two-sample testing analysis of \citet{tran2025minimax} beyond bounded-support distributions.
\subsection{Projected moment classes}\label{sec:projectedmoment}
For a probability measure $\mu\in\mathcal P(\mathbb R^d)$, a real number $r \ge 1$, and a direction $\theta\in\mathbb S^{d-1}$. Following \citet{manole2022minimax}, for $r\geq p$ and $b>0$, we consider the projected moment class
\begin{equation}\label{sec1:eqclass}
\mathcal K_{p,r}(b) \coloneqq \left\{ \mu\in\mathcal P(\mathbb R^d): \int_{\mathbb S^{d-1}}
\left(M_{\mu,r}(\theta) \right)^{p/r} d\sigma(\theta)
\le b\right\}.
\end{equation}
This class is naturally adapted to the sliced structure of the functional $\operatorname{SW}_p^p$, since both the parameter space and the estimand are defined through averages of one-dimensional quantities over projection directions. In contrast to a uniform directional moment condition, the definition of $\mathcal K_{p,r}(b)$ only requires the projected moments to be controlled on average over the sphere.

We first record several basic properties of the projected moment class in the following lemma, whose proof is deferred to Appendix~\ref{appendixA}.
\begin{lemma}\label{lem:PropertiesK_pr}
Let $r\geq p\geq 1$ and $b>0$. Then, $ \mathcal K_{p,r}(b)\subseteq\mathcal P_p(\mathbb R^d).$
Moreover, for every $B,R>0$,
$$ \left\{ \mu\in\mathcal P_r(\mathbb R^d): \mathbb E_{X\sim\mu}\left[\|X\|^r\right]\le B
\right\}\subseteq
\mathcal K_{p,r}\left(B^{p/r}\right),
$$
and
$$
\left\{\mu\in\mathcal P(\mathbb R^d):
\operatorname{supp}(\mu)\subseteq\mathbb B(0,R)\right\}\subseteq
\mathcal K_{p,r}(R^p),
$$
where $\mathbb B(0,R)$ denotes the closed Euclidean ball of radius $R$ centered at the origin. 

Finally, for every $\mu,\nu\in\mathcal K_{p,r}(b)$, $\operatorname{SW}_p^p(\mu,\nu) \le 2^p b.$
\end{lemma}

This lemma shows that the condition defining $\mathcal K_{p,r}(b)$ is sufficient to ensure that the powered Sliced Wasserstein functional is finite and uniformly bounded over the parameter space. The endpoint $r=p$ is the natural starting point, since it is the minimal moment order under which the projected moment condition guarantees finiteness of the powered Sliced Wasserstein functional for every pair of distributions in the parameter class. As shown later in Remark~\ref{rem:minimal-moment-order}, however, this condition is not sufficient for uniformly consistent estimation: the minimax risk over $\mathcal K_{p,p}(b)$ remains bounded away from zero. We therefore consider the full range $r\geq p$ to determine how increasingly strong projected moment assumptions affect estimability and the optimal convergence rate.

The preceding inclusions also show that $\mathcal K_{p,r}(b)$ contains several standard subclasses, including distributions with bounded Euclidean $r$-th moments and compactly supported distributions. Hence, our uniform upper bounds remain valid on these subclasses after an appropriate adjustment of the class radius. Their minimax rates, however, may differ and would require separate lower-bound analyses.

Beyond the estimation problem considered here, it is interesting to note that the bounds developed over projected moment classes may also help extend the two-sample testing analysis of \citet{tran2025minimax} beyond bounded-support distributions.
% Beyond the estimation problem studied here, sharp bounds over projected moment classes may also help extend the two-sample testing analysis of \citet{tran2025minimax} beyond bounded support.
% \todo[inline]{Transition between the sections?}

With this setting in place, we next summarize our main results and discuss the closest related work. 
\section{Related work and main results}\label{sec:relatedworkandmainresults}
\subsection{Related work}\label{sec:relatedwork}
\citet[Theorem~1]{fournier2015rate} established nonasymptotic bounds for $\mathbb E\operatorname{W}_p^p(\widehat{\mu}_n,\mu)$, where $\widehat{\mu}_n$ is the empirical measure based on $n$ i.i.d.\ observations from $\mu$. Their results hold for every $p\ge 1$ and dimension $d \ge 1$, under the assumption that $\mu$ has a finite $r$-th moment for some $r>p$, with rates depending on the relation between $d,p$ and $r$. These results concern distribution estimation by the empirical measure under powered Wasserstein loss whereas our objective is the two-sample estimation of the powered Sliced Wasserstein functional. Specializing their results to dimension one provides some of the bounds needed in our analysis. Indeed, when $p=1$, the reverse triangle inequality directly reduces the two-sample estimation error to two terms that can be controlled using their results. For $p>1$, however, no analogous direct reduction is available. Our decomposition in Lemma~\ref{lem:decomposition} then involves, in addition to terms that can still be controlled using their results, weighted quantile terms that require a separate analysis. Deriving suitable bounds for these weighted quantile terms under projected moment assumptions is a main technical contribution of our upper-bound analysis.

For the two-sample estimation problem, \citet[Theorem~1, Corollary~3.1]{staudt2025convergence} studied the convergence of empirical optimal transport costs under finite $(q+\epsilon)$-th moment assumptions, where $p<q \le 2p$ and $\epsilon>0$. Their result holds in arbitrary dimension and allows for unbounded supports. After specializing their result to the one-dimensional Wasserstein setting and translating their moment condition into our notation, one can recover the parametric rate in the regime $r>2p$. However, the same specialization does not yield the sharp rates obtained here when $p<r\le 2p$. We provide the detailed calculation of this comparison in Appendix~\ref{appendix:SpecializationStaudt} Besides, the main proof strategy of \citet{staudt2025convergence} relies on independence between the two samples in the two-sample setting (see their Section~6). In contrast, our upper-bound argument does not require this assumption: after the decomposition in Lemma~\ref{lem:decomposition}, each resulting random term depends on only one of the two empirical samples and therefore can be controlled separately. Although \citet{staudt2025convergence} also provide an alternative approach in Appendix~A that removes the independence assumption, the assumptions of this approach do not cover the cost $c(x,y)=|x-y|^p$ when $p>1$, which is the cost defining the powered $p$-Wasserstein distance in one dimension. The verification is given in Appendix~\ref{appendix:verificationStaudt}

% \citet{fournier2015rate} established nonasymptotic one-sample bounds for
% $ \mathbb E\operatorname{W}_p^p(\widehat{\mu}_n,\mu),$
% where $\widehat{\mu}_n$ is the empirical measure based on $n$ i.i.d.\ observations from $\mu$. Their results hold for every $p\ge 1$ and dimension $d\ge1$, under the assumption that $\mu$ has a finite $r$-th moment for some $r>p$, with rates depending on the relation between $d$, $p$, and $r$. These results concern one-sample approximation, whereas our objective is the two-sample estimation of the powered Sliced Wasserstein functional. Specializing their result to dimension one yields the noncritical empirical Wasserstein estimates used in Lemma~\ref{lem:projected-empirical-rate}. When $p=1$, the reverse triangle inequality directly reduces the two-sample error to two one-sample Wasserstein errors. For $p>1$, however, no analogous direct reduction is available for the powered functional, and obtaining the corresponding two-sample rates requires additional analysis. Indeed, Lemma~\ref{lem:decomposition} decomposes the two-sample error into the one-sample Wasserstein terms $\mathsf{C}_{n,\theta}$ and $\mathsf{D}_{m,\theta}$ and the weighted quantile terms $\mathsf{A}_{n,\theta}$ and $\mathsf{B}_{m,\theta}$. The latter are not controlled by standard one-sample Wasserstein estimates and must be treated separately. Deriving suitable bounds for these weighted quantile terms under projected moment assumptions is a main technical contribution of our upper-bound analysis.

\citet[Corollary~2]{nadjahi2020statistical} obtained dimension-free
two-sample upper bounds for the Sliced Wasserstein metric
$\operatorname{SW}_p$ under the assumption that, for some $r>p$, $\mathbb E_{X\sim\mu}\|X\|^r<+\infty$ and 
$\mathbb E_{Y\sim\nu}\|Y\|^r<+\infty$. Their result concerns the metric $\operatorname{SW}_p$, rather than
the powered functional $\operatorname{SW}_p^p$, and Corollary~2 does not establish matching lower bounds or minimax optimality of the
resulting rates.

A two-sample result for the powered functional was obtained by \citet[Proposition~1]{nietert2022statistical}. If $\mu$ and $\nu$ are compactly supported, with the diameters of their supports bounded by $R$, they prove
\begin{equation*}
\mathbb{E}\left|\operatorname{SW}^p_p(\widehat{\mu}_n,\widehat{\nu}_n)-\operatorname{SW}^p_p(\mu,\nu)\right| \le C_{p,R}n^{-1/2}.
\end{equation*}
This establishes a parametric rate under compact support, but does not cover distributions satisfying only moment assumptions.
% Their proof first uses a comparison inequality that bounds the difference between powered $\operatorname{W}_p$ costs by one-sample $\operatorname{W}_1$ errors for compactly supported distributions (see, for example \citet[Lemma~4]{goldfeld2024statistical}). Using the integral
% representation of $\operatorname{W}_1$, they then obtain
% \begin{equation*}
% \mathbb{E}\operatorname{W}_1\left(\widehat{\mu}_n,\mu\right) \le \frac{1}{\sqrt{n}}\int_{\mathbb{R}}\sqrt{F_{\mu}(x)(1-F_{\mu}(x))}dx,
% \end{equation*}
% where $F_{\mu}$ denotes the CDF of $\mu$. An analogous estimate holds for the second distribution. The bounded-support assumption is then used to control the length of the interval over which the CDF error is integrated. Under moment assumptions alone, both the comparison step and the tail integral require additional analysis.

The closest work to ours is \citet{manole2022minimax}, who studied the same projected moment class $\mathcal{K}_{p,r}(b)$. In the regime $r>2p$, they obtained convergence rates for both the Sliced Wasserstein distance $\operatorname{SW}_p$ and its powered version $\operatorname{SW}^p_p$. For $\operatorname{SW}_p$, Proposition~4 gives an upper bound of order $n^{-1/(2p)}$, and their lower-bound construction yields a matching minimax lower bound; hence, the plug-in estimator is minimax rate-optimal in this regime. \footnote{Although Theorem~1(i) of \citet{manole2022minimax} is stated for $\delta>0$, the lower-bound construction in their Appendix~C also applies to the case $\delta =0$, yielding a lower bound that matches the rate of the upper bound in their Proposition~4.} For the powered functional $\operatorname{SW}^p_p$, however, \citet[Theorem~2 and Eq.~(19)]{manole2022minimax} only established the upper bound
\begin{equation*}
\mathbb E\left|\operatorname{SW}_p^p(\widehat{\mu}_n,\widehat{\nu}_m)-\operatorname{SW}_p^p(\mu,\nu)\right|\lesssim_{p,r}b
\left(\sqrt{\frac{\log n}{n}}
+\sqrt{\frac{\log m}{m}} \right),
\end{equation*}
without a matching minimax lower bound. Thus, even in the regime $r>2p$, the sharp estimation rate for the powered functional remained unresolved. In contrast, our localized empirical quantile estimates and peeling arguments remove this logarithmic loss and recover the parametric rate. For comparison, Appendix~\ref{appendix:comparisonmanole} revisits the approach of \citet{manole2022minimax} and explains where the logarithmic factor enters their argument. Besides, they also observed that $r\ge p$ is the minimal moment requirement under which the projected moment condition guarantees the finiteness of $\operatorname{SW}_p^p$. When $r<p$, the class may contain distributions for which this functional is infinite. For $p<r<2p$, they only noted that, in the special case $d=1$ and $\mu=\nu$, the rate cannot be improved beyond $n^{-(r-p)/r}$, while the sharp rate under the restriction $\mu\neq\nu$ remained unknown. They further
expected that the parametric rate should remain unattainable even under this restriction. Our lower-bound construction confirms this expectation in arbitrary
dimension and shows that the rate $n^{-(r-p)/r}$ remains sharp even when the estimation problem is restricted to pairs satisfying $\mu\neq\nu$.

We now state the main results obtained in this paper.
\subsection{Main results}\label{sec:mainresults}
% The main results of this paper focus on the minimax estimation of the
% powered Sliced Wasserstein functional over the projected moment classes
% $\mathcal K_{p,r}(b)$ and on how the optimal rates depend on the
% available projected moment order $r$. We first characterize the minimax rates in the two noncritical regimes.
We study the minimax estimation of the powered Sliced Wasserstein functional over $\mathcal{K}_{p,r}(b)$ and characterize how the optimal rates depend on the moment order $r$. We recall that $n$ and $m$ denote the two sample sizes and assume that $1 \le n \le m$. First, we consider the two noncritical regimes.
\begin{theorem}
\label{thm:minimax-noncritical}
Let $p\ge1$, $r> p$, $b>0$. Suppose that $r\neq 2p$. The plug-in estimator is minimax rate-optimal over
$\mathcal K_{p,r}(b)$. More precisely,
\begin{equation*}
\mathcal R_{n,m}^{\mathrm{plug}}
\left(\mathcal K_{p,r}(b)
\right) \asymp_{p,r} \mathcal R_{n,m}^{\star} \left( \mathcal K_{p,r}(b)
\right)\asymp_{p,r}b
\begin{cases}
n^{-(r-p)/r}, & p< r<2p,\\[2mm]
n^{-1/2}, & r>2p.
\end{cases}
\end{equation*}
\end{theorem}

Theorem~\ref{thm:minimax-noncritical} identifies two statistically
distinct regimes away from the moment orders $r=p$ and
$r=2p$. When $p<r<2p$, the minimax risk converges to zero at
the nonparametric rate
$b\,n^{-(r-p)/r},$ whereas the parametric rate $b\,n^{-1/2}$ is attained when $r>2p$.

\begin{remark}[Distinct distributions]
\label{rem:distinct-distributions}
The lower bound in the regime $p<r<2p$ remains valid when the
parameter space is restricted to pairs of distinct distributions. Indeed, the testing pairs constructed in Subsection~\ref{sec:lowerboundp<r<2p} satisfy $\mu_0\neq\nu_0$ and $\mu_1\neq\nu_1$. Consequently,
\begin{equation*}
\inf_{\widehat T}
\sup_{\substack{\mu,\nu\in\mathcal K_{p,r}(b)\\\mu\neq\nu
}}\mathbb E_{\mu,\nu}\left|
\widehat T-\operatorname{SW}_p^p(\mu,\nu)
\right|\asymp_{p,r}bn^{-(r-p)/r}.
\end{equation*}
Thus, the same minimax rate persists even when the parameter space is
restricted to pairs of distinct distributions, confirming the
expectation stated by \citet{manole2022minimax} that excluding the
diagonal $\mu=\nu$ does not restore the parametric rate.
\end{remark}

\begin{remark}[The minimal moment order]
\label{rem:minimal-moment-order}
The construction used to establish the minimax lower bound in the
regime $p<r<2p$ remains valid at the endpoint $r=p$. Consequently,
$$
\mathcal R_{n,m}^{\star} \left( \mathcal K_{p,p}(b) \right) \gtrsim_p b.
$$
Thus, the minimax risk does not vanish as the sample sizes increase,
and no estimator is uniformly consistent over
$\mathcal K_{p,p}(b)$.
\end{remark}

We next turn to the critical moment order $r=2p$. 
% In this regime, the
% % worst-case risk of the plug-in estimator can be characterized up to
% % multiplicative constants, while the exact minimax rate remains open.
\begin{proposition}[Critical moment order]
\label{prop:r=2p}
Let $p\ge1$, $b>0$. Then
\begin{equation}
\label{eq:critical-plugin-risk}
\mathcal R_{n,m}^{\mathrm{plug}}
\left(\mathcal K_{p,2p}(b)
\right)\asymp_pb\sqrt{\frac{\log(n+1)}{n}}.
\end{equation}
Moreover, 
\begin{equation}
\label{eq:critical-minimax-lower}
\mathcal R_{n,m}^{\star}
\left(\mathcal K_{p,2p}(b)\right)
\gtrsim_p b \frac{\left[\log(n+1)\right]^{1/4}
}{\sqrt n}.
\end{equation}
\end{proposition}

The matching upper and lower bounds in
\eqref{eq:critical-plugin-risk} show that the logarithmic factor is
unavoidable for the plug-in estimator over the full projected moment
class. Combining \eqref{eq:critical-minimax-lower} with
$
\mathcal R_{n,m}^{\star}\le \mathcal R_{n,m}^{\mathrm{plug}}
$
gives
\begin{equation*}
\label{eq:critical-minimax-gap}
b\frac{\left[\log(n+1)\right]^{1/4}
}{\sqrt n
}\lesssim_p \mathcal R_{n,m}^{\star}\left(
\mathcal K_{p,2p}(b)\right) \lesssim_pb\sqrt{\frac{\log(n+1)}{n}
}.
\end{equation*}
Therefore, the worst-case risk of the plug-in estimator is within a
factor of order $\left[\log(n+1)\right]^{1/4}$ of the minimax risk at the
critical moment order. Equivalently,
\begin{equation*}
\mathcal R_{n,m}^{\mathrm{plug}}
\left(\mathcal K_{p,2p}(b)
\right)\lesssim_p \left[\log(n+1)\right]^{1/4}\mathcal R_{n,m}^{\star}
\left(\mathcal K_{p,2p}(b)\right).
\end{equation*}
The exact minimax rate when $r=2p$ remains open.

Taken together, these results yield the following phase diagram. At the
minimal moment order $r=p$, the minimax risk satisfies
$\mathcal R_{n,m}^{\star}(\mathcal K_{p,p}(b))\gtrsim_p b$
for all sample sizes. Hence, it cannot converge to zero as $n,m\to\infty$,
and no estimator is uniformly consistent over $\mathcal K_{p,p}(b)$.
When $p<r<2p$, the minimax risk is of order
$b n^{-(r-p)/r}$, whereas the parametric order $b n^{-1/2}$ is attained
when $r>2p$. In both noncritical estimable regimes, the plug-in estimator
is minimax rate-optimal. At the critical moment order $r=2p$, the worst-case risk of the
plug-in estimator is of order $b\sqrt{\log(n+1)/n}$ and is within a
factor of order $\left[\log(n+1)\right]^{1/4}$ of the minimax risk.
The exact minimax rate remains open.

In the next section, we present the main arguments leading to the upper bounds for the plug-in risk
$ \mathcal R_{n,m}^{\mathrm{plug}}\left(\mathcal K_{p,r}(b)\right)
$
in the different moment regimes. Since the plug-in estimator is one
admissible estimation procedure, these bounds also yield upper bounds
for the corresponding minimax risk.

\section{Upper-bound analysis}\label{sec:proofsupperbound}
In this section, we present the main arguments underlying the upper
bounds in Theorem~\ref{thm:minimax-noncritical} and
Proposition~\ref{prop:r=2p}. We begin by reducing the global Sliced Wasserstein error to a family of one-dimensional projected
errors.

Fix $\mu,\nu\in\mathcal K_{p,r}(b)$. By definition,
$$
\int_{\mathbb S^{d-1}}M_{\mu,r}(\theta)^{p/r}\,d\sigma(\theta)\le b,
\qquad
\int_{\mathbb S^{d-1}}M_{\nu,r}(\theta)^{p/r}\,d\sigma(\theta)\le b.
$$
Therefore, by the definition of
$\mathcal K_{p,r}(b)$, if we define
\begin{equation}\label{eq:gooddirections}
\mathcal{G}\coloneqq \left\{ \theta\in\mathbb S^{d-1}:
M_{\mu,r}(\theta)<\infty \ \text{and}\ M_{\nu,r}(\theta)<\infty \right\},
\end{equation}
then $\sigma(\mathbb{S}^{d-1} \setminus \mathcal{G})=0$. As a result, we have
\begin{align*}
\mathbb{E}\left|
\operatorname{SW}_p^p(\widehat{\mu}_n,\widehat{\nu}_m)-\operatorname{SW}_p^p(\mu,\nu)
\right| &= \mathbb{E}
\left|\int_{\mathbb{S}^{d-1}} \left[ \operatorname{W}_p^p
\left( \Pi^\theta_{\#}\widehat{\mu}_n,
\Pi^\theta_{\#}\widehat{\nu}_m\right)
-\operatorname{W}_p^p\left(\Pi^\theta_{\#}\mu,\Pi^\theta_{\#}\nu\right)\right]
d\sigma(\theta)\right| \\
&\le\int_{\mathbb{S}^{d-1}}\mathbb{E}
\left|\operatorname{W}_p^p\left(\Pi^\theta_{\#}\widehat{\mu}_n,\Pi^\theta_{\#}\widehat{\nu}_m
\right)-\operatorname{W}_p^p
\left(\Pi^\theta_{\#}\mu,
\Pi^\theta_{\#}\nu\right)\right|
d\sigma(\theta)\\
& = \int_{\mathcal{G}} \mathbb{E}\left|\operatorname{W}_p^p
\left(\Pi^\theta_{\#}\widehat{\mu}_n,\Pi^\theta_{\#}\widehat{\nu}_m\right)-\operatorname{W}_p^p\left(\Pi^\theta_{\#}\mu,\Pi^\theta_{\#}\nu\right)\right|d\sigma(\theta).
\end{align*}
This reduces the global problem to deriving suitable directional bounds.
Throughout this section, for $\theta\in \mathcal{G}$, let $F_\theta$ and $G_\theta$ denote the distribution functions of $\Pi^\theta_{\#}\mu$ and
$\Pi^\theta_{\#}\nu$, respectively, and let $F_{n,\theta}$ and $G_{m,\theta}$ denote the corresponding empirical distribution functions.
We denote their generalized inverse functions by $F_\theta^{-1}$,
$G_\theta^{-1}$, $F_{n,\theta}^{-1}$, and $G_{m,\theta}^{-1}$.
We fix $\theta\in \mathcal{G}$ when deriving the directional bounds below, which
will be expressed in terms of $M_{\mu,r}(\theta)$ and $M_{\nu,r}(\theta)$. These bounds can then be integrated over $\theta$
and controlled using the defining constraints of $\mathcal K_{p,r}(b)$. We first establish the auxiliary estimates needed for this purpose.
\subsection{Technical lemmas}\label{sec:technicallemmas}
The auxiliary results in this subsection serve two purposes. First,
we record a projected empirical measure estimation bound under the powered Wasserstein loss, which
already yields the desired upper bounds when $p=1$. Second, when
$p>1$, we decompose the projected two-sample error into unweighted
empirical terms, which are controlled by the empirical measure estimation bound, and
weighted quantile terms, which require a separate analysis.

We begin with an empirical measure estimation bound that is essentially the
one-dimensional specialization of
\citet[Theorem~1]{fournier2015rate}. Since the critical case
$s=2q$ is not stated explicitly there, we treat it separately.
A complete proof is deferred to
Appendix~\ref{appendix:proofofonesamplebounds}.
\begin{lemma}\label{lem:projected-empirical-rate}
Let $q\ge 1$ and $s>q$. Let $Q$ be a probability measure on $\mathbb{R}^d$ and $\widehat{Q}_n$ be the empirical measure associated with $n$ i.i.d. samples from $Q$. Fix $\theta\in \mathbb{S}^{d-1}$ such that
$M_{Q,s}(\theta)<\infty$. Then,
$$
\mathbb E\operatorname{W}_q^q\left( \Pi^\theta_{\#}\widehat{Q}_n,\Pi^\theta_{\#}Q\right)\lesssim_{q,s}\left(M_{Q,s}(\theta)\right)^{q/s}
\begin{cases}
n^{-(s-q)/s}, & q<s<2q,\\
\log (n+1)n^{-1/2}, &s=2q,\\
n^{-1/2}, & s>2q.
\end{cases}
$$
\end{lemma}

Lemma~\ref{lem:projected-empirical-rate} is already sufficient to prove the upper bound of Theorem~\ref{thm:minimax-noncritical} when $p=1$. Indeed, by the reverse triangle inequality for $\operatorname{W}_1$,
$$
\left|\operatorname W_1\left(\Pi^\theta_{\#}\widehat\mu_n,\Pi^\theta_{\#}\widehat\nu_m\right)-\operatorname W_1\left(\Pi^\theta_{\#}\mu,\Pi^\theta_{\#}\nu\right)\right|\le\operatorname W_1\left(\Pi^\theta_{\#}\widehat\mu_n,\Pi^\theta_{\#}\mu\right)+\operatorname W_1\left(\Pi^\theta_{\#}\widehat\nu_m,\Pi^\theta_{\#}\nu
\right).$$
Therefore, applying Lemma~\ref{lem:projected-empirical-rate} with $q=1$ and $s=r$ to each sample, integrating over $\theta$, using the defining
constraints of $\mathcal K_{1,r}(b)$, and recalling that $n\le m$ yield the desired bounds for $p=1$. For $p>1$, however, the reverse triangle inequality for $\operatorname{W}_p$ does not directly control the difference of the powered Wasserstein costs at the desired rate. We therefore use the following decomposition.
\begin{lemma}\label{lem:decomposition}
Let $p>1$. For the fixed $\mu,\nu\in\mathcal K_{p,r}(b)$ and
$\theta\in\mathcal G$, where $\mathcal K_{p,r}(b)$ and $\mathcal G$ are defined in Eqs.~\eqref{sec1:eqclass} and
\eqref{eq:gooddirections}, respectively, we have
% Let $p>1$. Let $\mu,\nu \in \mathcal{K}_{p,r}(b)$, where $\mathcal{K}_{p,r}(b)$ is
% introduced in Eq.~\eqref{sec1:eqclass}. Let $\widehat{\mu}_n$
% and $\widehat{\nu}_m$ be the empirical measures associated with $n$ and $m$
% i.i.d. samples from $\mu$ and $\nu$, respectively.

% Fix $\theta\in G$, where $G$ is introduced in Eq.~\eqref{eq:gooddirections}. Denote by $F_\theta$ and $G_\theta$ the
% distribution functions of $\Pi^\theta_{\#}\mu$ and $\Pi^\theta_{\#}\nu$,
% respectively. Similarly, denote by $F_{n,\theta}$ and $G_{m,\theta}$ the
% empirical distribution functions of $\Pi^\theta_{\#}\widehat{\mu}_n$ and
% $\Pi^\theta_{\#}\widehat{\nu}_m$, respectively. We denote their generalized
% inverse functions by $F_\theta^{-1}$, $G_\theta^{-1}$,
% $F_{n,\theta}^{-1}$ and $G_{m,\theta}^{-1}$. Then
% $$
$$
\begin{aligned}
\mathbb{E}\left|\operatorname{W}_p^p\left(\Pi^\theta_{\#}\widehat{\mu}_n,\Pi^\theta_{\#}\widehat{\nu}_m\right)-\operatorname{W}_p^p\left(\Pi^\theta_{\#}\mu,\Pi^\theta_{\#}\nu\right)\right| \lesssim_p\left(\mathsf{A}_{n,\theta}+\mathsf{B}_{m,\theta}+\mathsf{C}_{n,\theta}+\mathsf{D}_{m,\theta}\right),
\end{aligned}
$$
where
$$\mathsf{A}_{n,\theta} \coloneqq \mathbb{E}\int_0^1\left(|F_\theta^{-1}(u)|^{p-1}+
|G_\theta^{-1}(u)|^{p-1}\right)\left|F_{n,\theta}^{-1}(u)-F_\theta^{-1}(u)\right|du,
$$
$$\mathsf{B}_{m,\theta}\coloneqq \mathbb{E}\int_0^1\left(|F_\theta^{-1}(u)|^{p-1}+|G_\theta^{-1}(u)|^{p-1}\right)\left|G_{m,\theta}^{-1}(u)-G_\theta^{-1}(u)\right|du,
$$
and
$$\mathsf{C}_{n,\theta}\coloneqq \mathbb{E}\int_0^1\left|F_{n,\theta}^{-1}(u)-F_\theta^{-1}(u)\right|^pdu,\qquad \mathsf{D}_{m,\theta}\coloneqq \mathbb{E}\int_0^1\left|G_{m,\theta}^{-1}(u)-G_\theta^{-1}(u)\right|^pdu.
$$
\end{lemma}
We defer the proof of this lemma to Appendix~\ref{appendix:proofoflemdecomposition}.

The decomposition in Lemma~\ref{lem:decomposition} is inspired by the first reduction in \citet[Appendix~D, Step~1]{manole2022minimax}. This decomposition is useful because the terms $\mathsf{C}_{n,\theta}$ and $\mathsf{D}_{m,\theta}$ can be controlled directly by Lemma~\ref{lem:projected-empirical-rate}. However, controlling the weighted quantile terms $\mathsf{A}_{n,\theta}$ and $\mathsf{B}_{m,\theta}$ is substantially more challenging and constitutes the main technical part of the proof.

To control these terms, we decompose $(0,1)$ into layers according to
the magnitude of the quantile weight. On each layer, the weight is
bounded by the corresponding level, reducing the problem to a
localized estimate for the unweighted quantile error near the
endpoints of $(0,1)$. The next proposition provides this estimate; its complete proof is given in Appendix~\ref{appendix:proofofprop:peelingtool}.
\begin{proposition}[Localized empirical quantile estimate]\label{prop:peelingtool}
Let $r>2$. For the fixed $\mu\in\mathcal K_{p,r}(b)$ and
$\theta\in\mathcal G$, we have, for every $0\le \alpha\le1$,
% Let $r>2$, and let $\mu\in\mathcal K_{p,r}(b)$, where $\mathcal K_{p,r}(b)$ is defined in Eq.~\eqref{sec1:eqclass}. Let $\widehat\mu_n$ be the empirical measure associated with $n$ i.i.d. samples from $\mu$.
% Fix $\theta\in G$, where $G$ is defined in Eq.~\eqref{eq:gooddirections}. Denote by $F_\theta$ the distribution function of $\Pi^\theta_{\#}\mu$, and by $F_{n,\theta}$ the empirical distribution function of $\Pi^\theta_{\#}\widehat\mu_n$. Let
% $M_{\mu,r}(\theta)\coloneqq \mathbb E_{X\sim\mu}\left[|\langle X,\theta\rangle|^r\right].
% $
% Then for every $0\le \alpha\le 1$,
$$
\mathbb E\int_0^\alpha\left|F_{n,\theta}^{-1}(u)-F_\theta^{-1}(u)\right|du \lesssim_r \frac{M_{\mu,r}(\theta)^{1/r}}{\sqrt n} \alpha^{1/2-1/r}.
$$
\end{proposition}
% The proof is deferred to Appendix~\ref{appendix:proofofprop:peelingtool}
\begin{proof}[Proof sketch.]
The case $\alpha=0$ is trivial. We may therefore assume that
$0<\alpha\le1$. Adapting the arguments used in \citet[Proposition~2.17]{santambrogio2015optimal} to the interval
$(0,\alpha)$ yields
\begin{equation*}
\mathbb E \int_0^\alpha \left|
F_{n,\theta}^{-1}(u)-F_\theta^{-1}(u)
\right|du = \int_{\mathbb R}
\mathbb E\left| \min\{F_{n,\theta}(x),\alpha\} - \min\{F_\theta(x),\alpha\}
\right|dx,
\end{equation*}

When $\alpha\ge1/4$, the desired estimate follows directly from the
global empirical measure estimation bound in $\operatorname{W}_1$. We may therefore assume
that $0<\alpha<1/4$ and decompose $\mathbb R$ into
\begin{equation*}
J_1\coloneqq  \left\{ x:F_\theta(x)\le2\alpha
\right\}, \qquad J_2\coloneqq \left\{
x:2\alpha<F_\theta(x)\le3/4 \right\},
\qquad J_3 \coloneqq \left\{x:F_\theta(x)>3/4\right\}.
\end{equation*}

We first derive pointwise estimates on these three regions. 

On $J_1$,
the map $t\mapsto\min\{t,\alpha\}$ is $1$-Lipschitz, and the binomial
variance bound yields
\begin{equation*}
\mathbb E \left|
\min\{F_{n,\theta}(x),\alpha\} -
\min\{F_\theta(x),\alpha\} \right|\le \mathbb{E}\left|F_{n,\theta}(x)-F_{\theta}(x)\right| \le
\sqrt{\frac{F_\theta(x)(1-F_{\theta}(x))}{n}}\le \sqrt{\frac{F_{\theta}(x)}{n}}.
\end{equation*}

On $J_2\cup J_3$, since $F_\theta(x)>\alpha$, we have
$\min\{F_\theta(x),\alpha\}=\alpha$, and hence
\begin{equation*}
\mathbb E \left| \min\{F_{n,\theta}(x),\alpha\}- \min\{F_\theta(x),\alpha\}
\right|= \mathbb E \left[
\left(\alpha-F_{n,\theta}(x)\right)_+
\right]\le \alpha\mathbb P\left(
F_{n,\theta}(x)\le\alpha \right).
\end{equation*}
On $J_2$, the inequality $F_\theta(x)>2\alpha$ and a Chernoff bound
give
\begin{equation*}
\mathbb P\left( F_{n,\theta}(x)\le\alpha\right) \le \mathbb{P}\left(F_{n,\theta}(x)\le\frac{F_\theta(x)}{2}\right)\le\exp\left(
-\frac{nF_\theta(x)}{8}\right)\le
\exp\left(-\frac{n\alpha}{4}\right).
\end{equation*}
On $J_3$, since $F_{\theta}(x)>\frac{3}{4}$ and $\alpha <1/4$, $F_\theta(x)-\alpha>1/2$. Then, Chebyshev's inequality
yields
\begin{equation*}
\mathbb P \left(F_{n,\theta}(x)\le\alpha\right) \le \mathbb{P}\left(\left|F_{n,\theta}(x)-F_{\theta}(x)\right| \ge F_{\theta}(x)-\alpha\right) \le \mathbb{P}\left(\left|F_{n,\theta}(x)-F_{\theta}(x)\right|\ge \frac{1}{2}\right)\le
\frac{4\left(1-F_\theta(x)\right)}{n}.
\end{equation*}

To integrate these pointwise estimates, we use the following
consequence of Markov's inequality: for every $0<a<c<1$,
\begin{equation*}
\left| \left\{x\in\mathbb R:
a<F_\theta(x)\le c\right\}\right|\le
2M_{\mu,r}(\theta)^{1/r}\max\left\{
a^{-1/r},(1-c)^{-1/r}\right\}.
\end{equation*}
This estimate applies directly to $J_2$. On $J_1$, we further group
points according to the magnitude of $F_\theta(x)$, while on $J_3$
we group them according to the magnitude of $1-F_\theta(x)$. On each
resulting region, the corresponding pointwise estimate is uniform and
the preceding level-set bound controls its length. Since $r>2$, the
resulting geometric series are summable. Combining the contributions
of $J_1$, $J_2$, and $J_3$ yields the desired estimate.
\end{proof}
\begin{remark}\label{rem:peelingtool}
The same estimate holds for the second sample. More precisely, for every $0\le \alpha\le 1$,
% if
% $\nu\in\mathcal K_{p,r}(b)$ and $\widehat\nu_m$ is the empirical measure associated with $m$ i.i.d. samples from $\nu$, and if $G_\theta$ and $G_{m,\theta}$ denote the distribution functions of $\Pi^\theta_{\#}\nu$ and $\Pi^\theta_{\#}\widehat\nu_m$, respectively, then 
$$ \mathbb E \int_0^\alpha
\left| G_{m,\theta}^{-1}(u)-G_\theta^{-1}(u) \right|
du\lesssim_r\frac{M_{\nu,r}(\theta)^{1/r}}{\sqrt m}\alpha^{1/2-1/r}.
$$
Moreover, by applying the same argument to the upper tail, we obtain,
for every $0\le \beta\le1$,
$$ \mathbb E \int_{\beta}^{1}
\left| F_{n,\theta}^{-1}(u)-F_\theta^{-1}(u) \right|
du \lesssim_r \frac{M_{\mu,r}(\theta)^{1/r}}{\sqrt n} (1-\beta)^{1/2-1/r},
$$
and similarly
$$
\mathbb E \int_\beta^1 \left| G_{m,\theta}^{-1}(u)- G_\theta^{-1}(u) \right|du
\lesssim_r \frac{M_{\nu,r}(\theta)^{1/r}}{\sqrt m} (1-\beta)^{1/2-1/r}.
$$
\end{remark}
Proposition~\ref{prop:peelingtool} and
Remark~\ref{rem:peelingtool} provide the localized estimates needed
to control the weighted quantile terms $\mathsf{A}_{n,\theta}$ and
$\mathsf{B}_{m,\theta}$. In the next section, we first consider the regime $r>2p$.
\subsection{The regime $r>2p$}\label{sec:upperboundr>2p}
In this case, the localized
estimate in Proposition~\ref{prop:peelingtool}, combined with a peeling
argument according to the magnitude of the quantile weights, yields
the following bound.
\begin{lemma}[Weighted quantile estimate]\label{lem:weighted-quantile}
Let $p>1$ and $r>2p$. For the fixed $\mu,\nu\in\mathcal K_{p,r}(b)$ and $\theta\in\mathcal G$, where $\mathcal K_{p,r}(b)$ and $\mathcal G$ are defined in Eqs.~\eqref{sec1:eqclass} and
\eqref{eq:gooddirections}, respectively, we have
% Let $p>1$ and $r>2p$. Let $\mu,\nu\in\mathcal K_{p,r}(b)$, where $\mathcal K_{p,r}(b)$ is defined in Eq.~\eqref{sec1:eqclass}. Fix $\theta\in G$, where $G$ is defined in Eq.~\eqref{eq:gooddirections}. Let $F_\theta$ and $G_\theta$ be the distribution functions of $\Pi^\theta_{\#}\mu$ and $\Pi^\theta_{\#}\nu$, respectively, and let $F_{n,\theta}$ be the empirical distribution function associated with $n$ i.i.d. samples from $\Pi^\theta_{\#}\mu$. Then 
$$\mathbb E
\int_0^1 \left|
F_{n,\theta}^{-1}(u)-F_\theta^{-1}(u)
\right| \left| G_\theta^{-1}(u)
\right|^{p-1} du\lesssim_{p,r}
\frac{ M_{\mu,r}(\theta)^{1/r}
M_{\nu,r}(\theta)^{(p-1)/r} }{\sqrt n}.
$$
% where
% $
% M_{\mu,r}(\theta)
% \coloneqq
% \mathbb E_{X\sim\mu}\left[|\langle X,\theta\rangle|^r\right]$ and $M_{\nu,r}(\theta)
% \coloneqq
% \mathbb E_{Y\sim\nu}\left[|\langle Y,\theta\rangle|^r\right].
% $
The same bound holds with $G_\theta^{-1}$ replaced by
$F_\theta^{-1}$ and with the right-hand side replaced by
$M_{\mu,r}(\theta)^{p/r}/\sqrt n$.
The analogous estimates obtained by exchanging
$(\mu,n,F_\theta,F_{n,\theta})$ and
$(\nu,m,G_\theta,G_{m,\theta})$ also hold.
\end{lemma}
\begin{proof}
If $M_{\nu,r}(\theta)=0$, then $G_\theta^{-1}(u)=0$ for every $u\in(0,1)$. Since $p>1$, the weighted integral is equal to zero, and the result follows immediately.

We now assume that $M_{\nu,r}(\theta)\in(0,+\infty)$. For simplicity, we set
$
K\coloneqq M_{\nu,r}(\theta)^{1/r}>0.\\
$
We decompose $(0,1)$ according to the size of $|G_\theta^{-1}(u)|$. Define
$$E^* \coloneqq \left\{ u\in(0,1): |G_\theta^{-1}(u)|<K
\right\},
$$
and, for $k\ge0$,
$$ E_k \coloneqq \left\{ u\in(0,1): 2^kK\le |G_\theta^{-1}(u)|<2^{k+1}K\right\}.
$$

On $E^*$, the weight is bounded by $K^{p-1}$. Hence, by Lemma~\ref{lem:projected-empirical-rate} applied (with $Q=\mu$, $q=1, s=r$),
\begin{align*}
\mathbb E\int_{E^*}\left|F_{n,\theta}^{-1}(u)-F_\theta^{-1}(u)\right|\left|G_\theta^{-1}(u)\right|^{p-1}du
&\le K^{p-1} \mathbb E \int_0^1
\left|F_{n,\theta}^{-1}(u)-F_\theta^{-1}(u)
\right|du \\
&\le C_r \frac{ M_{\mu,r}(\theta)^{1/r} K^{p-1}}{\sqrt n}=C_r \frac{ M_{\mu,r}(\theta)^{1/r} M_{\nu,r}(\theta)^{(p-1)/r}}{\sqrt n}.
\end{align*}

We now consider the sets $E_k$. Define
$a_k\coloneqq G_\theta(-2^kK), b_k\coloneqq G_\theta(2^kK-),
$ where
$
G_\theta(2^kK-) \coloneqq \lim_{t\to  \left(2^kK\right)^-}G_\theta(t).$ The condition $|G_\theta^{-1}(u)|\ge 2^kK$ implies
$
u\in(0,a_k]\cup[b_k,1).
$
Therefore, using Proposition~\ref{prop:peelingtool} and its upper-tail analogue, we obtain
$$
\begin{aligned}
&\mathbb E \int_{E_k} \left| F_{n,\theta}^{-1}(u)-F_\theta^{-1}(u) \right| \left| G_\theta^{-1}(u)
\right|^{p-1}du \\
&\le (2^{k+1}K)^{p-1} \mathbb E
\left[\int_0^{a_k} \left| F_{n,\theta}^{-1}(u)-F_\theta^{-1}(u) \right|du
+\int_{b_k}^1\left|F_{n,\theta}^{-1}(u)-F_\theta^{-1}(u)
\right|du \right] \\
&\le C_r2^{(k+1)(p-1)}K^{p-1}\frac{M_{\mu,r}(\theta)^{1/r}}{\sqrt n}\left(a_k^{1/2-1/r}+(1-b_k)^{1/2-1/r}
\right).
\end{aligned}
$$

We now bound $a_k$ and $1-b_k$. By Markov's inequality,
$$a_k =G_\theta(-2^kK)=\mathbb P_{Y\sim\nu}
\left(\langle Y,\theta\rangle\le -2^kK
\right)\le\mathbb P_{Y\sim\nu}
\left(|\langle Y,\theta\rangle|\ge 2^kK
\right)\le2^{-kr}.
$$
Similarly, since $b_k=G_\theta(2^kK-)$, we have
$$1-b_k=\mathbb P_{Y\sim\nu}\left(\langle Y,\theta\rangle\ge 2^kK\right)\le\mathbb P_{Y\sim\nu}\left(|\langle Y,\theta\rangle|\ge 2^kK\right)\le2^{-kr}.
$$
Since $r>2p>2$, we have $1/2-1/r>0$. Consequently,
$$a_k^{1/2-1/r}+(1-b_k)^{1/2-1/r}\le2^{-k(r/2-1)+1}.
$$
Hence,
\begin{equation}\label{eq:peelingcontribution}
\mathbb E\int_{E_k}\left|F_{n,\theta}^{-1}(u)-F_\theta^{-1}(u)\right|\left|G_\theta^{-1}(u)\right|^{p-1}du\lesssim_{p,r}\frac{M_{\mu,r}(\theta)^{1/r}M_{\nu,r}(\theta)^{(p-1)/r}}{\sqrt n}2^{k(p-r/2)}.
\end{equation}

Since $r>2p$, we have $p-r/2<0$. Therefore,
$
\sum_{k\ge0}2^{k(p-r/2)}<+\infty.\\
$
Summing over $k\ge0$ and adding the contribution of $E^*$, we obtain
$$
\mathbb E\int_0^1\left|F_{n,\theta}^{-1}(u)-F_\theta^{-1}(u)\right|\left|
G_\theta^{-1}(u)\right|^{p-1}du \lesssim_{p,r}\frac{M_{\mu,r}(\theta)^{1/r}M_{\nu,r}(\theta)^{(p-1)/r}}{\sqrt n}.
$$
The bound with $F_\theta^{-1}$ as the weight follows from the same
argument with $\nu$ replaced by $\mu$. The corresponding estimates for
the second sample follow by exchanging
$(\mu,n,F_\theta,F_{n,\theta})$ and
$(\nu,m,G_\theta,G_{m,\theta})$.
\end{proof}
Combining Lemmas~\ref{lem:projected-empirical-rate},
\ref{lem:decomposition}, and \ref{lem:weighted-quantile} yields the
upper bound in Theorem~\ref{thm:minimax-noncritical} for the regime
$r>2p$. 
\begin{proof}[Proof sketch for Theorem~\ref{thm:minimax-noncritical}
when $r>2p$.]
For $p=1$, the result follows from
Lemma~\ref{lem:projected-empirical-rate} and the reverse triangle
inequality for $\operatorname W_1$. For $p>1$,
Lemma~\ref{lem:decomposition} reduces the directional error to
$\mathsf{A}_{n,\theta}$, $\mathsf{B}_{m,\theta}$, $\mathsf{C}_{n,\theta}$, and
$\mathsf{D}_{m,\theta}$. Lemma~\ref{lem:projected-empirical-rate} bounds the
unweighted terms $\mathsf{C}_{n,\theta}$ and $\mathsf{D}_{m,\theta}$ by $n^{-1/2}M_{\mu,r}(\theta)^{p/r}$ and $m^{-1/2}M_{\nu,r}(\theta)^{p/r}$ respectively, while Lemma~\ref{lem:weighted-quantile} gives the same
sample-size rates for the weighted terms. Integrating the resulting directional bounds over $\theta$
and applying Hölder's inequality to terms of the form $M_{\mu,r}(\theta)^{1/r} M_{\nu,r}(\theta)^{(p-1)/r}$ yields $\mathcal R_{n,m}^{\mathrm{plug}} \left(\mathcal K_{p,r}(b)\right)
\lesssim_{p,r}b\,n^{-1/2}.$
The complete proof is given in
Appendix~\ref{appendix:upperboundr>2p}.
\end{proof}
\subsection{The regime $p<r<2p$}\label{sec:upperboundp<r<2p}

We now turn to the regime $p<r<2p$. The proof differs from the preceding regime in two respects. First,
the contributions of the successive layers no longer form a summable
series, so the peeling decomposition must be truncated. Second,
Proposition~\ref{prop:peelingtool} is only available when $r>2$,
whereas the present regime also includes $1<r\le2$. The following
lemma resolves both difficulties.
    
\begin{lemma}[Truncated weighted quantile estimate]
\label{lem:truncated-weighted-quantile}
Let $p>1$ and $p<r<2p$. For the fixed $\mu,\nu\in\mathcal K_{p,r}(b)$ and any $\theta\in\mathcal G$, we have
\begin{equation*}
\mathbb E\int_0^1
\left|F_{n,\theta}^{-1}(u)-F_\theta^{-1}(u)\right|\left|G_\theta^{-1}(u)
\right|^{p-1}du\lesssim_{p,r}M_{\mu,r}(\theta)^{1/r}M_{\nu,r}(\theta)^{(p-1)/r}
n^{-(r-p)/r}.
\end{equation*}
The same bound holds with $G_\theta^{-1}$ replaced by
$F_\theta^{-1}$ and with the corresponding right-hand side $M_{\mu,r}(\theta)^{p/r}n^{-(r-p)/r}$.
Also, the analogous estimates obtained by exchanging
$(\mu,n,F_\theta,F_{n,\theta})$ and
$(\nu,m,G_\theta,G_{m,\theta})$ hold.
\end{lemma}

\begin{proof}[Proof sketch.]
Set $K=M_{\nu,r}(\theta)^{1/r}$. For an integer $J_n\ge1$, we
decompose $(0,1)$ into the region where
$|G_\theta^{-1}|<K$, the finitely many layers where $2^kK\le |G_\theta^{-1}|<2^{k+1}K$  for $0\le k\le J_n-1$,
and the far-tail region where
$|G_\theta^{-1}|\ge2^{J_n}K$.

On the first region, the weight is bounded, and the contribution is
controlled by the empirical measure estimation bound in $\operatorname W_1$ obtained by applying Lemma~\ref{lem:projected-empirical-rate}. On the intermediate layers, we use Proposition~\ref{prop:peelingtool} when $r>2$ and the global empirical measure estimation bound in $\operatorname W_1$ when $1<r<2$. At $r=2$, a direct use of the critical $\operatorname W_1$ bound on every layer would introduce an additional logarithmic factor. We avoid this loss by applying Hölder's inequality and an empirical measure estimation bound in $\operatorname W_{q_0}$ for a suitably chosen $1<q_0<2/(3-p).$

The far-tail contribution is controlled directly by Markov's and
Hölder's inequalities together with the projected $r$-th moment
bounds. Choosing $J_n$ so as to balance the contributions of the
retained layers and the far-tail remainder yields the rate $n^{-(r-p)/r}$. The complete proof is given in Appendix~\ref{appendix:prooflem:truncated-weighted-quantile}.
\end{proof}
Once Lemma~\ref{lem:truncated-weighted-quantile} is established, the
upper bound in Theorem~\ref{thm:minimax-noncritical} follows by the
same aggregation argument as in the regime $r>2p$. We therefore omit
the repetition.

\subsection{The critical regime $r=2p$}\label{sec:upperboundr=2p}
In this subsection, we focus on the upper bound in
Eq.~\eqref{eq:critical-plugin-risk}. Unlike the two preceding regimes,
Lemma~\ref{lem:projected-empirical-rate} is insufficient to obtain the
desired bound, even when $p=1$. We therefore first establish a sharper empirical measure estimation bound under the powered Wasserstein loss at the critical moment order.
\begin{lemma}\label{lem:projected-empirical-ratev2}
Let $q\ge1$. Let $Q$ be a probability measure on $\mathbb R^d$ and
$\widehat Q_n$ be the empirical measure associated with $n$ i.i.d.\
samples from $Q$. Fix $\theta\in\mathbb S^{d-1}$ such that
$M_{Q,2q}(\theta)<\infty$. Then
$$\mathbb E\operatorname{W}_q^q\left( \Pi^\theta_{\#}\widehat{Q}_n,\Pi^\theta_{\#}Q\right)\lesssim_q \left(M_{Q,2q}(\theta)\right)^{1/2}\sqrt{\log (n+1)n^{-1}}.
$$
\end{lemma}
The proof, given in
Appendix~\ref{appendix:proofoflem:projected-empirical-ratev2},
refines the argument of \citet[Corollary~7.18]{bobkov2019one}. By replacing the pointwise Chebyshev control of the critical tail with a Cauchy--Schwarz argument based on the exact tail-moment identity, it improves the rate from
$\log(n+1)/\sqrt n$ to $\sqrt{\log(n+1)/n}$.

When $p=1$, Lemma~\ref{lem:projected-empirical-ratev2} together with the reverse triangle inequality for
$\operatorname W_1$ already yields the upper bound in
Eq.~\eqref{eq:critical-plugin-risk}. As in the preceding regimes, when $p>1$ we must also control the
weighted terms $A_{n,\theta}$ and $B_{m,\theta}$ in
Lemma~\ref{lem:decomposition}. However, neither of the two preceding
peeling arguments attains the target rate at the critical order
$r=2p$. In the infinite peeling argument, the layer factor
$2^{k(p-r/2)}$ becomes constant and is therefore not summable. In the
truncated argument, Eq.~\eqref{eq:peelingcontribution} shows  that the
first $J_n$ layers contribute at order $J_n/\sqrt n$, while the
far-tail remainder (see Appendix~\ref{appendix:prooflem:truncated-weighted-quantile}, Step~4) contains the term
$n^{1/(2p)}2^{-(p+1)J_n}$. Unfortunately, no choice of $J_n$ makes both terms of order $\sqrt{\log(n+1)/n}$.

To overcome this difficulty, rather than applying Markov's inequality
separately to the tail mass of each layer, as in the proof of
Lemma~\ref{lem:weighted-quantile}, we retain the actual tail masses
and control them collectively after summation. This leads to the
following weighted quantile estimate.
\begin{lemma}\label{lem:endpointweightedquantile}
Let $p>1$. For the fixed
$\mu,\nu\in\mathcal K_{p,2p}(b)$ and any $\theta\in\mathcal G$, we have
% Let $p>1$. Let $\mu,\nu\in\mathcal P(\mathbb R^d)$, and fix $\theta\in \mathbb{S}^{d-1}$ such that
% $
% M_{\mu,2p}(\theta)+M_{\nu,2p}(\theta)<\infty.\\
% $
% Let $F_\theta$ and $G_\theta$ be the distribution functions of $\Pi^\theta_\#\mu$ and $\Pi^\theta_\#\nu$, respectively. Let $F_{n,\theta}$ and $G_{m,\theta}$ be the empirical distribution functions associated with $n$ i.i.d. samples from $\Pi^\theta_\#\mu$ and $m$ i.i.d. samples from $\Pi^\theta_\#\nu$, respectively. 
$$
\mathbb E\int_0^1 \left|F_{n,\theta}^{-1}(u)-F_\theta^{-1}(u)\right|
\left|G_\theta^{-1}(u)\right|^{p-1}du
\lesssim_p \left(
M_{\mu,2p}(\theta)+M_{\nu,2p}(\theta)
\right)^{1/2}\sqrt{\log (n+1)n^{-1}}.
$$
The same bound holds with $G_\theta^{-1}$ replaced by
$F_\theta^{-1}$. The analogous estimates obtained by exchanging
$(\mu,n,F_\theta,F_{n,\theta})$ and
$(\nu,m,G_\theta,G_{m,\theta})$ also hold.
\end{lemma}
\begin{proof}[Proof sketch.]
Set $M_\theta
\coloneqq
M_{\mu,2p}(\theta)+M_{\nu,2p}(\theta)$. The case $M_\theta=0$ is immediate, so we assume $M_\theta>0$. 

The proof starts from the following observation. In the two preceding regimes, the peeling argument is based on the size of $\left|G^{-1}_{\theta}\right|$. More precisely, for $t>0$, Markov's inequality gives an upper bound for the Lebesgue measure of the set $\left\{u\in (0,1):\left|G^{-1}_{\theta}(u)\right|\ge t\right\}$ in terms of $M_{\nu,r}(\theta)$ and $t$. Thus, the measure of each region can be controlled using only the information of $G^{-1}_{\theta}$. At the critical order $r=2p$, however, replacing these measures by their Markov bounds is not sufficient. We therefore need to keep their actual values throughout the argument. The difficulty is that, when controlling the empirical quantile error, we also need to control the regions where $\left|F^{-1}_{\theta}\right|$ is large. If the decomposition is defined only through $G^{-1}_{\theta}$, the resulting layer masses do not capture the regions where $F^{-1}_{\theta}$ is large. To resolve this problem, for $u\in (0,1)$, we introduce $H_{\theta}(u) \coloneqq \max \left\{\left|F^{-1}_{\theta}(u)\right|, \left|G^{-1}(u)\right|\right\}.$
Then, for every $t>0$,
\begin{equation*}
\{H_{\theta}\ge t\} = \left\{\left|F^{-1}_{\theta}\right|\ge t\right\} \cup \left\{\left|G^{-1}_{\theta}\right|\ge t\right\},
\end{equation*}
so the same regions can be used to control both quantiles.

% The proof begins with the following observation: In the two preceding regimes, the peeling argument is based on the size of $G^{-1}_{\theta}$. More precisely, on a layer where $\left|G^{-1}_{\theta}(u) \ge t\right|$, Markov's inequality provides an upper bound on the measure of this regions in terms of $M_{\nu,r}(\theta)$ and $t$, meaning it does not need the information from $F^{-1}_{\theta}$. This is sufficient in the previous arguments. At the critical order $r=2p$, to retain the actual measure of each regin, we need to also control where $\left|F^{-1}_{\theta}\right|$ is large. This suggest using a single decomposition that keeps track of both quantiles. To do that, we define $H_{\theta}(u)\coloneqq \max \left\{\left|F^{-1}_{\theta}(u)\right|,\left|G^{-1}_{\theta}(u)\right|\right\}$ for $u \in (0,1)$. Then, for every $t>0$, we have
% \begin{equation*}
% \{H_{\theta} \ge t\} =\left\{ \left|F_{\theta}^{-1}\ge t\right|\right\} \cup \left\{ \left|G_{\theta}^{-1}\ge t\right|\right\},
% \end{equation*}
% so the same regions can be used throughout the argument. Moreover, 
% \begin{equation*}
% \int_{0}^{1}H_{\theta}(u)^{2p}du \le M_{\mu,2p}(\theta) +M_{\nu,2p}(\theta) =M_{\theta},
% \end{equation*}
% so introducing $H_{\theta}$ does not require any additional moment control.

We now decompose $(0,1)$ into the central region
$
E^*=\{H_\theta<M_\theta^{1/(2p)}\}
$
and the layers
\begin{equation*}
E_k= \left\{
2^kM_\theta^{1/(2p)} \le H_\theta
<2^{k+1}M_\theta^{1/(2p)}\right\},
\qquad k\ge0.
\end{equation*}
On $E^*$, the weight is bounded by $M_\theta^{(p-1)/(2p)}$, so the global empirical measure estimation bound in $\operatorname W_1$ in Lemma~\ref{lem:projected-empirical-rate} gives a contribution of order
$M_\theta^{1/2}n^{-1/2}$.

Let $\lambda$ denote the Lebesgue measure on $(0,1)$. For $k\ge0$, define
\begin{equation*}
L_k\coloneqq \left\{ u\in(0,1):
H_\theta(u)\ge 2^kM_\theta^{1/(2p)}
\right\}, \qquad S_k\coloneqq\lambda(L_k),
\qquad m_k\coloneqq \min\left\{S_k,\sqrt{\frac{S_k}{n}}
\right\}.
\end{equation*}
The key point is to retain the actual tail masses $S_k$ rather than
replace them immediately by their Markov bounds as in the proof of Lemma~\ref{lem:weighted-quantile}. An auxiliary
localized estimate, proved in the appendix, gives
\begin{equation*}
\mathbb E \int_{L_k} \left|
F_{n,\theta}^{-1}(u)-F_\theta^{-1}(u)
\right|du\lesssim_p M_\theta^{1/(2p)}\sum_{j\ge k}2^jm_j.
\end{equation*}
Since the weight on $E_k$ is bounded by a constant multiple of
$2^{k(p-1)}M_\theta^{(p-1)/(2p)}$, summing over the layers and
exchanging the order of summation yield
\begin{align*}
\sum_{k\ge0}2^{k(p-1)} \sum_{j\ge k}2^jm_j &= \sum_{j\ge0}2^jm_j\sum_{k=0}^{j}2^{k(p-1)}\lesssim_p \sum_{j\ge0}2^{jp}m_j.
\end{align*}
Thus, it remains to control
$M_\theta^{1/2}\sum_{j\ge0}2^{jp}m_j.$ By the definition of $S_j$ and Tonelli's theorem, we have
\begin{align*}
\sum_{j\ge0}2^{2pj}S_j =\int_{0}^{1}\sum_{j\ge 0}2^{2pj}\mathbbm{1}_{\left\{H_{\theta}(u)\ge 2^jM_{\theta}^{1/(2p)}\right\}}du=\int_{0}^{1}\sum_{\substack{j\ge 0\\
2^j \le \frac{H_{\theta}(u)}{M_{\theta}^{1/(2p)}}}}2^{2pj}du&\lesssim_p
\frac{1}{M_\theta}\int_0^1 H_\theta(u)^{2p}du.
\end{align*}
The definition of $H_\theta$ and the quantile-moment identity then give
\begin{equation*}
\int_0^1 H_\theta(u)^{2p}du \le M_{\mu,2p}(\theta)+M_{\nu,2p}(\theta)=M_\theta.
\end{equation*}
Hence,
\begin{equation*}
\sum_{j\ge0}2^{2pj}S_j\lesssim_p1.
\end{equation*}

% The projected $2p$-th moment condition yields 
% \begin{align*}
% \sum_{j\ge0}2^{2pj}S_j
% &= \int_0^1 \sum_{j\ge0}2^{2pj}
% \mathbf 1_{\{H_\theta(u)\ge2^jM_\theta^{1/(2p)}\}}
% du\\
% &\lesssim_p\frac{1}{M_\theta}\int_0^1H_\theta(u)^{2p}du\lesssim_p\frac{
% M_{\mu,2p}(\theta)+M_{\nu,2p}(\theta)
% }{M_\theta}\lesssim_p1.
% \end{align*}

For any integer $J_n\ge0$, using
$m_j\le\sqrt{S_j/n}$ for $j\le J_n$, Cauchy--Schwarz and the preceding
summability estimate give
\begin{equation*}
\sum_{j=0}^{J_n}2^{jp}m_j \le \frac{1}{\sqrt n}\sum_{j=0}^{J_n}\left(2^{2pj}S_j\right)^{1/2}\lesssim_p\sqrt{\frac{J_n+1}{n}}.
\end{equation*}
For the remaining layers, using the fact that $m_j\le S_j$ gives
\begin{equation*}
\sum_{j>J_n}2^{jp}m_j\le 2^{-pJ_n}\sum_{j>J_n}2^{2pj}S_j\lesssim_p
2^{-pJ_n}.
\end{equation*}
Consequently,
\begin{equation*}
\sum_{j\ge0}2^{jp}m_j\lesssim_p\sqrt{\frac{J_n+1}{n}}+2^{-pJ_n}.
\end{equation*}
Choosing $J_n=\left\lceil(2p)^{-1}\log_2(n+1)\right\rceil$ makes both terms bounded by a constant multiple of $\sqrt{\log(n+1)/n}.$
Combining this estimate with the contribution of $E^*$ proves the
desired weighted quantile bound. The complete proof is given in
Appendix~\ref{appendix:proofoflem:endpointweightedquantile}.
\end{proof}
Combining Lemmas~\ref{lem:decomposition},
\ref{lem:projected-empirical-ratev2}, and
\ref{lem:endpointweightedquantile}, we obtain the upper bound in
Eq.~\eqref{eq:critical-plugin-risk}, stated in Proposition~\ref{prop:r=2p}.

\begin{proof}[Proof of the upper bound in
Eq.~\eqref{eq:critical-plugin-risk}]
Let $\rho_\ell\coloneqq\sqrt{\log(\ell+1)/\ell}.$ When $p=1$, the result follows from
Lemma~\ref{lem:projected-empirical-ratev2} and the reverse triangle
inequality for $\operatorname W_1$.

For $p>1$, Lemma~\ref{lem:decomposition} bounds the directional error
by
$C_p(\mathsf{A}_{n,\theta}+\mathsf{B}_{m,\theta}+\mathsf{C}_{n,\theta}+\mathsf{D}_{m,\theta}).
$
Lemmas~\ref{lem:projected-empirical-ratev2} and
\ref{lem:endpointweightedquantile}, together with
$\sqrt{u+v}\le\sqrt u+\sqrt v,
$
therefore yield
\begin{equation*}
\mathbb E\left|
\operatorname W_p^p\left(\Pi^\theta_{\#}\widehat\mu_n,
\Pi^\theta_{\#}\widehat\nu_m\right)
-\operatorname W_p^p\left(\Pi^\theta_{\#}\mu,\Pi^\theta_{\#}\nu\right)\right|\lesssim_p
\left[M_{\mu,2p}(\theta)^{1/2}+
M_{\nu,2p}(\theta)^{1/2}\right](\rho_n+\rho_m).
\end{equation*}
Since $n\le m$ and $\ell\mapsto\log(\ell+1)/\ell$ is decreasing,
$\rho_m\le\rho_n$. Integrating over $\mathcal G$ and using
$\mu,\nu\in\mathcal K_{p,2p}(b)$, we obtain
\begin{align*}
\mathcal R_{n,m}^{\mathrm{plug}}
\left(\mathcal K_{p,2p}(b)\right)
\lesssim_p\rho_n\int_{\mathcal{G}}\left[M_{\mu,2p}(\theta)^{1/2}+M_{\nu,2p}(\theta)^{1/2}
\right]d\sigma(\theta)\lesssim_p
b\sqrt{\frac{\log(n+1)}{n}}.
\end{align*}
This proves the desired upper bound.
\end{proof}

\section{Lower-bound analysis}\label{sec:prooflowerbound}
In this section, we establish the lower bounds for the two
noncritical regimes in Theorem~\ref{thm:minimax-noncritical}. The
construction used for $p<r<2p$ remains valid at the endpoint $r=p$
and proves Remark~\ref{rem:minimal-moment-order}. We then consider
the critical regime $r=2p$ and establish the lower bounds in
Eqs.~\eqref{eq:critical-plugin-risk}
and~\eqref{eq:critical-minimax-lower}.

\subsection{The noncritical regimes and the endpoint $r=p$}
\label{sec:noncritical-lower-bounds}
The lower bounds in this subsection rely on Le Cam's two-point
method; see \citet{lecam1973convergence} and
\citet[Chapter~2]{tsybakov2008introduction}. We first record the form
of the two-point bound used below.
\subsubsection{Le Cam's two-point bound}

Let
$
(\mu_0,\nu_0),(\mu_1,\nu_1)\in
\mathcal K_{p,r}(b)\times\mathcal K_{p,r}(b),$
and set
$T_i\coloneqq\operatorname{SW}_p^p(\mu_i,\nu_i)
$ for $i\in\{0,1\}$. Le Cam's
two-point method, combined with Pinsker's inequality, gives
\begin{equation}
\label{eq:lecam}
\mathcal R_{n,m}^{\star}\bigl(\mathcal K_{p,r}(b)\bigr)\ge
\frac{|T_1-T_0|}{4}\left(
1-\sqrt{
\frac{n\operatorname{KL}(\mu_0\|\mu_1)
+m\operatorname{KL}(\nu_0\|\nu_1)}{2}}
\right).
\end{equation}
For completeness, a proof of Eq.~\eqref{eq:lecam} is given in
Appendix~\ref{appendix:proofofeqlecam}.

We apply this bound to two different testing constructions, depending
on the moment regime. Throughout the remainder of this subsection,
let $ e_1=(1,0,\ldots,0)^\top\in\mathbb R^d$ and define $\kappa_{p,d} \coloneqq \int_{\mathbb S^{d-1}}|\langle e_1,\theta\rangle|^p\,\sigma(d\theta).$ We have $\kappa_{p,d}>0$; see for example, \citet[Lemma~25]{tran2025minimax}.
\subsubsection{The regime $r>2p$}
\label{sec:lowerboundr>2p}

Set $a\coloneqq(b/\kappa_{p,d})^{1/p}$ and $h\coloneqq1/(4\sqrt n).$
Consider the two testing pairs
\begin{equation*}
(\mu_0,\nu_0) =\left(
\frac{1}{2}\delta_0+\frac12\delta_{ae_1},
\delta_0\right),\qquad
(\mu_1,\nu_1)=\left( \left(\frac12-h\right)\delta_0
+\left(\frac{1}{2}+h\right)\delta_{ae_1},
\delta_0\right).
\end{equation*}
We first verify that these pairs belong to
$\mathcal K_{p,r}(b)\times\mathcal K_{p,r}(b).$ Clearly, $\nu_0=\nu_1=\delta_0\in\mathcal K_{p,r}(b)$. For
$i\in\{0,1\}$, let
$
\lambda_0\coloneqq1/2
$
and
$
\lambda_1\coloneqq1/2+h.
$
For $i \in \{0,1\}$, since $\lambda_i\le1$ , we have
\begin{align*}
\int_{\mathbb S^{d-1}}
M_{\mu_i,r}(\theta)^{p/r}\,\sigma(d\theta)=\lambda_i^{p/r}a^p\int_{\mathbb S^{d-1}}
|\langle e_1,\theta\rangle|^p\,\sigma(d\theta)=
\lambda_i^{p/r}a^p\kappa_{p,d}
\le b.
\end{align*}
Hence both testing pairs belong to the required parameter space.

Since $\nu_0=\nu_1=\delta_0$, we further have
\begin{equation*}
T_i
=
\operatorname{SW}_p^p(\mu_i,\nu_i)
=
\lambda_i a^p\kappa_{p,d},
\qquad i\in\{0,1\}.
\end{equation*}
Consequently,
\begin{equation*}
|T_1-T_0|
=
h a^p\kappa_{p,d}
=
bh.
\end{equation*}

It remains to control the Kullback--Leibler divergence. Since
$\nu_0=\nu_1$,
$
\operatorname{KL}(\nu_0\|\nu_1)=0.
$
Moreover, a direct calculation gives
\begin{equation*}
\operatorname{KL}(\mu_0\|\mu_1)
=\frac{1}{2}\log \left(\frac{1/2}{1/2-h}\right)+\frac{1}{2}\log\left(\frac{1/2}{1/2+h}\right)=
\frac{1}{2}\log\frac{1}{1-4h^2}
\le
\frac{8}{3}h^2,
\end{equation*}
where the inequality follows from $h\le1/4$ and
$
-\log(1-x)\le x/(1-x) 
$ for $x \in [0,1)$. Substituting these estimates into Eq.~\eqref{eq:lecam}, we obtain
\begin{align*}
\mathcal R_{n,m}^{\star}
\bigl(\mathcal K_{p,r}(b)\bigr)\ge
\frac{bh}{4}
\left(
1-
\sqrt{
\frac{
n\operatorname{KL}(\mu_0\|\mu_1)
}{2}
}
\right)\ge
\frac{bh}{4}
\left(
1-2h\sqrt{\frac n3}
\right)=
\frac{b}{16}
\left(
1-\frac{1}{2\sqrt3}
\right)n^{-1/2}.
\end{align*}
This proves the lower bound in
Theorem~\ref{thm:minimax-noncritical} for $r>2p$.

Although this construction is valid for every $r\ge p$, it yields
the optimal rate only when $r>2p$. The regime $p<r<2p$ requires a
different construction.
\subsubsection{The regime $p<r<2p$ and the endpoint $r=p$}
\label{sec:lowerboundp<r<2p}

Set
$\epsilon\coloneqq 1/(2n)$, $a_\epsilon
\coloneqq\left(\frac{b}{\kappa_{p,d}}\right)^{1/p}\epsilon^{-1/r}$ and $z
\coloneqq\frac{1}{4}
\left(\frac{b}{\kappa_{p,d}}\right)^{1/p}$. 
We consider the following testing pairs:
\begin{equation*}
(\mu_0,\nu_0)=\left(\delta_0,\delta_{ze_1}
\right),\qquad(\mu_1,\nu_1)=\left((1-\epsilon)\delta_0+\epsilon\delta_{a_\epsilon e_1},\delta_{ze_1}
\right).
\end{equation*}
We first verify that both testing pairs belong to
$\mathcal K_{p,r}(b)\times\mathcal K_{p,r}(b).$
Clearly, $\mu_0=\delta_0\in\mathcal K_{p,r}(b)$. Moreover,
\begin{equation*}
\int_{\mathbb S^{d-1}}M_{\nu_0,r}(\theta)^{p/r}\,\sigma(d\theta)
=z^p\kappa_{p,d}=4^{-p}b\le b,
\end{equation*}
and the same holds for $\nu_1=\nu_0$. For $\mu_1$, we have
\begin{align*}
\int_{\mathbb S^{d-1}}M_{\mu_1,r}(\theta)^{p/r}\sigma(d\theta)=
\epsilon^{p/r}a_\epsilon^p\kappa_{p,d}=b.
\end{align*}
Hence both testing pairs belong to the required parameter space.

We further have
\begin{equation*}
T_0=\operatorname{SW}_p^p(\mu_0,\nu_0)=z^p\kappa_{p,d},
\end{equation*}
and
\begin{equation*}
T_1=\operatorname{SW}_p^p(\mu_1,\nu_1)=
\left[(1-\epsilon)z^p+\epsilon(a_\epsilon-z)^p\right]\kappa_{p,d}.
\end{equation*}
Since $\epsilon\le1$, we have $a_\epsilon\ge4z.$
Consequently,
\begin{align*}
T_1-T_0=\epsilon\kappa_{p,d}\left[
(a_\epsilon-z)^p-z^p\right]\ge
\left[\left(\frac34\right)^p-\left(\frac14\right)^p\right]\kappa_{p,d}\epsilon a_\epsilon^p=c_{p,r}b\,n^{-(r-p)/r}.
\end{align*}

We now analyze the Kullback--Leibler terms. Since
$\nu_0=\nu_1$, we have
$\operatorname{KL}(\nu_0\|\nu_1)=0.$
Moreover,
\begin{equation*}
\operatorname{KL}(\mu_0\|\mu_1)
=\mu_0\left(\{0\}\right)\log\left(\frac{\mu_0(\{0\})}{\mu_1(\{0\})}\right)=-\log(1-\epsilon)\le \frac{\epsilon}{1-\epsilon}\le
2\epsilon=\frac{1}{n},
\end{equation*}
where we used $-\log (1-x) \le x/(1-x)$ for $x\in [0,1)$ and $0<\epsilon\le1/2$. Combining these estimates with
Eq.~\eqref{eq:lecam}, we obtain
\begin{align*}
\mathcal R_{n,m}^{\star}\bigl(\mathcal K_{p,r}(b)\bigr)\ge\frac{|T_1-T_0|}{4}\left(1-\sqrt{\frac{n\operatorname{KL}(\mu_0\|\mu_1)}{2}
}\right) \gtrsim_{p,r}b\,n^{-(r-p)/r}.
\end{align*}
This proves the lower bound in
Theorem~\ref{thm:minimax-noncritical} for $p<r<2p$.

In addition, both testing pairs consist of distinct distributions,
which proves Remark~\ref{rem:distinct-distributions}. At the endpoint
$r=p$, the same construction gives $\mathcal R_{n,m}^{\star}\bigl(\mathcal K_{p,p}(b)\bigr)
\gtrsim_p b,$
and therefore proves
Remark~\ref{rem:minimal-moment-order}.
\subsection{The critical regime $r=2p$}
\label{sec:critical-lower-bounds}
At the critical moment order, we establish two different lower bounds.
We first prove the matching lower bound for the plug-in risk in
Eq.~\eqref{eq:critical-plugin-risk}. We then establish the minimax
lower bound in Eq.~\eqref{eq:critical-minimax-lower}.
\subsubsection{Lower bound for the plug-in risk}
At the critical order $r=2p$, the construction used for the regime
$r>2p$ yields only the parametric lower bound $bn^{-1/2}$ and therefore
does not match the plug-in upper bound in
Eq.~\eqref{eq:critical-plugin-risk}. To obtain the matching lower
bound, we choose a pair $(\mu^*,\nu^*)$
such that $\mu^*=\nu^*\in\mathcal K_{p,2p}(b)$. Since
$\operatorname{SW}_p^p(\mu^*,\nu^*)=0
$, the definition of the plug-in risk gives
\begin{equation*}
\mathcal R_{n,m}^{\mathrm{plug}}\left(\mathcal K_{p,2p}(b)\right)\ge
\mathbb E\operatorname{SW}_p^p\left(\widehat\mu_n^*,\widehat\nu_m^*\right),
\end{equation*}
where $\widehat\mu_n^*$ and $\widehat\nu_m^*$ are the empirical
measures associated with two independent samples from $\mu^*$.
Thus, it suffices to construct
$\mu^*\in\mathcal K_{p,2p}(b)$ such that
\begin{equation*}
\mathbb E\operatorname{SW}_p^p\left(
\widehat\mu_n^*,\widehat\nu_m^*\right)\gtrsim_pb\sqrt{\frac{\log(n+1)}{n}}.
\end{equation*}

We first construct such a distribution in one dimension. The general
$d$-dimensional result is then obtained by embedding the
one-dimensional construction into $\mathbb R^d$ through the map
$T(x)=xe_1.$
Before proceeding with the construction, we record a two-sample
consequence of \citet[Lemma~3.8]{bobkov2019one}, which provides the
lower bound on empirical tail fluctuations needed in our
one-dimensional construction. The proof is given in
Appendix~\ref{appendix:proofoflem:extendedlemma3.8ledoux}.
\begin{lemma}\label{lem:extendedlemma3.8ledoux}
Let $0<s \le \frac{1}{2}$ and $n\le m$ such that $ns \ge 1$. Let $\widetilde N\sim \text{Binomial}(n,s)$ and $\widetilde M\sim \text{Binomial}(m,s)$
be independent. Then there exists a universal constant $c>0$ such that
\begin{equation*}
\mathbb E\left|\frac{\widetilde N}{n}-\frac{\widetilde M}{m} \right| \ge c\sqrt{\frac{s}{n}}.
\end{equation*}
\end{lemma}
We now present the desired one-dimensional construction.
\begin{lemma}\label{lem:onedimensionalcaser=2p}
Let $p\ge 1$ and $B>0$. Assume that $1 \le n \le m$. Then there exists a probability measure $P$ on $\mathbb{R}$ such that
\begin{equation*}
\left(\int_{\mathbb{R}} \left|x\right|^{2p}dP(x)\right)^{1/2} \le B.
\end{equation*}
Moreover, if $\widehat P_n$ and $\widehat P_m$ are independent
empirical measures based on $n$ and $m$ i.i.d.\ samples from $P$,
respectively, then
\begin{equation*}
\mathbb{E}\operatorname{W}^p_p(\widehat{P}_n,\widehat{P}_m)\gtrsim_pB\sqrt{\log(n+1)n^{-1}}.
\end{equation*}
\end{lemma}
\begin{proof}[Proof sketch.]
We construct $P$ as a discrete probability measure supported on $0=x_0<x_1<\cdots<x_L$. 

We denote $A_k=[x_k,\infty)$ and $S_k=P(A_k) $ for $k=1,\ldots,L$. The key observation, proved in
Appendix~\ref{appendix:proofoflemmaonedimensionalcaser=2p}, is that for
any two probability measures $\widehat P_n$ and $\widehat P_m$
supported on this grid,
\begin{equation*}
\operatorname W_p^p(\widehat P_n,\widehat P_m)\ge\sum_{k=1}^{L}(x_k-x_{k-1})^p\left|\widehat P_n(A_k)-\widehat P_m(A_k)\right|.
\end{equation*}
Taking expectation, and noting that
$n\widehat{P}_n(A_k)\sim\operatorname{Binomial}(n,S_k)$
and
$m\widehat P_m(A_k)\sim\operatorname{Binomial}(m,S_k)
$ are independent, Lemma~\ref{lem:extendedlemma3.8ledoux} yields, provided
$0<S_k\le1/2$ and $nS_k\ge1,$
\begin{equation}\label{eq:usefuleqconstructionr=2p}
\mathbb E\operatorname W_p^p(\widehat P_n,\widehat P_m)\gtrsim\sum_{k=1}^{L}(x_k-x_{k-1})^p \sqrt{\frac{S_k}{n}}.
\end{equation}

This suggests constructing the support points $(x_k)_{k=1}^L$ and
the tail probabilities $(S_k)_{k=1}^L$ so that the following
conditions hold simultaneously:
\begin{itemize}
\item The sequence $(S_k)_{k=1}^L$ is nonincreasing and defines valid tail
probabilities for a probability measure supported on $0=x_0<x_1<\cdots<x_L.$
\item For every $k=1,\ldots,L$,$0<S_k\le1/2$ and $nS_k\ge1,$ so that Lemma~\ref{lem:extendedlemma3.8ledoux} can be applied.
\item The resulting probability measure satisfies the moment  constraint.
\item The right-hand side of Eq.~\eqref{eq:usefuleqconstructionr=2p}
attains the desired lower-bound order $B\sqrt{\log(n+1)n^{-1}}.$
\end{itemize}

We now make this choice more explicit. Set $S_{L+1}=0$. Then,
\begin{equation*}
\int_{\mathbb R}|x|^{2p}dP(x) =\sum_{k=1}^{L}(S_k-S_{k+1})x_k^{2p}\le\sum_{k=1}^{L}S_kx_k^{2p}.
\end{equation*}
A simple way to satisfy the moment constraint is therefore to
distribute the available budget equally among the $L$ terms by
choosing
\begin{equation*}
S_kx_k^{2p}=\frac{B^2}{L},\qquad k=1,\ldots,L.
\end{equation*}
If the support points are chosen in geometric progression,
$x_k=R2^k,$
then $x_k-x_{k-1}\ge x_k/2,$
and consequently
\begin{equation*}
(x_k-x_{k-1})^p\sqrt{\frac{S_k}{n}}\ge\frac{B}{2^p\sqrt{Ln}}.
\end{equation*}
Hence, summing over the $L$ scales gives
\begin{equation*}
\sum_{k=1}^{L}(x_k-x_{k-1})^p
\sqrt{\frac{S_k}{n}}\ge\frac{B\sqrt L}{2^p\sqrt n}.
\end{equation*}
It remains to choose $L$ and $R$ so that all the preceding
conditions are satisfied while allowing $L$ to be proportional to $\log(n+1)$. An explicit choice, together with the complete verification,
is given in Appendix~\ref{appendix:proofoflemmaonedimensionalcaser=2p}.
\end{proof}
With Lemma~\ref{lem:onedimensionalcaser=2p} in hand, we now lift the
one-dimensional construction to $\mathbb R^d$.
\begin{proof}[Proof of the lower bound in
Eq.~\eqref{eq:critical-plugin-risk}]
Set
$B\coloneqq b/\kappa_{p,d}.$
By Lemma~\ref{lem:onedimensionalcaser=2p}, there exists a probability
measure $P$ on $\mathbb R$ such that
$(\mathbb E_{X\sim P}|X|^{2p})^{1/2}\le B
$
and\begin{equation*}
\mathbb E\operatorname W_p^p(\widehat P_n,\widehat P_m)\gtrsim_p B\sqrt{\frac{\log(n+1)}{n}},
\end{equation*}
where $\widehat P_n$ and $\widehat P_m$ are independent empirical
measures based on $n$ and $m$ i.i.d.\ samples from $P$, respectively.

Let $T:\mathbb R\to\mathbb R^d$ be defined by $T(x)=xe_1$, and set
$\mu=\nu\coloneqq T_{\#}P.
$
For $\theta\in\mathbb S^{d-1}$, writing
$
\theta_1=\langle e_1,\theta\rangle,
$
we have
\begin{align*}
\int_{\mathbb S^{d-1}}
M_{\mu,2p}(\theta)^{1/2}\sigma(d\theta)=
\left(\mathbb E_{X\sim P}|X|^{2p}
\right)^{1/2}\int_{\mathbb S^{d-1}}
|\theta_1|^p\,\sigma(d\theta)\le
\kappa_{p,d}B=b.
\end{align*}
Hence $\mu=\nu\in\mathcal K_{p,2p}(b)$.

Moreover, if
$
\widehat\mu_n=T_{\#}\widehat P_n
$
and
$\widehat\nu_m=T_{\#}\widehat P_m,$
then
\begin{align*}
\mathbb E
\operatorname{SW}_p^p(\widehat\mu_n,\widehat\nu_m)=\kappa_{p,d}\mathbb E
\operatorname W_p^p(\widehat P_n,\widehat P_m)\gtrsim_p  b\sqrt{\frac{\log(n+1)}{n}}.
\end{align*}
Since $\mu=\nu$,
$
\operatorname{SW}_p^p(\mu,\nu)=0.
$
Therefore,
\begin{equation*}
\mathcal R_{n,m}^{\mathrm{plug}}
\left(\mathcal K_{p,2p}(b)
\right)\gtrsim_p b\sqrt{\frac{\log(n+1)}{n}}.
\end{equation*}
\end{proof}

\subsubsection{Minimax lower bound}
Although the two-point construction used in the regime $r>2p$
remains valid at the critical order $r=2p$, it yields only the
parametric lower bound $bn^{-1/2}$. Motivated by
\citet{niles2022estimation} and \citet{kim2022minimax}, we replace the single
alternative in Le Cam's two-point argument with a finite mixture of
alternatives. The following lemma formalizes this reduction. Its proof is deferred to
Appendix~\ref{appendix:proofoflem:twopointmixtureversion}.
\begin{lemma}\label{lem:twopointmixtureversion}
Let $\mathcal A$ be a finite index set with $|\mathcal A|\ge 2$, and let
$$ \mu_0,\nu_0,\{\mu_\alpha:\alpha\in\mathcal A\}
\subseteq \mathcal K_{p,2p}(b)
$$
satisfy $\mu_0=\nu_0$ and
$$ \operatorname{SW}_p^p(\mu_\alpha,\nu_0) = \operatorname{SW}_p^p(\mu_\beta,\nu_0) >0
\quad \text{for all }\alpha,\beta\in\mathcal A.
$$
Fix any $\alpha_\star\in\mathcal A$ and define
$$ \mathsf P_0 \coloneqq \mu_0^{\otimes n},
\quad \mathsf P_1 \coloneqq \frac{1}{|\mathcal A|} \sum_{\alpha\in\mathcal A}
\mu_\alpha^{\otimes n}.
$$
Then
$$ \mathcal R_{n,m}^{\star} \left(\mathcal K_{p,2p}(b)\right) \ge \frac{\operatorname{SW}_p^p(\mu_{\alpha_\star},\nu_0)}{4} \left[
1-\operatorname{TV}(\mathsf P_0,\mathsf P_1) \right].
$$
\end{lemma}
The preceding lemma shows that it suffices to construct a family of
alternatives that are equally separated from the null in the powered
Sliced Wasserstein functional, while the corresponding mixture
observation law remains close to the null law in total variation.
Motivated by the Rademacher-sign construction in
\citet[Appendix~H]{kim2022minimax}, we introduce the following
general block construction.

Fix an integer $L\ge1$ and
\begin{equation*}
0<a_1<b_1<a_2<b_2<\cdots<a_L<b_L.
\end{equation*}
For $k=1,\ldots,L$, set $x_k=a_ke_1$ and $y_k=b_ke_1.$ Let $\omega_1,\ldots,\omega_L>0$ and $\gamma_1,\ldots,\gamma_L\ge0$ satisfy
\begin{equation*}
\Omega\coloneqq\sum_{k=1}^L\omega_k\le1,\qquad\gamma_k\le\frac{\omega_k}{2},\quad k=1,\ldots,L.
\end{equation*}
Define
\begin{equation*}
\mu_0\coloneqq(1-\Omega)\delta_0+\sum_{k=1}^L\frac{\omega_k}{2}\left(\delta_{x_k}+\delta_{y_k}\right),
\end{equation*}
and, for
$\varepsilon=(\varepsilon_1,\ldots,\varepsilon_L)\in\{-1,1\}^L,$
define
\begin{equation*}
\mu_\varepsilon\coloneqq
(1-\Omega)\delta_0+\sum_{k=1}^L\left(\frac{\omega_k}{2}+\varepsilon_k\gamma_k\right)\delta_{x_k}+\sum_{k=1}^L\left(\frac{\omega_k}{2}-\varepsilon_k\gamma_k\right)\delta_{y_k}.
\end{equation*}
The following lemma records the two properties of this construction
that will be used below; its proof is deferred to
Appendix~\ref{appendix:proofoflem:rademacher-block-construction}.

\begin{lemma}\label{lem:rademacher-block-construction}
For every $\varepsilon\in\{-1,1\}^L$, $\mu_0$ and $\mu_\varepsilon$
are probability measures and
\begin{equation*}
\operatorname{SW}_p^p(\mu_\varepsilon,\mu_0)=\kappa_{p,d}\sum_{k=1}^L
\gamma_k(b_k-a_k)^p.
\end{equation*}
Moreover, defining
\begin{equation*}
\mathsf P_0\coloneqq\mu_0^{\otimes n},\qquad\mathsf P_1\coloneqq
2^{-L}\sum_{\varepsilon\in\{-1,1\}^L}
\mu_\varepsilon^{\otimes n},
\end{equation*}
we have
\begin{equation*}
\operatorname{TV}(\mathsf P_0,\mathsf P_1)
\le\frac12\left[\exp\left(8n^2\sum_{k=1}^L\frac{\gamma_k^4}{\omega_k^2}
\right)-1\right]^{1/2}.
\end{equation*}
\end{lemma}
Combining Lemmas~\ref{lem:twopointmixtureversion} and
\ref{lem:rademacher-block-construction}, provided that $\mu_0,\mu_\varepsilon\in\mathcal K_{p,2p}(b)$ for every $\varepsilon\in\{-1,1\}^L$, we obtain
\begin{align}
\mathcal R_{n,m}^{\star}\bigl(\mathcal K_{p,2p}(b)\bigr)\ge\frac{\kappa_{p,d}}{4}\sum_{k=1}^L
\gamma_k(b_k-a_k)^p\left\{1-\frac12\left[\exp\left(8n^2\sum_{k=1}^L\frac{\gamma_k^4}{\omega_k^2}
\right)-1\right]^{1/2}\right\}.
\label{eq:critical-minimax-reduction}
\end{align}
Thus, the parameter choice is guided by two competing objectives:
we want to make
$\sum_{k=1}^L\gamma_k(b_k-a_k)^p$ as large as possible, while keeping
$n^2\sum_{k=1}^L\gamma_k^4/\omega_k^2$
bounded and ensuring that all the constructed measures satisfy the
projected moment constraint.

To make this choice more transparent, we first examine the projected
moment constraint. For every $\varepsilon\in\{-1,1\}^L$,
\begin{align*}
\int_{\mathbb S^{d-1}}M_{\mu_\varepsilon,2p}(\theta)^{1/2}\sigma(d\theta)=\kappa_{p,d}\left[\sum_{k=1}^L\left(\frac{\omega_k}{2}+\varepsilon_k\gamma_k\right)a_k^{2p}
+\sum_{k=1}^L\left(\frac{\omega_k}{2}-\varepsilon_k\gamma_k\right)b_k^{2p}\right]^{1/2}.
\end{align*}
Since $\gamma_k\le \omega_k/2$ and $a_k<b_k$, we obtain
\begin{equation*}
\int_{\mathbb S^{d-1}}M_{\mu_\varepsilon,2p}(\theta)^{1/2}\sigma(d\theta)\le\kappa_{p,d}\left(2\sum_{k=1}^L \omega_kb_k^{2p}\right)^{1/2}.
\end{equation*}
Similarly,
\begin{align*}
\int_{\mathbb S^{d-1}}M_{\mu_0,2p}(\theta)^{1/2}\sigma(d\theta)=
\kappa_{p,d}\left[\sum_{k=1}^L\frac{\omega_k}{2}\left(a_k^{2p}+b_k^{2p}\right)\right]^{1/2}\le\kappa_{p,d}
\left(\sum_{k=1}^L\omega_kb_k^{2p}\right)^{1/2}.
\end{align*}
Hence, a sufficient condition for
$\mu_0,\mu_\varepsilon\in\mathcal K_{p,2p}(b)$ for every
$\varepsilon\in\{-1,1\}^L$ is
\begin{equation*}
\sum_{k=1}^L\omega_kb_k^{2p}\le\frac{b^2}{2\kappa_{p,d}^2}.
\end{equation*}
We therefore distribute this upper-bound budget equally among the
$L$ blocks by imposing
\begin{equation}
\label{eq:critical-minimax-moment-allocation}
\omega_kb_k^{2p}=\frac{b^2}{2\kappa_{p,d}^2L},
\qquad k=1,\ldots,L.
\end{equation}

We next consider the total variation term. A simple way to keep this quantity bounded uniformly in $L$ is to
distribute the corresponding budget equally among the $L$ blocks.
We therefore choose
\begin{equation}
\label{eq:critical-minimax-perturbation}
\gamma_k=\frac{1}{4}\left(\frac{\omega_k}{n\sqrt L}\right)^{1/2},\qquad k=1,\ldots,L,
\end{equation}
for which
\begin{equation*}
8n^2\sum_{k=1}^L\frac{\gamma_k^4}{\omega_k^2}
=\frac{1}{32}.
\end{equation*}
Consequently,
\begin{equation*}
\operatorname{TV}(\mathsf P_0,\mathsf P_1)\le\frac{1}{2}\left(e^{1/32}-1\right)^{1/2}<\frac{1}{4}.
\end{equation*}

Equation~\eqref{eq:critical-minimax-moment-allocation} and the ordering
of $(b_k)_{k=1}^L$ imply that $(\gamma_k)_{k=1}^L$ must be decreasing.
Moreover, Eq.~\eqref{eq:critical-minimax-perturbation} shows that the
admissibility condition $\gamma_k\le \omega_k/2$ is ensured whenever $\omega_k\ge\frac{1}{4n\sqrt L}.$ This suggests choosing $(\omega_k)_{k=1}^L$ as a decreasing geometric sequence and then defining $(b_k)_{k=1}^L$ through
Eq.~\eqref{eq:critical-minimax-moment-allocation}.

On the other hand, Eq.~\eqref{eq:critical-minimax-reduction} and the
preceding total variation bound yield
\begin{align*}
\mathcal R_{n,m}^{\star}\bigl(\mathcal K_{p,2p}(b)\bigr)\ge\frac{3\kappa_{p,d}}{16}
\sum_{k=1}^L\gamma_k(b_k-a_k)^p =\frac{3b}{64\sqrt2\,\sqrt n\,L^{3/4}}
\sum_{k=1}^L\left(1-\frac{a_k}{b_k}\right)^p.
\end{align*}
We may therefore set $a_k=cb_k$ for $c\in(0,1)$, which gives
\begin{equation*}
\mathcal R_{n,m}^{\star}\left(\mathcal K_{p,2p}(b)\right)\ge\frac{3(1-c)^p}{64\sqrt2}b\frac{L^{1/4}}{\sqrt n}.
\end{equation*}
The geometric ratio of $(\omega_k)$ and the constant $c$ can be chosen so
that the ordering and probability constraints hold while allowing $L$ to be proportional to $\log (n+1)$. An explicit choice and the
complete verification are given in Appendix~\ref{appendix:proofoftheminimaxlowerboundr=2p}. For that
choice, the preceding bound yields
\begin{equation*}
\mathcal R_{n,m}^{\star}\left(\mathcal K_{p,2p}(b)\right)\gtrsim_p b\frac{\left[\log(n+1)\right]^{1/4}}{\sqrt n}.
\end{equation*}
It is unclear whether a different choice of parameters within the
present construction can yield a stronger minimax lower bound.
More generally, determining the exact minimax rate at the critical
moment order $r=2p$ remains an open problem.
\begin{appendix}
\section*{Appendices}\label{appn} %% if no title is needed, leave empty \section*{}.
% Appendices should be provided in \verb|{appendix}| environment,
% before Acknowledgements.

% If there is only one appendix,
% then please refer to it in text as \ldots\ in the \hyperref[appn]{Appendix}.
\end{appendix}
%%%%%%%%%%%%%%%%%%%%%%%%%%%%%%%%%%%%%%%%%%%%%%
%% Example with multiple Appendixes:        %%
%%%%%%%%%%%%%%%%%%%%%%%%%%%%%%%%%%%%%%%%%%%%%%
\begin{appendix}
\section{Proof of Lemma~\ref{lem:PropertiesK_pr}}\label{appendixA}
\begin{proof}
We first prove that
$ \mathcal K_{p,r}(b)\subseteq\mathcal P_p(\mathbb R^d).$
Define
$$\kappa_{p,d}\coloneqq \int_{\mathbb S^{d-1}} \left|\langle e_1,\theta\rangle\right|^pd\sigma(\theta),
$$
where $e_1=(1,0,\ldots,0)^\top\in\mathbb R^d$. By rotational invariance of $\sigma$ (see, for example,
\citet[Proposition~11.24]{peyre2025optimal}), for every
$x\in\mathbb R^d$,
\begin{equation}
\label{eq:spherical-moment-identity}
\int_{\mathbb S^{d-1}}\left|\langle x,\theta\rangle\right|^p
d\sigma(\theta)=\kappa_{p,d}\|x\|^p.
\end{equation}
Moreover, $\kappa_{p,d}>0$.

Let $\mu\in\mathcal K_{p,r}(b)$ and let $X\sim\mu$. Since the integrand
is nonnegative, Tonelli's theorem and
\eqref{eq:spherical-moment-identity} yield
\begin{equation*}
\kappa_{p,d}\,
\mathbb E_{X\sim\mu}\|X\|^p= \mathbb E_{X\sim\mu} \left[
\int_{\mathbb S^{d-1}} \left|\langle X,\theta\rangle\right|^p d\sigma(\theta)
\right]= \int_{\mathbb S^{d-1}}
\mathbb E_{X\sim\mu} \left|\langle X,\theta\rangle\right|^p d\sigma(\theta).
\end{equation*}
Since $r\ge p$, the monotonicity of moments gives, for every
$\theta\in\mathbb S^{d-1}$,
$$
\mathbb E_{X\sim\mu}
\left|\langle X,\theta\rangle\right|^p
\le \left( \mathbb E_{X\sim\mu}
\left|\langle X,\theta\rangle\right|^r
\right)^{p/r} = M_{\mu,r}(\theta)^{p/r}.
$$
Consequently,
$$ \kappa_{p,d} \mathbb E_{X\sim\mu}\|X\|^p \le \int_{\mathbb S^{d-1}} M_{\mu,r}(\theta)^{p/r} d\sigma(\theta)\le b.
$$
Since $\kappa_{p,d}>0$, it follows that
$$
\mathbb E_{X\sim\mu}\|X\|^p \le \frac{b}{\kappa_{p,d}} <+\infty.
$$
Thus,
$
\mu\in\mathcal P_p(\mathbb R^d).
$

We next prove the inclusion for distributions with bounded Euclidean
$r$-th moments. Let $\mu\in\mathcal P_r(\mathbb R^d)$ satisfy $\mathbb E_{X\sim\mu}\|X\|^r\le B.$
For every $\theta\in\mathbb S^{d-1}$, the Cauchy--Schwarz inequality
gives
$$
|\langle X,\theta\rangle| \le \|X\|\|\theta\| = \|X\|.
$$
Therefore,
$$ M_{\mu,r}(\theta) =
\mathbb E_{X\sim\mu} \left|\langle X,\theta\rangle\right|^r \le \mathbb E_{X\sim\mu}\|X\|^r \le B.
$$
Since $\sigma$ is a probability measure,
$$\int_{\mathbb S^{d-1}} M_{\mu,r}(\theta)^{p/r} d\sigma(\theta) \le B^{p/r}.
$$
Hence,
$$
\mu\in\mathcal K_{p,r}\left(B^{p/r}\right).
$$

Suppose now that $\operatorname{supp}(\mu)\subseteq\mathbb B(0,R).$
Then $\|X\|\le R$ almost surely, and hence
$$
\mathbb E_{X\sim\mu}\|X\|^r\le R^r.
$$
Applying the preceding inclusion with $B=R^r$, we obtain
$$
\mu \in \mathcal K_{p,r} \left( (R^r)^{p/r} \right) = \mathcal K_{p,r}(R^p).
$$

Finally, let $\mu,\nu\in\mathcal K_{p,r}(b)$. For every
$\theta\in\mathbb S^{d-1}$, the triangle inequality for
$\operatorname{W}_p$ gives
$$ \operatorname{W}_p \left(
\Pi^\theta_{\#}\mu, \Pi^\theta_{\#}\nu
\right)\le \operatorname{W}_p
\left(\Pi^\theta_{\#}\mu,\delta_0
\right) + \operatorname{W}_p
\left(\Pi^\theta_{\#}\nu,\delta_0 \right).
$$
Using $(a+c)^p\le2^{p-1}(a^p+c^p) $ for $a,c\ge0, p\ge 1$, we obtain
\begin{align*}
\operatorname{W}_p^p \left(
\Pi^\theta_{\#}\mu,\Pi^\theta_{\#}\nu\right)&\le 2^{p-1} \operatorname{W}_p^p
\left( \Pi^\theta_{\#}\mu,\delta_0
\right)+2^{p-1}
\operatorname{W}_p^p\left(\Pi^\theta_{\#}\nu,\delta_0\right)\\&=2^{p-1}
\mathbb E_{X\sim\mu}\left|\langle X,\theta\rangle\right|^p+2^{p-1}
\mathbb E_{Y\sim\nu}\left|\langle Y,\theta\rangle\right|^p\\&\le
2^{p-1} \left(
M_{\mu,r}(\theta)^{p/r}+M_{\nu,r}(\theta)^{p/r}\right),
\end{align*}
where the last inequality follows again from $r\ge p$. Integrating
over $\mathbb S^{d-1}$ yields
\begin{align*}
\operatorname{SW}_p^p(\mu,\nu)
&\le 2^{p-1} \int_{\mathbb S^{d-1}}
M_{\mu,r}(\theta)^{p/r} d\sigma(\theta)+
2^{p-1} \int_{\mathbb S^{d-1}}
M_{\nu,r}(\theta)^{p/r} d\sigma(\theta)\\
&\le 2^{p-1}b+2^{p-1}b=2^p b.
\end{align*}
This completes the proof.
\end{proof}
\section{Proofs of the Technical Results in Section~\ref{sec:technicallemmas}}
In this appendix, we provide the complete proofs of the technical results stated in Section~\ref{sec:technicallemmas}.
\subsection{Proof of Lemma~\ref{lem:projected-empirical-rate}}\label{appendix:proofofonesamplebounds}
\begin{proof}
Fix $\theta\in \mathbb{S}^{d-1}$ such that $M_{Q,s}(\theta)<\infty$. Assume first that $s\neq 2q$. By applying \citet[Theorem~1]{fournier2015rate} in dimension one with moment order $s$, we obtain $$\mathbb E\left[\operatorname{W}_q^q \left( \Pi^\theta_{\#}\widehat{Q}_n,\Pi^\theta_{\#}Q\right) \right]\le C_{q,s}\left(M_{Q,s}(\theta)\right)^{q/s} \left(n^{-1/2}+n^{-(s-q)/s}\right). $$ If $q<s<2q$, then $(s-q)/s<1/2$. Hence $n^{-1/2}\le n^{-(s-q)/s}.$ Therefore, $$ n^{-1/2}+n^{-(s-q)/s}\le2n^{-(s-q)/s}. $$ If $s>2q$, then $(s-q)/s>1/2$. Hence $n^{-(s-q)/s}\le n^{-1/2}.$ Therefore, $$ n^{-1/2}+n^{-(s-q)/s}\le 2n^{-1/2}. $$ Absorbing the factor $2$ into the constant $C_{q,s}$ gives the desired bound in both the first and the last regimes.

It remains to consider the critical case $s=2q$. This case is not stated explicitly in \citet[Theorem~1]{fournier2015rate}, but it follows from the argument in Step~2 in the proof of their Theorem~1. Since the ambient dimension is one and $q\ge1$, the condition $q>1/2$ required in that argument is automatically satisfied. We therefore obtain
\begin{equation*}
\mathbb E\left[\operatorname{W}_q^q \left( \Pi^\theta_{\#}\widehat{Q}_n,\Pi^\theta_{\#}Q\right) \right] \lesssim_{q}M_{Q,s}\left(\theta\right)^{q/s} \sum_{\ell \ge 0}2^{q\ell}\min \left\{2^{-s\ell},\left(\frac{2^{-s\ell}}{n}\right)^{1/2}\right\}.
\end{equation*}
Since $s=2q$, we have $2^{q\ell}2^{-s\ell}=2^{-q\ell}$, whereas
$2^{q\ell}\sqrt{2^{-s\ell}n^{-1}}=n^{-1/2}$. Thus
\begin{equation*}
\mathbb E\left[\operatorname{W}_q^q \left( \Pi^\theta_{\#}\widehat{Q}_n,\Pi^\theta_{\#}Q\right) \right]\lesssim_{q}M_{Q,s}(\theta)^{q/s}S_n,
\end{equation*}
where $S_n\coloneqq \sum_{\ell\ge 0} \min \left\{2^{-q\ell},n^{-1/2}\right\}$. For $n=1$, we have
\begin{equation*}
S_1=\sum_{\ell \ge 0}2^{-q\ell}=\frac{1}{1-2^{-q}},
\end{equation*}
which is bounded by a constant multiple of $\log 2$. Assume now that $n \ge 2$ and define
$$
L_n\coloneqq \left\lceil \frac{\log_2 n}{2q} \right\rceil.
$$
Since $n\ge2$, we have $\log_2 n>0$, and hence $L_n\ge1$. Moreover, by the definition of the ceiling function,
$$L_n\ge\frac{\log_2 n}{2q}.$$
Therefore,
$$ 2^{-qL_n} \le 2^{-q\frac{\log_2 n}{2q}}
= 2^{-\frac12\log_2 n} =n^{-1/2}.
$$
We split
$$ S_n = \sum_{\ell=0}^{L_n-1} \min\left\{
2^{-q\ell}, n^{-1/2} \right\} + \sum_{\ell=L_n}^{\infty} \min\left\{ 2^{-q\ell}, n^{-1/2} \right\}.
$$
We have
$$ \sum_{\ell=0}^{L_n-1}\min\left\{ 2^{-q\ell},
n^{-1/2} \right\} \le L_n n^{-1/2} \quad 
\sum_{\ell=L_n}^{\infty}
\min\left\{ 2^{-q\ell}, n^{-1/2} \right\}
\le \sum_{\ell=L_n}^{\infty}2^{-q\ell} =
\frac{2^{-qL_n}}{1-2^{-q}} \le \frac{1}{1-2^{-q}}n^{-1/2}.
$$
Combining the two estimates, we obtain
$$
S_n \le L_nn^{-1/2} + \frac{1}{1-2^{-q}}n^{-1/2} = \left( L_n+\frac{1}{1-2^{-q}}
\right)n^{-1/2}.
$$
Moreover, we also have
$$ L_n \le \frac{\log_2 n}{2q}+1 = \frac{\log n}{2q\log 2}+1 \le \frac{\log (n+1)}{2q\log 2}+1.
$$
As a consequence,
$$
S_n\le\left(\frac{\log(n+1)}{2q\log 2}+1+\frac{1}{1-2^{-q}}
\right)n^{-1/2}.
$$
Since $\log(n+1)\ge \log 2$ for every $n\ge1$, the constant terms can be absorbed into a multiple of $\log(n+1)$. Hence,
$$S_n\lesssim_{q}n^{-1/2}\log(n+1).
$$
Combining the preceding estimates yields
\begin{equation*}
\mathbb{E}\left[\operatorname{W}^q_q\left(\Pi^{\theta}_{\#}\widehat{Q}_n,\Pi^{\theta}_{\#}Q\right)\right]\lesssim_{q}M_{Q,s}\left(\theta\right)^{q/s}n^{-1/2}\log(n+1)
\end{equation*}
when $s=2q$. This proves the desired estimate in all three regimes.
\end{proof}

\subsection{Proof of Lemma~\ref{lem:decomposition}}\label{appendix:proofoflemdecomposition}
To prove Lemma~\ref{lem:decomposition}, we first need the following elementary result.
\begin{lemma}\label{lem:sec5basicinequality}
Let $p> 1$. Then, for all $a,b\ge 0$,
$$|a^p-b^p|\le p2^{p-1}
\left(b^{p-1}|a-b|+|a-b|^p\right).
$$
Consequently, for all $x,y\in\mathbb{R}$,
$$
\left||x+y|^p-|x|^p\right|\le
p2^{p-1}\left(
|x|^{p-1}|y|+|y|^p\right).
$$
\end{lemma}

\begin{proof}
We first prove the inequality for $a,b\ge0$. Consider the function
$f(t)=t^p$ on $[0,\infty)$. By the mean value theorem, there exists a point
$\xi$ between $a$ and $b$ such that
$$
a^p-b^p=f'( \xi )(a-b)=p\xi^{p-1}(a-b).
$$
Hence
$$|a^p-b^p|\le p|a-b|\max(a,b)^{p-1}.
$$
Moreover,
$$\max(a,b)\le b+|a-b|.
$$
Therefore,
$$|a^p-b^p| \le p|a-b| \left(b+|a-b|\right)^{p-1}.
$$
Using the elementary inequality
$$
(u+v)^{p-1} \le2^{p-1}\left(u^{p-1}+v^{p-1}\right),\qquad u,v\ge0,
$$
we obtain
$$ |a^p-b^p| \le p2^{p-1}|a-b| \left( b^{p-1}+|a-b|^{p-1} \right).
$$
Hence
$$
|a^p-b^p| \le p2^{p-1} \left( b^{p-1}|a-b|+|a-b|^p \right).
$$

Now take
$$ a=|x+y|, \qquad b=|x|.
$$
By the reverse triangle inequality,
$$|a-b|=\left||x+y|-|x|\right|\le|y|.
$$
Substituting this into the previous estimate gives
$$ \left||x+y|^p-|x|^p\right|\le p2^{p-1}\left(|x|^{p-1}|y|+|y|^p\right).
$$
This proves the claim.
\end{proof}
% Now we give the proof of Lemma~\ref{lem:decomposition}
\begin{proof}[Proof of Lemma~\ref{lem:decomposition}]
% Throughout the proof, $C_p>0$ denotes a constant depending only on $p$, whose value may change from line to line.

Since $\theta \in \mathcal{G}$ and $r \ge p$, the projected population
measures have finite $p$-th moments. The projected empirical measures are finitely supported and hence also have finite $p$-th moments. Therefore, by the one-dimensional quantile representation of the Wasserstein distance (see, for example, \citet[Theorem~2.10]{bobkov2019one}), we have
$$
\operatorname{W}_p^p
\left( \Pi^\theta_{\#}\widehat{\mu}_n, \Pi^\theta_{\#}\widehat{\nu}_m \right)
=\int_0^1\left|F_{n,\theta}^{-1}(u)-G_{m,\theta}^{-1}(u)\right|^pdu
$$
and
$$
\operatorname{W}_p^p \left( \Pi^\theta_{\#}\mu, \Pi^\theta_{\#}\nu
\right)=\int_0^1\left|F_\theta^{-1}(u)-G_\theta^{-1}(u)\right|^pdu.
$$
Therefore,
$$
\left|\operatorname{W}_p^p\left(\Pi^\theta_{\#}\widehat{\mu}_n,
\Pi^\theta_{\#}\widehat{\nu}_m \right)-\operatorname{W}_p^p\left(\Pi^\theta_{\#}\mu,\Pi^\theta_{\#}\nu\right)\right|\le
\int_0^1\left|\left|
F_{n,\theta}^{-1}(u)-G_{m,\theta}^{-1}(u)
\right|^p-\left|F_\theta^{-1}(u)-G_\theta^{-1}(u)\right|^p
\right|du.
$$

For each $u\in(0,1)$, set
$$
x(u)=F_\theta^{-1}(u)-G_\theta^{-1}(u)
$$
and
$$
y(u)=\left(F_{n,\theta}^{-1}(u)-F_\theta^{-1}(u)\right)
-\left(G_{m,\theta}^{-1}(u)-G_\theta^{-1}(u)
\right).
$$
Then
$$
F_{n,\theta}^{-1}(u)-G_{m,\theta}^{-1}(u)=
x(u)+y(u).
$$
Using Lemma~\ref{lem:sec5basicinequality}, we obtain
$$
\begin{aligned}
&\left|
\left|
F_{n,\theta}^{-1}(u)-G_{m,\theta}^{-1}(u)
\right|^p-\left|F_\theta^{-1}(u)-G_\theta^{-1}(u)\right|^p\right|
\\
&\le
C_p\left|F_\theta^{-1}(u)-G_\theta^{-1}(u)
\right|^{p-1}\left|\left(F_{n,\theta}^{-1}(u)-F_\theta^{-1}(u)\right)
-\left(G_{m,\theta}^{-1}(u)-G_\theta^{-1}(u)\right)\right|
\\
&\quad
+C_p\left|\left(F_{n,\theta}^{-1}(u)-F_\theta^{-1}(u)
\right)-\left(G_{m,\theta}^{-1}(u)-G_\theta^{-1}(u)\right)\right|^p.
\end{aligned}
$$

Next, we use
$$
\left|F_\theta^{-1}(u)-G_\theta^{-1}(u)
\right|^{p-1}\le C_p\left(|F_\theta^{-1}(u)|^{p-1}+|G_\theta^{-1}(u)|^{p-1}\right).$$
Thus,
$$
\begin{aligned}
&\left|\left|F_{n,\theta}^{-1}(u)-G_{m,\theta}^{-1}(u)\right|^p-\left|
F_\theta^{-1}(u)-G_\theta^{-1}(u)\right|^p\right|
\\
&\le C_p \left(
|F_\theta^{-1}(u)|^{p-1}+|G_\theta^{-1}(u)|^{p-1}\right)\left|F_{n,\theta}^{-1}(u)-F_\theta^{-1}(u)\right|
\\
&\quad
+C_p\left(|F_\theta^{-1}(u)|^{p-1}+
|G_\theta^{-1}(u)|^{p-1}\right)
\left|G_{m,\theta}^{-1}(u)-G_\theta^{-1}(u)
\right|
\\
&\quad
+C_p\left|F_{n,\theta}^{-1}(u)-F_\theta^{-1}(u)\right|^p
+C_p\left|G_{m,\theta}^{-1}(u)-G_\theta^{-1}(u)\right|^p.
\end{aligned}
$$
Integrating over $u\in(0,1)$ and taking expectation gives
$$
\mathbb{E}_{\mu,\nu}
\left|
\operatorname{W}_p^p\left(\Pi^\theta_{\#}\widehat{\mu}_n,\Pi^\theta_{\#}\widehat{\nu}_m
\right)-\operatorname{W}_p^p\left(\Pi^\theta_{\#}\mu,
\Pi^\theta_{\#}\nu\right)\right|
\lesssim_{p}\left(\mathsf{A}_{n,\theta}+\mathsf{B}_{m,\theta}+\mathsf{C}_{n,\theta}+\mathsf{D}_{m,\theta}
\right).
$$
This concludes the proof.
\end{proof}
\subsection{Proof of Proposition~\ref{prop:peelingtool}}\label{appendix:proofofprop:peelingtool}
We first establish two auxiliary lemmas that will be used in the proof of
Proposition~\ref{prop:peelingtool}.

It is well known that, for one-dimensional probability measures with finite first moments,
\begin{equation*}
\mathbb{E}\int_0^1
\left|F_{n,\theta}^{-1}(u)-F_\theta^{-1}(u)\right|du=\mathbb{E}\int_{\mathbb R}
\left|F_{n,\theta}(x)-F_\theta(x)\right|dx;
\end{equation*}
see, for instance, \citet[Theorem~2.10]{bobkov2019one}. For our purposes, however, we need localized versions of this identity near the two endpoints of $(0,1)$. The following lemma provides the required representations.
\begin{lemma}\label{lem:Quantile-CDF}
Fix $\theta\in \mathcal{G}$, where $\mathcal{G}$ is defined in Eq.~\eqref{eq:gooddirections}. Denote by $F_{\theta}$ the distribution
function of $\Pi^\theta_{\#}\mu$, and by $F_{n,\theta}$ the empirical
distribution function of $\Pi^\theta_{\#}\widehat{\mu}_n$, where
$\widehat{\mu}_n$ is the empirical measure associated with $n$ i.i.d. samples
from $\mu$. Then, for every $\alpha,\beta\in[0,1]$,
\begin{equation*}
\int_0^\alpha\left|F_{n,\theta}^{-1}(u)-F_{\theta}^{-1}(u)\right|du=
\int_{\mathbb{R}}
\left|\min\{F_{n,\theta}(x),\alpha\}
-\min\{F_{\theta}(x),\alpha\}
\right|dx,
\end{equation*}
and
\begin{equation*}
\int_\beta^1\left|F_{n,\theta}^{-1}(u)-F_{\theta}^{-1}(u)\right|du=
\int_{\mathbb{R}}\left|
\max\{F_{n,\theta}(x),\beta\}-
\max\{F_{\theta}(x),\beta\}\right|dx.
\end{equation*}
\end{lemma}
\begin{proof}
Recall that, for every $u \in (0,1)$, and $x \in \mathbb{R}$,
\begin{equation*}
F^{-1}_{\theta}(u) \le x \Leftrightarrow u \le F_{\theta}(x),
\end{equation*}
and 
\begin{equation*}
F^{-1}_{n,\theta}(u) \le x \Leftrightarrow u \le F_{n,\theta}(x).
\end{equation*}
A proof of these equivalences can be found in
\citet[Lemma~A.3]{bobkov2019one}. Moreover, for every $a,b \in \mathbb{R}$, we have
\begin{equation*}
\left|a-b\right| = \int_{\mathbb{R}}\left|\mathbbm{1}_{\{a \le x\}}-\mathbbm{1}_{\{b \le x\}}\right|dx.
\end{equation*}
Therefore,
\begin{equation*}
\left|F^{-1}_{n,\theta}(u)-F^{-1}_{\theta}(u)\right|= \int_{\mathbb{R}}\left|\mathbbm{1}_{\{u \le F_{n,\theta}(x)\}}-\mathbbm{1}_{\{u \le F_{\theta}(x)\}}\right|dx
\end{equation*}
We treat the two localized identities separately.

\textbf{Lower tail.} Fix $\alpha \in [0,1]$. Since the integrand is nonnegative, Tonelli's theorem gives
\begin{equation*}
\int_{0}^{\alpha}\left|F^{-1}_{n,\theta}(u)-F^{-1}_{\theta}(u)\right| du= \int_{\mathbb{R}}\int_0^{\alpha}\left|\mathbbm{1}_{\{u\le F_{n,\theta}(x)\}}-\mathbbm{1}_{\{u \le F_{\theta}(x)\}}\right|dudx.
\end{equation*}
Besides, a direct calculation gives, for every $x\in \mathbb{R}$
\begin{equation*}
\int_0^{\alpha} \left|\mathbbm{1}_{\{u\le F_{n,\theta}(x)\}}-\mathbbm{1}_{\{u \le F_{\theta}(x)\}}\right|du=\left|\min \{F_{n,\theta}(x),\alpha\}-\min \{F_{\theta}(x),\alpha\}\right|.
\end{equation*}
Consequently,
\begin{equation*}
\int_0^\alpha\left|F_{n,\theta}^{-1}(u)-F_{\theta}^{-1}(u)\right|du=
\int_{\mathbb{R}}
\left|\min\{F_{n,\theta}(x),\alpha\}
-\min\{F_{\theta}(x),\alpha\}
\right|dx.
\end{equation*}

\textbf{Upper tail.} Fix $\beta \in [0,1]$. Again, by Tonelli's theorem,
\begin{equation*}
\int_{\beta }^{1}\left|F^{-1}_{n,\theta}(u)-F^{-1}_{\theta}(u)\right|du = \int_{\mathbb{R}}\int_\beta^{1}\left|\mathbbm{1}_{\{u\le F_{n,\theta}(x)\}}-\mathbbm{1}_{\{u \le F_{\theta}(x)\}}\right|dudx.
\end{equation*}
Similarly, a direct calculation gives, for every $x\in\mathbb{R}$,
\begin{equation*}
\int_{\beta}^1 \left|\mathbbm{1}_{\{u \le F_{n,\theta}(x)\}}-\mathbbm{1}_{\{u \le F_{\theta}(x)\}}\right|du= \left|\max \{F_{n,\theta}(x),\beta\}-\max \{F_{\theta}(x),\beta\}\right|.
\end{equation*}
As a result,
\begin{equation*}
\int_\beta^1\left|F_{n,\theta}^{-1}(u)-F_{\theta}^{-1}(u)\right|du=
\int_{\mathbb{R}}\left|
\max\{F_{n,\theta}(x),\beta\}-
\max\{F_{\theta}(x),\beta\}\right|dx.
\end{equation*}
\end{proof}
We also need a simple estimate on the Lebesgue measure of CDF-level intervals.
\begin{lemma}[Length bound for CDF-level intervals]\label{lem:cdf-level-interval}
Let $\mu\in\mathcal K_{p,r}(b)$, where $\mathcal K_{p,r}(b)$ is defined in Eq.~\eqref{sec1:eqclass}. Fix $\theta\in \mathcal{G}$, where $\mathcal{G}$ is defined in Eq.~\eqref{eq:gooddirections}, and let $F_\theta$ denote the distribution function of $\Pi^\theta_{\#}\mu$. Then, for any $0<a<c<1$,
$$ \left| \left\{ x\in\mathbb R:
a<F_\theta(x)\le c\right\}
\right| \le 2M_{\mu,r}(\theta)^{1/r}
\max\left\{ a^{-1/r}, (1-c)^{-1/r} \right\},
$$
where $M_{\mu,r}(\theta) \coloneqq \mathbb E_{X\sim\mu}|\langle X,\theta\rangle|^r.
$
\end{lemma}
\begin{proof}
By the definition of $\mathcal{G}$, we have $M_{\mu,r}(\theta)<+\infty$. 

We first consider the degenerate case $M_{\mu,r}(\theta)=0$. In this case,
$$
|\langle X,\theta\rangle|=0
\quad\text{almost surely}.
$$
Hence
$$
F_\theta(x) =
\begin{cases}
0, & x<0,\\1, & x\ge0.
\end{cases}
$$
Since $0<a<c<1$, we have
$$\left\{x\in\mathbb R:a<F_\theta(x)\le c\right\} = \emptyset.
$$
Therefore,
$$
\left| \left\{ x\in\mathbb R: a<F_\theta(x)\le c \right\} \right| =0.
$$
The desired inequality follows immediately in this case.

We now assume that $0<M_{\mu,r}(\theta)<+\infty$. Define
$$
R_a\coloneqq\left(\frac{M_{\mu,r}(\theta)}{a}\right)^{1/r},\qquad
R_c\coloneqq
\left(\frac{M_{\mu,r}(\theta)}{1-c}\right)^{1/r}.
$$
We claim that
$$\left\{x\in\mathbb R:
a<F_\theta(x)\le c
\right\}\subseteq(-R_a,R_c].
$$

Indeed, for every $x\le -R_a<0$, Markov's inequality gives
$$
F_\theta(x) =
\mathbb P_{X\sim\mu}\left(\langle X,\theta\rangle\le x\right) \le
\mathbb P_{X\sim\mu}\left(|\langle X,\theta\rangle|\ge |x|\right) \le
\frac{M_{\mu,r}(\theta)}{|x|^r}
\le \frac{M_{\mu,r}(\theta)}{R_a^r}
= a.
$$
Thus, if $F_\theta(x)>a$, then necessarily $x>-R_a$.

Similarly, for every $x>R_c>0$, Markov's inequality gives
$$
1-F_\theta(x) = \mathbb P_{X\sim\mu}\left(\langle X,\theta\rangle>x\right) \le \mathbb P_{X\sim\mu}\left(|\langle X,\theta\rangle|\ge x\right)
\le
\frac{M_{\mu,r}(\theta)}{x^r}
< \frac{M_{\mu,r}(\theta)}{R_c^r} = 1-c.
$$
Hence $F_\theta(x)>c$ for every $x>R_c$. Therefore, if $F_\theta(x)\le c$, then necessarily $x\le R_c$.

Combining the two observations, we obtain
$$
\left\{x\in\mathbb R:
a<F_\theta(x)\le c\right\}
\subset(-R_a,R_c].
$$
Consequently,
$$\left|\left\{
x\in\mathbb R:a<F_\theta(x)\le c
\right\}\right|
\le|(-R_a,R_c]|=R_a+R_c.
$$
Finally,
$$R_a+R_c
=M_{\mu,r}(\theta)^{1/r}
\left(a^{-1/r}+(1-c)^{-1/r}
\right)\le
2M_{\mu,r}(\theta)^{1/r}\max\left\{
a^{-1/r},(1-c)^{-1/r}
\right\}.
$$
This proves the lemma.
\end{proof}
We are now ready to prove Proposition~\ref{prop:peelingtool}.
\begin{proof}[Proof of Proposition~\ref{prop:peelingtool}]
The case $\alpha=0$ is trivial. Hence, throughout the proof, we assume that $0<\alpha\le 1$.

We first consider the degenerate case $M_{\mu,r}(\theta)=0$. In this case,
$$
|\langle X,\theta\rangle|=0
\quad\text{almost surely}.
$$
Hence, almost surely
$$
F_\theta(x)=F_{n,\theta}(x) =
\begin{cases}
0, & x<0,\\
1, & x\ge 0.
\end{cases}
$$
Consequently, for every $u\in(0,1)$,
$$
F_{n,\theta}^{-1}(u)=F_\theta^{-1}(u)=0.
$$
Since the right-hand side of the desired estimate is also equal to zero when $M_{\mu,r}(\theta)=0$, the result follows in this case.

We now assume that $M_{\mu,r}(\theta)\in(0,+\infty)$. By Lemma~\ref{lem:Quantile-CDF}, we have
$$ \int_0^\alpha \left| F_{n,\theta}^{-1}(u)-F_\theta^{-1}(u)\right|du=\int_{\mathbb R}\left|
\min\{F_{n,\theta}(x),\alpha\}-\min\{F_\theta(x),\alpha\}
\right|dx.
$$
We split the proof into two regimes, according to whether $\alpha<1/4$ or $\alpha\ge1/4$.

Assume first that $0<\alpha<1/4$. We decompose $\mathbb R$ into the following three regions.

\textbf{Case 1.} Let
$$
J_1=\left\{x\in\mathbb R:F_\theta(x)\le 2\alpha\right\}.
$$
For every $x\in J_1$, we have
$$
\left|
\min\{F_{n,\theta}(x),\alpha\}-\min\{F_\theta(x),\alpha\}
\right|\le\left|F_{n,\theta}(x)-F_\theta(x)\right|.
$$
Therefore,
\begin{align*}
\mathbb E\left|\min\{F_{n,\theta}(x),\alpha\}
-\min\{F_\theta(x),\alpha\}\right|
&\le\mathbb E\left|F_{n,\theta}(x)-F_\theta(x)\right| \\
&\le\sqrt{\frac{F_\theta(x)(1-F_\theta(x))}{n}} \\
&\le\sqrt{\frac{F_\theta(x)}{n}}.
\end{align*}

For $j\ge0$, define
$$
D_j
\coloneqq\left\{x\in\mathbb R:2^{-j}\alpha<F_\theta(x)\le 2^{-j+1}\alpha\right\}.
$$
The sets $(D_j)_{j\ge0}$ cover the region $\{x\in\mathbb R:0<F_\theta(x)\le 2\alpha\}$, while the set where $F_\theta(x)=0$ does not contribute to the integral. Hence,
\allowdisplaybreaks
\begin{align*}
\int_{J_1}\mathbb E\left|\min\{F_{n,\theta}(x),\alpha\}
-\min\{F_\theta(x),\alpha\}\right|dx 
&\le\int_{J_1}\sqrt{\frac{F_\theta(x)}{n}}dx\\
&\le\sum_{j\ge0}\int_{D_j}\sqrt{\frac{F_\theta(x)}{n}}
dx\\
&\le\sum_{j\ge0}\sqrt{\frac{2^{-j+1}\alpha}{n}}|D_j| \\
&\le\sum_{j\ge0}\sqrt{\frac{2^{-j+1}\alpha}{n}}2M_{\mu,r}(\theta)^{1/r}\left(\min\{2^{-j}\alpha,1-2^{-j+1}\alpha\}
\right)^{-1/r} \\
&\le\sum_{j\ge0}\sqrt{\frac{2^{-j+1}\alpha}{n}}2M_{\mu,r}(\theta)^{1/r}(2^{-j}\alpha)^{-1/r} \\
&\lesssim_{r}\frac{M_{\mu,r}(\theta)^{1/r}}{\sqrt n}
\alpha^{1/2-1/r}\sum_{j\ge0}2^{-j(1/2-1/r)} \\
&\lesssim_{r}\frac{M_{\mu,r}(\theta)^{1/r}}{\sqrt n}\alpha^{1/2-1/r}.
\end{align*}
Here we used Lemma~\ref{lem:cdf-level-interval}. Moreover, since $\alpha<1/4$, we have
$$
2^{-j}\alpha<\frac14
\qquad\text{and}\qquad
1-2^{-j+1}\alpha>\frac12,
$$
so that
$$
\min\{2^{-j}\alpha,1-2^{-j+1}\alpha\}=2^{-j}\alpha.
$$
Besides, the last inequality follows from the summability of
$\sum_{j\ge0}2^{-j(1/2-1/r)}$, since $r>2$.

\textbf{Case 2.} Let
$$
J_2=\left\{x\in\mathbb R:2\alpha<F_\theta(x)\le\frac34\right\}.
$$
For $x\in J_2$, we have $F_\theta(x)>\alpha$, and hence
$$
\min\{F_\theta(x),\alpha\}=\alpha.
$$
Therefore,
\begin{align*}
\left|\min\{F_{n,\theta}(x),\alpha\}-
\min\{F_\theta(x),\alpha\}\right|
&=\left|\min\{F_{n,\theta}(x),\alpha\}-\alpha\right| \\
&=\left(\alpha-F_{n,\theta}(x)\right)_+ \\
&\le\alpha\mathbf 1_{\{F_{n,\theta}(x)\le \alpha\}}.
\end{align*}
Since $F_\theta(x)>2\alpha$, we obtain
\begin{align*}
\mathbb E\left|\min\{F_{n,\theta}(x),\alpha\}-
\min\{F_\theta(x),\alpha\}\right|
&\le\alpha\mathbb P\left(F_{n,\theta}(x)\le\alpha\right) \\
&\le\alpha\mathbb P\left(F_{n,\theta}(x)\le\frac{F_\theta(x)}{2}\right) \\
&\le\alpha\exp\left(-\frac n8 F_\theta(x)\right) \\
&\le\alpha\exp\left(-\frac{n\alpha}{4}\right).
\end{align*}
The third inequality follows from a Chernoff bound for Bernoulli sums; see, for example, the proof of \citet[Corollary~2.3.4]{Vershynin_2026}.

Combining this estimate with Lemma~\ref{lem:cdf-level-interval}, we get
\begin{align*}
\int_{J_2}\mathbb{E} \left| \min\{F_{n,\theta}(x),\alpha\}- \min\{F_\theta(x),\alpha\}
\right|dx &
\le 2M_{\mu,r}(\theta)^{1/r} \alpha\exp\left(-\frac{n\alpha}{4}\right)\left(\min\left\{2\alpha,\frac14\right\}\right)^{-1/r} \\
&\le C_rM_{\mu,r}(\theta)^{1/r} \alpha^{1-1/r} \exp\left(-\frac{n\alpha}{4}\right) \\
&=C_rM_{\mu,r}(\theta)^{1/r}\frac{\alpha^{1/2-1/r}}{\sqrt n}\sqrt{n\alpha}\exp\left(-\frac{n\alpha}{4}\right) \\
&\lesssim_{r}\frac{M_{\mu,r}(\theta)^{1/r}}{\sqrt n} \alpha^{1/2-1/r}.
\end{align*}
In the second inequality, we used $\alpha<1/4$, which implies $\min\left\{2\alpha,\frac14\right\}\ge \alpha.$
The last inequality follows from the fact that $
\sup_{y\ge0}\sqrt y e^{-y/4}<+\infty.$

\textbf{Case 3.} Let
$$
J_3=\left\{x\in\mathbb{R}:\frac34<F_\theta(x)\le1\right\}.
$$
As in Case 2, for $x\in J_3$, we have
$$
\min\{F_\theta(x),\alpha\}=\alpha
$$
and therefore
$$
\left| \min\{F_{n,\theta}(x),\alpha\}-\min\{F_\theta(x),\alpha\}\right|\le\alpha\mathbf 1_{\{F_{n,\theta}(x)\le\alpha\}}.
$$
Since $F_\theta(x)>3/4$ and $\alpha<1/4$, we have
$$F_\theta(x)-\alpha>\frac12.
$$
By Chebyshev's inequality,
\begin{align*}
\mathbb P(F_{n,\theta}(x)\le\alpha)
&\le \mathbb P\left( |F_{n,\theta}(x)-F_\theta(x)|
\ge F_\theta(x)-\alpha
\right) \\
&\le\mathbb P\left(|F_{n,\theta}(x)-F_\theta(x)|\ge
\frac{1}{2}\right) \\
&\le 4\operatorname{Var}(F_{n,\theta}(x)) \\
&=\frac{4F_\theta(x)(1-F_\theta(x))}{n} \\
&\le\frac{4(1-F_\theta(x))}{n}.
\end{align*}
Thus,
\begin{align*}
\int_{J_3}\mathbb E\left|\min\{F_{n,\theta}(x),\alpha\}-\min\{F_\theta(x),\alpha\}\right|dx 
&\le\frac{4\alpha}{n}\int_{J_3}\left(1-F_\theta(x)\right)dx.
\end{align*}

For $j\ge0$, define
$$
E_j \coloneqq \left\{ x\in\mathbb R: 1-2^{-j-2}<F_\theta(x)\le 1-2^{-j-3} \right\}.
$$
Then
\begin{align*}
\int_{J_3}\left(1-F_\theta(x)\right)dx
\le \sum_{j\ge0}\int_{E_j}\left(1-F_\theta(x)\right)dx
\le\sum_{j\ge0}2^{-j-2}|E_j|.
\end{align*}
By Lemma~\ref{lem:cdf-level-interval},
$$ |E_j| \le 2M_{\mu,r}(\theta)^{1/r} \left(\min\{1-2^{-j-2},2^{-j-3}\}\right)^{-1/r}.$$
Since $\min\{1-2^{-j-2},2^{-j-3}\}=2^{-j-3}$,
we get
$$
|E_j|\le C_rM_{\mu,r}(\theta)^{1/r}2^{j/r}.
$$
Therefore,
$$
\int_{J_3}\left(1-F_\theta(x)\right)dx
\le C_rM_{\mu,r}(\theta)^{1/r} \sum_{j\ge0}2^{-j}2^{j/r}\le C_rM_{\mu,r}(\theta)^{1/r},
$$
where the last inequality follows from $1-1/r>0$. Consequently,
\begin{align*}
\int_{J_3}\mathbb E \left| \min\{F_{n,\theta}(x),\alpha\}
-\min\{F_\theta(x),\alpha\} \right|dx
&\lesssim_{r}\frac{\alpha}{n}M_{\mu,r}(\theta)^{1/r} \\
&\lesssim_{r} \frac{M_{\mu,r}(\theta)^{1/r}}{\sqrt n} \alpha^{1/2-1/r}.
\end{align*}
In the last inequality, we used $n^{-1}\le n^{-1/2}$ and, since $r>2$ and $0<\alpha<1$, we have
$$
\alpha\le \alpha^{1/2-1/r}.
$$

It remains to consider the case $\alpha\ge1/4$. In this case, we use the direct bound
$$ \mathbb E\int_0^\alpha \left| F_{n,\theta}^{-1}(u)-F_\theta^{-1}(u)\right|du \le \mathbb E\int_0^1
\left|F_{n,\theta}^{-1}(u)-F_\theta^{-1}(u)
\right|du=\mathbb E \left[
\operatorname W_1 \left( \Pi^\theta_{\#}\widehat\mu_n,
\Pi^\theta_{\#}\mu
\right)
\right].
$$
Therefore, by Lemma~\ref{lem:projected-empirical-rate} applied with $q=1$, 
$$
\mathbb E\int_0^\alpha \left| F_{n,\theta}^{-1}(u)-F_\theta^{-1}(u) \right|du\le C_rM_{\mu,r}(\theta)^{1/r}n^{-1/2}.
$$
Since $\alpha\ge 1/4$ and $1/2-1/r>0$, we have
$$ \alpha^{1/2-1/r} \ge \left(\frac14\right)^{1/2-1/r}.
$$
Equivalently,
$$
1 \le 4^{1/2-1/r}\alpha^{1/2-1/r}.
$$
Thus,
$$
M_{\mu,r}(\theta)^{1/r}n^{-1/2} \le 4^{1/2-1/r} \frac{M_{\mu,r}(\theta)^{1/r}}{\sqrt n} \alpha^{1/2-1/r}.
$$
Absorbing the factor $4^{1/2-1/r}$ into the constant $C_r$, we obtain
$$ \mathbb E\int_0^\alpha \left| F_{n,\theta}^{-1}(u)-F_\theta^{-1}(u) \right|du
\lesssim_r \frac{M_{\mu,r}(\theta)^{1/r}}{\sqrt n} \alpha^{1/2-1/r}.
$$
This completes the proof.
\end{proof}

\section{Omitted Proofs from Sections~\ref{sec:upperboundr>2p}--\ref{sec:upperboundr=2p}}
\subsection{The regime $r>2p$}\label{appendix:upperboundr>2p}
\begin{proof}[Proof of Theorem~\ref{thm:minimax-noncritical} in the regime $r>2p$]
The case $p=1$ has already been established in the discussion of Lemma~\ref{lem:projected-empirical-rate}. We therefore assume that $p>1$.
By Lemma~\ref{lem:decomposition}, for each fixed $\theta\in \mathcal{G}$, the one-dimensional error is bounded by
$C_p\left(\mathsf{A}_{n,\theta}+\mathsf{B}_{m,\theta}+\mathsf{C}_{n,\theta}+\mathsf{D}_{m,\theta}\right).
$
We first control the terms $\mathsf{C}_{n,\theta}$ and $\mathsf{D}_{m,\theta}$. By Lemma~\ref{lem:projected-empirical-rate}, since $r>2p$, we have
$$\mathsf{C}_{n,\theta}\lesssim_{p,r} \frac{M_{\mu,r}(\theta)^{p/r}}{\sqrt n},
\qquad
\mathsf{D}_{m,\theta}\lesssim_{p,r} \frac{M_{\nu,r}(\theta)^{p/r}}{\sqrt m}.
$$
Therefore, using $\mu,\nu\in\mathcal K_{p,r}(b)$ and $n\le m$,
$$
\int_{\mathcal{G}} \mathsf{C}_{n,\theta}d\sigma(\theta) \lesssim_{p,r} \frac{b}{\sqrt n},
\qquad
\int_\mathcal{G}\mathsf{D}_{m,\theta}d\sigma(\theta)\lesssim_{p,r} \frac{b}{\sqrt n}.
$$

It remains to control the weighted terms $\mathsf{A}_{n,\theta}$ and $\mathsf{B}_{m,\theta}$. By Lemma~\ref{lem:weighted-quantile}, we have
$$\mathsf{A}_{n,\theta}\lesssim_{p,r} \frac{M_{\mu,r}(\theta)^{p/r}
+M_{\mu,r}(\theta)^{1/r}M_{\nu,r}(\theta)^{(p-1)/r}
}{\sqrt n}.$$
Similarly,
$$\mathsf{B}_{m,\theta}\lesssim_{p,r}\frac{M_{\nu,r}(\theta)^{p/r}+M_{\nu,r}(\theta)^{1/r}M_{\mu,r}(\theta)^{(p-1)/r}}{\sqrt m}.
$$

We now integrate these estimates over $\mathcal{G}$. First,
$$\int_\mathcal{G} M_{\mu,r}(\theta)^{p/r}d\sigma(\theta)\le b,
\quad\int_\mathcal{G} M_{\nu,r}(\theta)^{p/r}d\sigma(\theta)\le b.
$$
Moreover, by Hölder's inequality,
\begin{align*}
\int_\mathcal{G}
M_{\mu,r}(\theta)^{1/r}M_{\nu,r}(\theta)^{(p-1)/r}
d\sigma(\theta)&\le\left(\int_\mathcal{G} M_{\mu,r}(\theta)^{p/r}\,d\sigma(\theta)\right)^{1/p}\left(
\int_\mathcal{G} M_{\nu,r}(\theta)^{p/r}\,d\sigma(\theta)
\right)^{(p-1)/p}\\
&\le b.
\end{align*}

Similarly, we have
$$
\int_{\mathcal{G}} M_{\nu,r}(\theta)^{1/r}M_{\mu,r}(\theta)^{(p-1)/r}d\sigma(\theta)\le b.
$$
Consequently, using again $n\le m$,
$$
\int_\mathcal{G}\mathsf{A}_{n,\theta}\,d\sigma(\theta)\lesssim_{p,r}\frac{b}{\sqrt n},
\qquad
\int_\mathcal{G}\mathsf{B}_{m,\theta}\,d\sigma(\theta)\lesssim_{p,r}\frac{b}{\sqrt n}.
$$

Combining the estimates for $\mathsf{A}_{n,\theta}$, $\mathsf{B}_{m,\theta}$, $\mathsf{C}_{n,\theta}$, and $\mathsf{D}_{m,\theta}$, we obtain
$$
\mathbb E\left|
\operatorname{SW}_p^p(\widehat\mu_n,\widehat\nu_m)
-\operatorname{SW}_p^p(\mu,\nu)
\right|\lesssim_{p,r} b n^{-1/2}.
$$
We complete the proof by taking the supremum over $\mu,\nu \in \mathcal{K}_{p,r}(b)$.
\end{proof}
\subsection{The regime $p<r<2p$}
\subsubsection{Proof of Lemma~\ref{lem:truncated-weighted-quantile}}\label{appendix:prooflem:truncated-weighted-quantile}
\begin{proof}
We prove the first estimate. The second estimate follows by taking the weight distribution to be $\mu$ instead of $\nu$, and the estimates involving $G_{m,\theta}^{-1}$ are obtained by exchanging the roles of $(\mu,n,F^{-1}_\theta,F^{-1}_{n,\theta})$ and $(\nu,m,G^{-1}_\theta,G^{-1}_{m,\theta})$.

If $M_{\nu,r}(\theta)=0$, then $G_\theta^{-1}(u)=0$ for every $u\in(0,1)$. Consequently, the desired estimate follows immediately. Hence, we assume that $M_{\nu,r}(\theta)\in(0,+\infty).$
Set $K\coloneqq M_{\nu,r}(\theta)^{1/r}>0.$ Let $J_n\ge1$ be an integer to be chosen later. We decompose $(0,1)$ according to the size of $|G_\theta^{-1}(u)|$. Define
$$E^*\coloneqq\left\{u\in(0,1): |G_\theta^{-1}(u)|<K\right\},$$ for $0\le k\le J_n-1$,
$$E_k \coloneqq \left\{ u\in(0,1): 2^kK\le |G_\theta^{-1}(u)|<2^{k+1}K \right\},$$
and $$E_\infty\coloneqq \left\{u\in(0,1):|G_\theta^{-1}(u)|\ge 2^{J_n}K
\right\}.$$
Then $$(0,1)\subset E^*\cup\bigcup_{k=0}^{J_n-1}E_k\cup E_\infty.
$$

\textbf{Step 1: Control on the region $E^*$.} 
On $E^*$, we have $$|G_\theta^{-1}(u)|^{p-1} \le K^{p-1}= M_{\nu,r}(\theta)^{(p-1)/r}.$$
Therefore, by Lemma~\ref{lem:projected-empirical-rate} (with $q=1$ and $s=r$), we obtain
\begin{align*}
\mathbb E\int_{E^*}\left|F_{n,\theta}^{-1}(u)-F_\theta^{-1}(u)\right|\left|G_\theta^{-1}(u)
\right|^{p-1}du &\le
M_{\nu,r}(\theta)^{(p-1)/r}
\mathbb E
\int_0^1\left|
F_{n,\theta}^{-1}(u)-F_\theta^{-1}(u)\right|
du\\
& \le C_r M_{\mu,r}(\theta)^{1/r} M_{\nu,r}(\theta)^{(p-1)/r}\rho_n,
\end{align*}
where
$$
\rho_n=
\begin{cases}
n^{-(r-1)/r}, & 1<r<2,\\
n^{-1/2}\log(n+1), &r=2,\\
n^{-1/2}, & r>2.
\end{cases}
$$
It is worth noting that $\rho_n$ is always bounded by the desired rate in the regime $p<r<2p$. Indeed:
\begin{itemize}
\item If $1<r<2$, then, since $p>1$, we have
$
n^{-(r-1)/r}\le n^{-(r-p)/r}.
$
\item If $r>2$, then, since $r<2p$, we have
$
n^{-1/2}\le n^{-(r-p)/r}.
$
\item If $r=2$, then the condition $p<r<2p$ implies $1<p<2$. In this case, for all $n\ge1$,
$$
\log(n+1)\le \frac{2}{p-1}n^{(p-1)/2}.
$$
Therefore,
$$
n^{-1/2}\log(n+1)
\le
\frac{2}{p-1}n^{-(2-p)/2}.
$$
\end{itemize}
Consequently, in all cases under the condition $p<r<2p$, we have
$\rho_n\le C_{p,r}n^{-(r-p)/r}.
$
Hence
$$\mathbb E \int_{E^*} \left|
F_{n,\theta}^{-1}(u)-F_\theta^{-1}(u) \right|
\left| G_\theta^{-1}(u)\right|^{p-1}du
\le C_{p,r} M_{\mu,r}(\theta)^{1/r} M_{\nu,r}(\theta)^{(p-1)/r} n^{-(r-p)/r}.$$
\textbf{Step 2: Control on the dyadic regions $E_k$, $0\le k\le J_n-1$.}
On $E_k$, we have
$|G_\theta^{-1}(u)|^{p-1}\le(2^{k+1}K)^{p-1}.\\
$
Define $$A_k \coloneqq \left\{ u\in(0,1): |G_\theta^{-1}(u)|\ge 2^kK \right\}.$$

We split the estimate of the shell contribution into three cases.

First, assume that $r>2$. Using Proposition~\ref{prop:peelingtool} and following the same process in Lemma~\ref{lem:weighted-quantile}, we obtain
$$\mathbb E \int_{E_k} \left|F_{n,\theta}^{-1}(u)-F_\theta^{-1}(u) \right|\left| G_\theta^{-1}(u)\right|^{p-1} du \lesssim_{p,r}M_{\mu,r}(\theta)^{1/r} M_{\nu,r}(\theta)^{(p-1)/r}n^{-1/2} 2^{k(p-r/2)}.$$
As a result, summing over $0\le k\le J_n-1$, and using $p-r/2>0$, we get
$$\sum_{k=0}^{J_n-1} \mathbb E\int_{E_k}\left|F_{n,\theta}^{-1}(u)-F_\theta^{-1}(u)\right|\left|G_\theta^{-1}(u)\right|^{p-1} du \lesssim_{p,r} M_{\mu,r}(\theta)^{1/r} M_{\nu,r}(\theta)^{(p-1)/r}
n^{-1/2}2^{J_n(p-r/2)}.$$

Second, assume that $1 <r < 2$. In this case, we use Lemma~\ref{lem:projected-empirical-rate}. To be more specific, we have
\begin{align*}
\mathbb E \int_{E_k} \left|
F_{n,\theta}^{-1}(u)-F_\theta^{-1}(u)
\right| \left| G_\theta^{-1}(u)
\right|^{p-1} du &\le 2^{(k+1)(p-1)}K^{p-1}\mathbb{E}\int_{E_k}\left|
F_{n,\theta}^{-1}(u)-F_\theta^{-1}(u)
\right| du\\
& \le 2^{(k+1)(p-1)}M_{\nu,r}(\theta)^{(p-1)/r}\mathbb{E}\left(\operatorname{W}_1(\Pi^{\theta}_{\#}\widehat{\mu}_n,\Pi^{\theta}_{\#}\mu)\right)\\
& \lesssim_{p,r}2^{k(p-1)} M_{\nu,r}(\theta)^{(p-1)/r}M_{\mu,r}(\theta)^{1/r}n^{-(r-1)/r}
\end{align*}
Summing over $0\le k \le J_n-1$, and using $p>1$, we obtain
$$\sum_{k=0}^{J_n-1} \mathbb E\int_{E_k}\left|F_{n,\theta}^{-1}(u)-F_\theta^{-1}(u)\right|\left|G_\theta^{-1}(u)
\right|^{p-1} du \lesssim_{p,r} M_{\mu,r}(\theta)^{1/r} M_{\nu,r}(\theta)^{(p-1)/r}n^{-(r-1)/r}2^{J_n(p-1)}.$$

Third, assume that $r=2$. Since $p<r<2p$, we have $1<p<2$. Therefore, we can choose $q_0=q_0(p)$ such that $1<q_0<2/(3-p).$

On $E_k$, we have
$$|G_\theta^{-1}(u)|^{p-1}\le(2^{k+1}K)^{p-1}.$$
Denote
$$A_k=\left\{u\in(0,1):|G_\theta^{-1}(u)|\ge 2^kK\right\}.
$$
Then $E_k\subset A_k$. Moreover, for the Lebesgue measure $\lambda$ on $(0,1)$, we also have
$$\lambda(A_k)=\mathbb P_{U\sim\operatorname{Unif}(0,1)}(U\in A_k)=
\mathbb P_{U\sim\operatorname{Unif}(0,1)}\left(\left|G_\theta^{-1}(U)\right|\ge 2^kK\right)=\mathbb P_{Y\sim\nu}
\left(|\langle Y,\theta\rangle|\ge 2^kK
\right).
$$
In the last equality, we used the fact that if $U\sim\operatorname{Unif}(0,1)$, then $G_\theta^{-1}(U)$ has distribution $\Pi^\theta_{\#}\nu$. \\
Then, using Markov's inequality, we have
\begin{equation*}
\lambda(A_k) \le \frac{M_{\nu,2}(\theta)}{(2^{k}K)^2}=2^{-2k}.
\end{equation*}

By Hölder's inequality with exponent $q_0$, we have
$$
\begin{aligned}
\int_{E_k}\left|F_{n,\theta}^{-1}(u)-F_\theta^{-1}(u)\right|du
&\le\int_{A_k}\left|F_{n,\theta}^{-1}(u)-F_\theta^{-1}(u)\right|du \\
&\le\lambda(A_k)^{1-1/q_0}\left(\int_0^1
\left|F_{n,\theta}^{-1}(u)-F_\theta^{-1}(u)\right|^{q_0}du
\right)^{1/q_0}.
\end{aligned}
$$
Taking expectation and using Jensen's inequality, we obtain
$$\begin{aligned}\mathbb E\int_{E_k}\left|F_{n,\theta}^{-1}(u)-F_\theta^{-1}(u)\right|du
&\le\lambda(A_k)^{1-1/q_0}\left[\mathbb E\int_0^1
\left|F_{n,\theta}^{-1}(u)-F_\theta^{-1}(u)\right|^{q_0}du
\right]^{1/q_0} \\
&=\lambda(A_k)^{1-1/q_0}\left[\mathbb E\operatorname{W}_{q_0}^{q_0}
\left(\Pi^\theta_{\#}\widehat{\mu}_n,\Pi^\theta_{\#}\mu
\right)\right]^{1/q_0}.
\end{aligned}
$$
Since $q_0>1$ and $r=2$, we have $q_0<2<2q_0.$ Therefore, by Lemma~\ref{lem:projected-empirical-rate} applied with $q=q_0$ and $s=2$,
$$\mathbb E\operatorname{W}_{q_0}^{q_0} \left( \Pi^\theta_{\#}\widehat{\mu}_n,\Pi^\theta_{\#}\mu \right) 
\lesssim_{q_0}M_{\mu,2}(\theta)^{q_0/2}n^{-(2-q_0)/2}.
$$
Combining this estimate with $\lambda(A_k)\le 2^{-2k}$ gives
$$\mathbb E \int_{E_k} \left|F_{n,\theta}^{-1}(u)-F_\theta^{-1}(u)\right|du\lesssim_{q_0} M_{\mu,2}(\theta)^{1/2} 2^{-2k(1-1/q_0)}n^{-(2-q_0)/(2q_0)}.
$$
Consequently,
$$\begin{aligned}
&\mathbb E\int_{E_k}\left|F_{n,\theta}^{-1}(u)-F_\theta^{-1}(u)\right|
\left|G_\theta^{-1}(u)\right|^{p-1}du \\
&\quad\lesssim_{p,q_0}M_{\mu,2}(\theta)^{1/2}M_{\nu,2}(\theta)^{(p-1)/2}n^{-(2-q_0)/(2q_0)}2^{k(p-3+2/q_0)}.
\end{aligned}
$$
By our choice of $q_0$, we have $p-3+\frac{2}{q_0}>0.$ Hence,
$$ \sum_{k=0}^{J_n-1} 2^{k(p-3+2/q_0)}\lesssim_{p,q_0} 2^{J_n(p-3+2/q_0)}.
$$
Therefore,
$$\begin{aligned}
&\sum_{k=0}^{J_n-1}\mathbb E\int_{E_k}\left|F_{n,\theta}^{-1}(u)-F_\theta^{-1}(u)\right|\left|G_\theta^{-1}(u)\right|^{p-1}du \\
&\quad\lesssim_{p,q_0}M_{\mu,2}(\theta)^{1/2}M_{\nu,2}(\theta)^{(p-1)/2}n^{-(2-q_0)/(2q_0)}2^{J_n(p-3+2/q_0)}.
\end{aligned}
$$
% Choose $K_n$ so that
% $$
% 2^{K_n}\le 2n^{1/2}.
% $$
% Then
% $$
% \begin{aligned}
% n^{-(2-q_0)/(2q_0)} 2^{K_n(p-3+2/q_0)}
% &\le
% C_{p,q_0} n^{-(2-q_0)/(2q_0)}
% n^{(p-3+2/q_0)/2} \\
% &= C_{p,q_0} n^{-(2-p)/2}.
% \end{aligned}
% $$
% Hence,
% $$
% \sum_{k=0}^{K_n-1} \mathbb E \int_{E_k}
% \left|F_{n,\theta}^{-1}(u)-F_\theta^{-1}(u)\right| \left|G_\theta^{-1}(u)\right|^{p-1}du \le C_{p,q_0} M_{\mu,2}(\theta)^{1/2}
% M_{\nu,2}(\theta)^{(p-1)/2} n^{-(2-p)/2}.
% $$
% Since $r=2$, this is exactly
% $$
% C_{p,r} M_{\mu,r}(\theta)^{1/r} M_{\nu,r}(\theta)^{(p-1)/r} n^{-(r-p)/r}.
% $$
\textbf{Step 3: Control on the far-tail region $E_\infty$.}
Let $\lambda$ denote the Lebesgue measure on $(0,1)$. By the same quantile representation and Markov argument used in Step 2, with threshold $2^{J_n}K$, we obtain
% $$\lambda(E_{\infty})=
% \mathbb P_{U\sim\operatorname{Unif}(0,1)}(U\in E_{\infty})=
% \mathbb P_{U\sim\operatorname{Unif}(0,1)}\left(\left|G_\theta^{-1}(U)\right|\ge 2^{K_n}K\right)=\mathbb P_{Y\sim\nu}
% \left(|\langle Y,\theta\rangle|\ge 2^{K_n}K
% \right),
% $$
% where $\lambda$ denotes the Lebesgue measure on $(0,1)$. In the last equality, we used the fact that if $U\sim\operatorname{Unif}(0,1)$, then $G_\theta^{-1}(U)$ has distribution $\Pi^\theta_{\#}\nu$. 
% Then, using Markov's inequality, we have
\begin{equation*}
\lambda(E_{\infty}) \le \frac{M_{\nu,r}(\theta)}{(2^{J_n}K)^r}=2^{-J_nr}.
\end{equation*}
Using the triangle inequality, we have
\begin{equation*}
\begin{aligned}
&\mathbb{E}\int_{E_\infty} \left|F^{-1}_{n,\theta}(u)-F^{-1}_{\theta}(u)\right|
\left|G^{-1}_{\theta}(u)\right|^{p-1}d u \\
&\quad \le 
\underbrace{\int_{E_{\infty}}\left|F^{-1}_{\theta}(u)\right|
\left|G^{-1}_{\theta}(u)\right|^{p-1}du}_{I_1} +\underbrace{\mathbb{E}\int_{E_{\infty}}\left|F^{-1}_{n,\theta}(u)\right|\left|G^{-1}_{\theta}(u)\right|^{p-1} du }_{I_2}.
\end{aligned}
\end{equation*}

\textbf{Step 3.1:  Control $I_1$.} By Hölder's inequality, 
\begin{equation*}
\int_{E_{\infty}}\left|F^{-1}_{\theta}(u)\right|\left|G^{-1}_{\theta}(u)\right|^{p-1}du \le \left(\int_{E_{\infty}}\left|F^{-1}_{\theta}(u)\right|^pdu\right)^{1/p}\left(\int_{E_{\infty}}\left|G^{-1}_{\theta}(u)\right|^pdu\right)^{(p-1)/p}.
\end{equation*}
On the other hand, by Hölder's inequality, we also have
\begin{equation*}
\int_{E_{\infty}}\left|F^{-1}_{\theta}(u)\right|^pdu =\int_{0}^{1}\left|F^{-1}_{\theta}(u)\right|^p\mathbbm{1}_{E_{\infty}}du\le \left(\int_{0}^{1}\left|F^{-1}_{\theta}(u)\right|^rdu\right)^{p/r}\left(\lambda(E_{\infty})\right)^{(r-p)/r}
\end{equation*}
Using the fact that $\lambda(E_{\infty})\le 2^{-J_nr}$ and also the definition of $M_{\mu,r}(\theta)$, we obtain
\begin{equation*}
\int_{E_{\infty}}\left|F^{-1}_{\theta}(u)\right|^pdu  \le M_{\mu,r}(\theta)^{p/r}2^{-J_n(r-p)}.
\end{equation*}
Moreover, since $\left|G^{-1}_{\theta}(u)\right|\ge 2^{J_n}K$ on $E_{\infty}$ and $p<r$, we obtain
\begin{equation*}
\int_{E_\infty}\left|G^{-1}_{\theta}(u)\right|^pdu \le \left(2^{J_n}K\right)^{p-r}\int_{0}^{1}\left|G^{-1}_{\theta}(u)\right|^rdu=\left(2^{J_n}K\right)^{p-r}M_{\nu,r}(\theta)=M_{\nu,r}(\theta)^{p/r}2^{-J_n(r-p)}.
\end{equation*}
Therefore, 
\begin{equation*}
\int_{E_\infty}\left|F^{-1}_{\theta}(u)\right|\left|G^{-1}_{\theta}(u)\right|^{p-1}du \le M_{\mu,r}(\theta)^{1/r}M_{\nu,r}(\theta)^{(p-1)/r}2^{-J_n(r-p)}.
\end{equation*}
\textbf{Step 3.2: Control $I_2$.} By Hölder's inequality, 
\begin{equation*}
\int_{E_{\infty}} \left|F^{-1}_{n,\theta}(u)\right|\left|G^{-1}_{\theta}(u)\right|^{p-1}du \le \left(\int_{E_\infty}\left|F^{-1}_{n,\theta}(u)\right|^pdu\right)^{1/p}\left(\int_{E_\infty}\left|G^{-1}_{\theta}(u)\right|^pdu\right)^{(p-1)/p}.
\end{equation*}
Taking expectation and using Jensen's inequality, we have
\begin{equation*}
\mathbb{E}\int_{E_{\infty}} \left|F^{-1}_{n,\theta}(u)\right|\left|G^{-1}_{\theta}(u)\right|^{p-1}du \le \left(\mathbb{E}\int_{E_\infty}\left|F^{-1}_{n,\theta}(u)\right|^pdu\right)^{1/p}\left(\int_{E_\infty}\left|G^{-1}_{\theta}(u)\right|^pdu\right)^{(p-1)/p}.
\end{equation*}
Since $F^{-1}_{n,\theta}$ only takes values among the projected observations, we have
\begin{equation*}
\int_{E_{\infty}}\left|F^{-1}_{n,\theta}(u)\right|^pdu \le \lambda(E_{\infty})\max_{1 \le i \le n} \left|\langle X_i,\theta\rangle\right|^p \le 2^{-J_nr}\max_{1 \le i \le n} \left|\langle X_i,\theta\rangle\right|^p. 
\end{equation*}
Consequently, with $p<r$, we have
$$\mathbb{E}\int_{E_{\infty}}\left|F^{-1}_{n,\theta}(u)\right|^pdu  \le 2^{-J_nr}\mathbb{E} \max_{1 \le i \le n} \left|\langle X_i,\theta\rangle \right|^p \le 2^{-J_nr}\left(\mathbb{E} \max_{1 \le i \le n} \left|\langle X_i,\theta\rangle \right|^r\right)^{p/r}.$$
Using
$\max_{1\le i\le n}|\langle X_i,\theta\rangle|^r
\le\sum_{i=1}^n|\langle X_i,\theta\rangle|^r$, we deduce that
\begin{equation*}
\mathbb{E}\int_{E_{\infty}}\left|F^{-1}_{n,\theta}(u)\right|^pdu
\le M_{\mu,r}(\theta)^{p/r}n^{p/r}2^{-J_nr}.
\end{equation*}
% \begin{align*}
% \mathbb{E}\int_{E_{\infty}}\left|F^{-1}_{n,\theta}(u)\right|^pdu & \le 2^{-K_nr}\mathbb{E} \max_{1 \le i \le n} \left|\langle X_i,\theta\rangle \right|^p\\
% &\le 2^{-K_nr}\left(\mathbb{E} \max_{1 \le i \le n} \left|\langle X_i,\theta\rangle \right|^r\right)^{p/r}\\
% & \le 2^{-K_nr}\left(\sum_{i=1}^{n}\mathbb{E}\left|\langle X_i,\theta\rangle \right|^r\right)^{p/r}\\
% &= M_{\mu,r}(\theta)^{p/r}n^{p/r}2^{-K_nr}.
% \end{align*}
Combining this estimate with the previous bound on the $G^{-1}_{\theta}$ in \textbf{Step 3.1} gives
\begin{equation*}
\mathbb{E}\int_{E_\infty} \left|F^{-1}_{n,\theta}(u)\right|\left|G^{-1}_{\theta}(u)\right|^{p-1}du \lesssim_{p,r}M_{\mu,r}(\theta)^{1/r}M_{\nu,r}(\theta)^{(p-1)/r}n^{1/r}2^{-J_n(r-p+1)}.
\end{equation*}

Consequently, combining the estimates for $I_1$ and $I_2$, we obtain
\begin{equation*}
\begin{aligned}
&\mathbb{E}\int_{E_{\infty}}\left|F^{-1}_{n,\theta}(u)-F^{-1}_{\theta}(u)\right|\left|G^{-1}_{\theta}(u)\right|^{p-1}du \\
&\quad \lesssim_{p,r}M_{\mu,r}(\theta)^{1/r}
M_{\nu,r}(\theta)^{(p-1)/r}\left(2^{-J_n(r-p)}+n^{1/r}2^{-J_n(r-p+1)}
\right).
\end{aligned}
\end{equation*}
\textbf{Step 4: Choosing $J_n$.}
Combining the preceding steps, we obtain that
\begin{align*}
&\mathbb E\int_0^1\left|F_{n,\theta}^{-1}(u)-F_\theta^{-1}(u)\right|
\left|G_\theta^{-1}(u)\right|^{p-1}du \\
&\quad \lesssim_{p,r}M_{\mu,r}(\theta)^{1/r}
M_{\nu,r}(\theta)^{(p-1)/r}\left[n^{-(r-p)/r}+\mathcal{T}_{n,J_n}+2^{-J_n(r-p)}+n^{1/r}2^{-J_n(r-p+1)}\right],
\end{align*}
where
\begin{equation*}
\mathcal{T}_{n,J_n}=\begin{cases}n^{-(r-1)/r}2^{J_n(p-1)}, & 1 <r<2\\
n^{-(2-q_0)/(2q_0)}2^{J_n(p-3+2/q_0)}, & r=2\\
n^{-1/2}2^{J_n(p-r/2)},& r>2.
\end{cases}
\end{equation*}
We now choose $J_n=\max\left\{1,\left\lceil \frac{1}{r}\log_2 n\right\rceil\right\}.\\
$
With this choice, each term in the bracket is bounded by a constant multiple of $n^{-(r-p)/r}$. Therefore,
$$\mathbb E\int_0^1\left|F_{n,\theta}^{-1}(u)-F_\theta^{-1}(u)\right|
\left|G_\theta^{-1}(u)\right|^{p-1}du\lesssim_{p,r}M_{\mu,r}(\theta)^{1/r}M_{\nu,r}(\theta)^{(p-1)/r}n^{-(r-p)/r}.
$$
\end{proof}
\subsubsection{Proof of Theorem~\ref{thm:minimax-noncritical} in the regime $p<r<2p$}
\begin{proof}[Proof of Theorem~\ref{thm:minimax-noncritical} in the regime
$p<r<2p$]
The case $p=1$ follows directly from Lemma~\ref{lem:projected-empirical-rate}
and the reverse triangle inequality for $\operatorname{W}_1$. We therefore assume that $p>1$.

The proof then follows the same aggregation argument as in the regime
$r>2p$. Indeed, Lemma~\ref{lem:decomposition} reduces the directional error
to the four terms $\mathsf{A}_{n,\theta}$, $\mathsf{B}_{m,\theta}$, $\mathsf{C}_{n,\theta}$, and
$\mathsf{D}_{m,\theta}$. By Lemma~\ref{lem:projected-empirical-rate}, since
$p<r<2p$,
$$\mathsf{C}_{n,\theta}
\lesssim_{p,r}M_{\mu,r}(\theta)^{p/r}n^{-(r-p)/r},
\qquad \mathsf{D}_{m,\theta}
\lesssim_{p,r}M_{\nu,r}(\theta)^{p/r}m^{-(r-p)/r},
$$
while Lemma~\ref{lem:truncated-weighted-quantile} gives the same
sample-size rates for $\mathsf{A}_{n,\theta}$ and $\mathsf{B}_{m,\theta}$.

Integrating over $\mathcal{G}$ and applying Hölder's inequality exactly as in the
proof for the regime $r>2p$, together with
$\mu,\nu\in\mathcal K_{p,r}(b)$ and $n\le m$, yields
$$
\mathbb E\left|
\operatorname{SW}_p^p(\widehat\mu_n,\widehat\nu_m)-
\operatorname{SW}_p^p(\mu,\nu)
\right|\lesssim_{p,r}b\,n^{-(r-p)/r}.
$$
Taking the supremum over $\mu,\nu\in\mathcal K_{p,r}(b)$ completes the proof.
\end{proof}
\subsection{The regime $r=2p$}
\subsubsection{Proof of Lemma~\ref{lem:projected-empirical-ratev2}}\label{appendix:proofoflem:projected-empirical-ratev2}
\begin{proof}
Fix $\theta\in\mathbb S^{d-1}$ such that $M_{Q,2q}(\theta)<\infty$. If $\mathbb{E}_{X \sim Q}\left|\langle X,\theta\rangle \right|^{2q}=0$, then $\Pi^{\theta}_{\#}Q=\delta_{0}$ and $\Pi^{\theta}_{\#}\widehat{Q}_n=\delta_{0}$ almost surely. Hence the desired estimate is immediate. Thus, in the sequel, we assume that $\mathbb{E}_{X \sim Q}\left|\langle X,\theta\rangle \right|^{2q}>0$.

Applying \citet[Proposition~7.14]{bobkov2019one} then using Tonelli's theorem, we obtain
\begin{equation*}
\mathbb E\left[\operatorname{W}_q^q\left( \Pi^\theta_{\#}\widehat{Q}_n,\Pi^\theta_{\#}Q\right) \right] \le q2^{q-1}\int_{-\infty}^{+\infty} \left|x\right|^{q-1}\mathbb{E}\left|F_{n,\theta}(x)-F_{\theta}(x)\right|dx, 
\end{equation*}
where $F_{\theta}$ and $F_{n,\theta}$ denote the cumulative distribution functions of $\Pi^\theta_{\#}Q$ and $\Pi^\theta_{\#}\widehat{Q}_n$, respectively. We emphasize that this notation is used only within the present proof.

Following the proof of \citet[Corollary~7.18]{bobkov2019one}, we obtain
\begin{equation*}
\mathbb E\left[\operatorname{W}_q^q\left( \Pi^\theta_{\#}\widehat{Q}_n,\Pi^\theta_{\#}Q\right) \right]  \le q2^{q-1}\int_{-\infty}^{+\infty}\left|x\right|^{q-1}\min \left\{2F_{\theta}(x)(1-F_{\theta}(x)),\sqrt{\frac{F_{\theta}(x)(1-F_{\theta}(x))}{n}}\right\}dx.
\end{equation*}
Denote $T(u) \coloneqq \mathbb{P}_{X\sim Q}\left(\left|\langle X,\theta\rangle \right|\ge u\right)$. Then, for every $x\in\mathbb{R}$, we have
\begin{equation*}
F_{\theta}(x)(1-F_{\theta}(x)) \le \min \left\{F_{\theta}(x),1-F_{\theta}(x)\right\} \le T\left(\left|x\right|\right).
\end{equation*}
Moreover, the function $s \mapsto \min \left\{2s,\sqrt{\dfrac{s}{n}}\right\}$ is nondecreasing on $[0,\infty)$. Therefore,
\begin{align*}
\mathbb E\left[\operatorname{W}_q^q\left( \Pi^\theta_{\#}\widehat{Q}_n,\Pi^\theta_{\#}Q\right) \right] &\le q2^{q-1}\int_{-\infty}^{+\infty}\left|x\right|^{q-1}\min \left\{2T\left(\left|x\right|\right),\sqrt{\frac{T\left(\left|x\right|\right)}{n}}\right\}dx\\
&= q2^{q}\int_{0}^{+\infty}x^{q-1}\min \left\{2T\left(x\right),\sqrt{\frac{T\left(x\right)}{n}}\right\}dx.
\end{align*}
We now decompose the integral into three main parts: from $0$ to $R_1$, from $R_1$ to $R_2$, and from $R_2$ to $+\infty$, where the cut-off levels $R_1$ and $R_2$ will be chosen later.

For the first term, since $T(x)\le 1$ for all $x\ge 0$, we have
\begin{align*}
I_1\coloneqq q2^{q}\int_{0}^{R_1}x^{q-1}\min \left\{2T\left(x\right),\sqrt{\frac{T\left(x\right)}{n}}\right\}dx
\le \frac{2^{q} R_1^q}{\sqrt{n}}.
\end{align*}
Choosing $R_1\coloneqq \left(\mathbb{E}_{X \sim Q}\left|\langle X,\theta\rangle \right|^{2q}\right)^{1/(2q)}$, we obtain
\begin{equation*}
I_1 \le \frac{2^q\left(\mathbb{E}_{X \sim Q}\left|\langle X,\theta\rangle \right|^{2q}\right)^{1/2}}{\sqrt{n}} \lesssim_{q} \left(M_{Q,2q}(\theta)\right)^{1/2}\sqrt{\log(n+1)n^{-1}}.
\end{equation*}

For the second term, we have
\begin{equation*}
I_2 \coloneqq q2^{q}\int_{R_1}^{R_2}x^{q-1}\min\left\{2T(x),\sqrt{\frac{T(x)}{n}}\right\}dx \le \frac{q2^{q}}{\sqrt{n}}\int_{R_1}^{R_2}\sqrt{T(x)x^{2q}}\frac{dx}{x}.
\end{equation*}
Using Hölder's inequality, we have
\begin{equation*}
I_2 \le \frac{q2^{q}}{\sqrt{n}}\left(\int_{R_1}^{R_2}\frac{dx}{x}\right)^{1/2}\left(\int_{R_1}^{R_2}T(x)x^{2q}\frac{dx}{x}\right)^{1/2}\le\frac{q2^{q}}{\sqrt{n}}\sqrt{\log \frac{R_2}{R_1}}\left(\int_{0}^{+\infty}T(x)x^{2q-1}dx\right)^{1/2}.
\end{equation*}
Using the tail integral formula for moments, we have
\begin{equation*}
\displaystyle \int_{0}^{+\infty}T(x)x^{2q-1}dx=\dfrac{M_{Q,2q}(\theta)}{2q}.
\end{equation*}
Hence,
\begin{equation*}
I_2 \le \frac{q2^q}{\sqrt{n}}\sqrt{\log \frac{R_2}{R_1}}\sqrt{\frac{M_{Q,2q}(\theta)}{2q}}.
\end{equation*}
Choosing $R_2\coloneqq R_1(n+1)$, we finally obtain
\begin{equation*}
I_2 \lesssim_{q}\left(M_{Q,2q}(\theta)\right)^{1/2}\sqrt{\log(n+1)n^{-1}}.
\end{equation*}

For the last term, we have
\begin{equation*}
I_3 \coloneqq q2^{q}\int_{R_2}^{+\infty}x^{q-1}\min \left\{2T(x),\sqrt{\frac{T(x)}{n}}\right\}dx \le q2^{q+1}\int_{R_2}^{+\infty}x^{q-1}T(x) dx.
\end{equation*}
Since $x \ge R_2$, we have $x^{q-1}\le R_2^{-q}x^{2q-1}$. Therefore,
\begin{equation*}
I_3 \le q2^{q+1}R_{2}^{-q}\int_{0}^{+\infty}x^{2q-1}T(x)dx = q2^{q+1}R^{-q}_2\frac{M_{Q,2q}(\theta)}{2q}.
\end{equation*}
Using the definition of $R_2$, we obtain
\begin{equation*}
I_3 \lesssim_{q} (M_{Q,2q}(\theta))^{1/2}\frac{1}{(n+1)^q} \lesssim_{q} \left(M_{Q,2q}(\theta)\right)^{1/2}\sqrt{\log(n+1)n^{-1}}
\end{equation*}
Finally, combining the bounds for $I_1, I_2$ and $I_3$ completes the proof.
\end{proof}
\subsubsection{Proof of Lemma~\ref{lem:endpointweightedquantile}}\label{appendix:proofoflem:endpointweightedquantile}
\begin{proof}
% Throughout the proof, $C_p>0$ denotes a constant depending only on $p$, while $C,c>0$ denote universal constants. The values of $C_p$ and $C$ may change from line to line.
Set 
\begin{equation*}
M_{\theta}\coloneqq M_{\mu,2p}(\theta) +M_{\nu,2p}(\theta).
\end{equation*}
If $M_{\theta}=0$, then $M_{\mu,2p}(\theta)=M_{\nu,2p}(\theta)=0$. Thus, $F^{-1}_{\theta}(u)=G^{-1}_{\theta}(u)=0$ for every $u \in (0,1)$ and the desired estimate follows immediately. We therefore assume that $M_{\theta}>0.$

Define
\begin{equation*}
H_{\theta}(u)\coloneqq \max \left\{\left|F^{-1}_{\theta}(u)\right|, \left|G^{-1}_{\theta}(u)\right| \right\}, \quad u \in (0,1).
\end{equation*}
To obtain the desired estimate, it is enough to show that
$$\mathbb{E}\displaystyle \int_{0}^{1}\left|F^{-1}_{n,\theta}(u)-F^{-1}_{\theta}(u)\right|H_{\theta}(u)^{p-1}du\lesssim_{p}\left(M_{\mu,2p}(\theta)+M_{\nu,2p}(\theta)\right)^{1/2}\sqrt{\frac{\log(n+1)}{n}}.$$
We decompose $(0,1)$ according to the size of $H_{\theta}(u)$. Define
$$E^* \coloneqq \left\{ u\in(0,1): H_\theta(u)< M^{1/(2p)}_{\theta}
\right\},
$$
and, for $k\ge0$,
$$ E_k \coloneqq \left\{ u\in(0,1): 2^kM^{1/(2p)}_{\theta}\le H_\theta(u)<2^{k+1}M^{1/(2p)}_{\theta}\right\}.$$
Then $$(0,1)=E^*\cup\bigcup_{k\ge0}E_k.$$

\textbf{Step 1: Control on the region $E^*$.} 
On $E^*$, we have
\begin{equation*}
H^{p-1}_{\theta}(u)\le M_{\theta}^{(p-1)/(2p)}.
\end{equation*}
Therefore, by Lemma~\ref{lem:projected-empirical-rate} (with $q=1$, $s=2p$), we get
\begin{equation*}
\mathbb{E}\int_{E^*}\left|F^{-1}_{n,\theta}(u)-F^{-1}_{\theta}(u)\right|H^{p-1}_{\theta}(u) du \lesssim_{p}M_{\theta}^{(p-1)/(2p)}\left(M_{\mu,2p}(\theta)\right)^{1/(2p)}n^{-1/2}\lesssim_{p}M^{1/2}_{\theta}\sqrt{\frac{\log (n+1)}{n}}.
\end{equation*}

\textbf{Step 2: Control on the remaining regions $E_k$ for $k \ge 0$.}
In this step, we first need the following auxiliary lemma.

\medskip\noindent\textbf{Auxiliary lemma.} For $k\ge 0$, we set
\begin{equation*}
L_{k}\coloneqq \left\{u \in (0,1):H_{\theta}(u)\ge 2^kM_{\theta}^{1/(2p)}\right\}.
\end{equation*}
Let $S_k\coloneqq \lambda(L_k)$ where $\lambda$ denotes Lebesgue measure on $(0,1)$. Define
\begin{equation*}
m_{k}\coloneqq \min\left\{S_k,\sqrt{\frac{S_k}{n}}\right\}.
\end{equation*}
Then there exists a universal constant $C>0$ such that, for every $k\ge 0$,
\begin{equation*}
\mathbb{E}\int_{L_k}\left|F^{-1}_{n,\theta}(u)-F^{-1}_{\theta}(u)\right|du \le CM_{\theta}^{1/(2p)}\sum_{j\ge k}2^jm_j.
\end{equation*}
\begin{proof}[Proof of the auxiliary lemma]
Since $F^{-1}_{\theta}$ and $G^{-1}_{\theta}$ are non-decreasing, each set of the form
\begin{equation*}
\left\{u \in (0,1):\left|F^{-1}_{\theta}(u)\right|\ge 2^kM^{1/(2p)}_{\theta}\right\} \quad \text{or} \quad \left\{u \in (0,1):\left|G^{-1}_{\theta}(u)\right|\ge 2^kM^{1/(2p)}_{\theta}\right\}
\end{equation*}
is the union of a left endpoint interval and a right endpoint interval. Therefore, $L_k$ is the union of at most one left endpoint interval and at most one right endpoint interval. We treat only the left endpoint since the right endpoint interval can be handled similarly.

Let $a_k$ denote the length of the left endpoint part of $L_k$. Then, by the definition, we have $0 \le a_k\le S_{k}$. By Lemma~\ref{lem:Quantile-CDF} and Tonelli's theorem, we have
$$ \mathbb{E}\int_0^{a_k} \left| F_{n,\theta}^{-1}(u)-F_\theta^{-1}(u)\right|du=\int_{\mathbb R}\mathbb{E}\left|
\min\{F_{n,\theta}(x),a_k\}-\min\{F_\theta(x),a_k\}
\right|dx.
$$
We split the integral in $x$ into two regions: $\left|x\right| \le 2^kM^{1/(2p)}_{\theta}$ and $\left|x\right| > 2^kM^{1/(2p)}_{\theta}$.

Since the map $y \mapsto \min\{y,a_k\}$ is $1$-Lipschitz, for any $x \in \mathbb{R}$, we have
\begin{equation*}
\mathbb{E}\left|
\min\{F_{n,\theta}(x),a_k\}-\min\{F_\theta(x),a_k\}
\right| \le \mathbb{E}\left|F_{n,\theta}(x)-F_{\theta}(x)\right|\le \sqrt{\frac{F_{\theta}(x)(1-F_{\theta}(x))}{n}}.
\end{equation*}
If $F_{\theta}(x) \le 2S_k$, then 
\begin{equation*}
\mathbb{E}\left|
\min\{F_{n,\theta}(x),a_k\}-\min\{F_\theta(x),a_k\}
\right| \le \sqrt{\frac{2S_k}{n}}.
\end{equation*}
If $F_{\theta}(x)>2S_k$, then $a_k\le S_k<F_{\theta}(x)/2$. As a consequence, 
\begin{equation*}
\left|\min\{F_{n,\theta}(x),a_k\}-\min\{F_\theta(x),a_k\}
\right|=\left|\min\{F_{n,\theta}(x),a_k\}-a_k
\right| \le a_k\mathbbm{1}_{\left\{F_{n,\theta}(x)\le a_k\right\}}.
\end{equation*}
Since $a_k \le S_k$, we have
\begin{equation*}
\mathbb{E}\left|
\min\{F_{n,\theta}(x),a_k\}-\min\{F_\theta(x),a_k\}
\right|\le S_k\mathbb{P}\left(F_{n,\theta}(x)\le S_k\right).
\end{equation*}
Using $S_k \le F_{\theta}(x) /2$ and Chernoff's inequality (see, for example, the proof of \citet[Corollary~2.3.4]{Vershynin_2026}), we have
\begin{equation*}
\mathbb{P}\left(F_{n,\theta}(x)\le S_k\right) \le \mathbb{P}\left(F_{n,\theta}(x)\le F_{\theta}(x)/2\right)\le \exp\left(-cnF_\theta(x)\right)\le \exp(-cnS_k).
\end{equation*}
Thus,
\begin{equation*}
S_k\mathbb{P}\left(F_{n,\theta}(x)\le S_k\right) \le S_k\exp\left(-cnS_k\right)=\sqrt{\frac{S_k}{n}}\sqrt{nS_k} \exp\left(-cnS_k\right)\le C\sqrt{\frac{S_k}{n}},
\end{equation*}
where we used the fact that the function $z\mapsto \sqrt ze^{-cz}$ is bounded on $[0,+\infty)$.
On the other hand, we also have
\begin{equation*}
\left|\min\{F_{n,\theta}(x),a_k\}-\min\{F_\theta(x),a_k\}
\right|\le a_k \le S_k.
\end{equation*}
Consequently, combining the estimates above, for any $x \in \mathbb{R}$, we have
\begin{equation*}
\mathbb{E}
\left|\min\left\{F_{n,\theta}(x),a_k\right\}-\min\left\{F_{\theta}(x),a_k\right\}\right|
\le Cm_k.
\end{equation*}
Hence,
\begin{equation*}
\int_{\left|x\right|\le 2^kM^{1/(2p)}_{\theta}}\mathbb{E}\left|\min\left\{F_{n,\theta}(x),a_k\right\}-\min\left\{F_{\theta}(x),a_k\right\}\right|dx \le C2^kM_{\theta}^{1/(2p)}m_k.
\end{equation*}

For the second region, as we did in the proof of Lemma~\ref{lem:projected-empirical-ratev2}, following the proof in \citet[Corollary~7.17]{bobkov2019one}, we have
\begin{align*}
\mathbb{E}\left|
\min\{F_{n,\theta}(x),a_k\}-\min\{F_\theta(x),a_k\}
\right|& \le \mathbb{E}\left|F_{n,\theta}(x)-F_{\theta}(x)\right|\\
& \le \min \left\{2F_{\theta}(x)\left(1-F_{\theta}(x)\right),\sqrt{\frac{F_{\theta}(x)\left(1-F_{\theta}(x)\right)}{n}}\right\}\\
& \le \min \left\{2T(\left|x\right|),\sqrt{\frac{T(\left|x\right|)}{n}}\right\},
\end{align*}
where $T(t)=\mathbb{P}_{X\sim\mu}\left(\left|\langle X,\theta\rangle\right|\ge t\right)$ for $t\ge0$.

Now, suppose that $2^jM^{1/(2p)}_{\theta}< \left|x\right|\le 2^{j+1}M^{1/(2p)}_{\theta}$ for some $j \ge k$. Since $\left|F^{-1}_{\theta}(u)\right| \le H_{\theta}(u)$ for all $u\in (0,1)$, we have
\begin{equation*}
\lambda\left\{u \in (0,1):\left|F^{-1}_{\theta}(u)\right|\ge \left|x\right|\right\} \le \lambda\left(L_j\right)=S_j.
\end{equation*}
On the other hand, using the fact that if $U\sim \text{Unif}(0,1)$ then $F^{-1}_{\theta}(U)$ has distribution $\Pi^{\theta}_{\#}\mu$, we obtain
\begin{equation*}
\lambda\left\{u \in (0,1):\left|F^{-1}_{\theta}(u)\right|\ge \left|x\right|\right\}=\mathbb{P}_{X \sim \mu}\left(\left|\langle X,\theta\rangle\right|\ge \left|x\right|\right)=T(\left|x\right|).
\end{equation*}
Consequently, 
\begin{equation*}
\min \left\{2T\left(\left|x\right|\right),\sqrt{\frac{T\left(\left|x\right|\right)}{n}}\right\} \le \min \left\{2S_j,\sqrt{\frac{S_j}{n}}\right\}\le  2\min \left\{S_j,\sqrt{\frac{S_j}{n}}\right\}=2m_j.
\end{equation*}
Moreover, the Lebesgue measure of the shell
\begin{equation*}
\left\{x \in \mathbb{R}:2^jM^{1/(2p)}_{\theta} < |x| \le 2^{j+1} M^{1/(2p)}_{\theta}\right\}
\end{equation*}
is equal to $2^{j+1}M_{\theta}^{1/(2p)}$. Hence,
\begin{align*}
&\int_{\left|x\right|>2^kM_\theta^{1/(2p)}}
\mathbb E\left|\min\left\{F_{n,\theta}(x),a_k\right\}-\min\left\{F_\theta(x),a_k\right\}\right|dx \\
&\quad =\sum_{j\ge k}\int_{2^jM_\theta^{1/(2p)}<\left|x\right|\le 2^{j+1}M_\theta^{1/(2p)}}\mathbb E\left|\min\left\{F_{n,\theta}(x),a_k\right\}-\min\left\{F_\theta(x),a_k\right\}
\right|dx \\
&\quad \lesssim M_\theta^{1/(2p)} \sum_{j\ge k}2^jm_j.
\end{align*}
Combining this estimate with the bound obtained on the region $\left|x\right|\le 2^kM^{1/(2p)}_{\theta}$ yields the desired estimate for the left endpoint part. Together with the analogous estimate for the right endpoint part, this completes the proof.
\end{proof}
We now return to the weighted integral. On $E_k$, we have $H^{p-1}_{\theta}(u)\le \left(2^{k+1}M^{1/(2p)}_{\theta}\right)^{p-1}$. Moreover, since $E_k\subset L_k$, the auxiliary lemma gives
\begin{equation*}
\mathbb{E}\int_{E_k}\left|F^{-1}_{n,\theta}(u)-F^{-1}_{\theta}(u)\right|H_{\theta}(u)^{p-1}du \lesssim_{p}M^{1/2}_{\theta}2^{k(p-1)}\sum_{j\ge k}2^jm_j.
\end{equation*}
On the other hand, since $p>1$, we have
\begin{equation*}
\sum_{k\ge 0}2^{k(p-1)}\sum_{j\ge k}2^jm_j=\sum_{j\ge 0}2^jm_j\sum_{k=0}^{j}2^{k(p-1)}\lesssim_{p}\sum_{j \ge 0}2^{jp}m_j.
\end{equation*}
Thus
\begin{equation*}
\sum_{k\ge 0}\mathbb{E}\int_{E_k}\left|F^{-1}_{n,\theta}(u)-F^{-1}_{\theta}(u)\right|H_{\theta}(u)^{p-1}du\lesssim_{p}M^{1/2}_{\theta}\sum_{j\ge 0}2^{jp}m_j.
\end{equation*}

Now, it remains to bound the last sum. By the definition of $S_j$ and Tonelli's theorem, we have
\begin{equation*}
\sum_{j\ge 0}2^{2pj}S_j = \int_{0}^{1}\sum_{j\ge 0}2^{2pj}\mathbbm{1}_{\left\{H_{\theta}(u)\ge 2^jM^{1/(2p)}_{\theta}\right\}}du.
\end{equation*}
For every $u \in (0,1)$, we have
\begin{equation*}
\sum_{j \ge 0}2^{2pj}\mathbbm{1}_{\left\{H_{\theta}(u)\ge 2^jM^{1/(2p)}_{\theta}\right\}} \lesssim_{p}M^{-1}_{\theta}H_{\theta}(u)^{2p}.
\end{equation*}
Therefore,
\begin{equation*}
\sum_{j\ge 0}2^{2pj}S_j \lesssim_{p}M^{-1}_{\theta} \int_0^1H_{\theta}(u)^{2p}du.
\end{equation*}
On the other hand, by the definition of $H_{\theta}$, we have
\begin{equation*}
\int_{0}^{1}H_{\theta}(u)^{2p}du \le \int_{0}^{1}\left|F^{-1}_{\theta}(u)\right|^{2p}du +\int_{0}^{1}\left|G^{-1}_{\theta}(u)\right|^{2p}du=M_{\theta}.
\end{equation*}
Hence, we deduce
\begin{equation*}
\sum_{j\ge 0}2^{2pj}S_j\lesssim_{p}1.
\end{equation*}
Let $J_n$ be a positive integer, to be chosen later. For $j \le J_n$, using the definition of $m_j$ and Cauchy--Schwarz inequality, we have
\begin{equation*}
\sum_{j=0}^{J_n}2^{jp}m_j \le \sum_{j=0}^{J_n}\sqrt{\frac{2^{2pj}S_j}{n}} \le \sqrt{\frac{J_n+1}{n}}\left(\sum_{j=0}^{J_n}2^{2pj}S_j\right)^{1/2}\lesssim_{p}\sqrt{\frac{J_n+1}{n}}.
\end{equation*}
On the other hand, for $j>J_n$, we have
\begin{equation*}
\sum_{j>J_n}2^{jp}m_j \le\sum_{j>J_n}2^{jp}S_j=\sum_{j>J_n}2^{-jp}2^{2jp}S_j  \le 2^{-pJ_n}\sum_{j >J_n}2^{2jp}S_j\lesssim_{p}2^{-pJ_n}.
\end{equation*}
Choose $J_n\coloneqq \left\lceil \frac{\log_2(n+1)}{2p}\right\rceil
$. By the definition of the ceiling function,
\begin{equation*}
J_n+1 \le \frac{\log(n+1)}{2p\log 2} +2 \lesssim_{p}\log(n+1),
\end{equation*}
where we use $\log(n+1) \ge \log 2$. Moreover,
\begin{equation*}
2^{-pJ_n} \le 2^{-1/2 \log_2(n+1)}=(n+1)^{-1/2}\lesssim\sqrt{\frac{\log(n+1)}{n}}.
\end{equation*}
Consequently,
\begin{equation*}
\sqrt{\frac{J_n+1}{n}} +2^{-pJ_n}\lesssim_{p}\sqrt{\frac{\log(n+1)}{n}}.
\end{equation*}
Hence,
\begin{equation*}
\sum_{k\ge 0}\mathbb{E}\int_{E_k}\left|F^{-1}_{n,\theta}(u)-F^{-1}_{\theta}(u)\right|H_{\theta}(u)^{p-1}du\lesssim_{p}M^{1/2}_{\theta}\sqrt{\frac{\log(n+1)}{n}}.
\end{equation*}
Combining this estimate with the estimate obtained in Step 1 completes the proof.
\end{proof}
\section{Omitted Proofs from Section~\ref{sec:prooflowerbound}}
\subsection{Proof of the two-point bound in Eq.~(\ref{eq:lecam})}
\label{appendix:proofofeqlecam}
Let $(\mu_0,\nu_0),(\mu_1,\nu_1)\in \mathcal{K}_{p,r}(b)\times \mathcal{K}_{p,r}(b)$, and define
$$ T_0\coloneqq \operatorname{SW}_p^p(\mu_0,\nu_0),
\qquad T_1\coloneqq \operatorname{SW}_p^p(\mu_1,\nu_1).
$$
For any estimator $\widehat T=\widehat T(X_1,\dots,X_n,Y_1,\dots,Y_m)$, we have
\begin{align*}
\sup_{\mu,\nu\in \mathcal{K}_{p,r}(b)}\mathbb{E}_{\mu,\nu}\left|\widehat{T}_{n,m}-\operatorname{SW}_p^p(\mu,\nu)\right|&\ge \max \left\{\mathbb{E}_{\mu_0,\nu_0}\left|\widehat{T}_{n,m}-T_0\right|,\mathbb{E}_{\mu_1,\nu_1}\left|\widehat{T}_{n,m}-T_1\right|\right\}\\&\ge \frac{1}{2}\mathbb{E}_{\mu_0,\nu_0}\left|\widehat{T}_{n,m}-T_0\right|+\frac{1}{2}\mathbb{E}_{\mu_1,\nu_1}\left|\widehat{T}_{n,m}-T_1\right|,
\end{align*}
where $\mathbb E_{\mu_i,\nu_i}$ denotes expectation with respect to the joint law $\mu_i^{\otimes n}\otimes \nu_i^{\otimes m}$ for $i=0,1$.\\
We define $\Psi(\widehat{T}_{n,m})$ by
\begin{equation*}
\Psi(\widehat{T}_{n,m})\coloneqq \begin{cases}
0, & \text{if } \left|\widehat{T}_{n,m}-T_0\right| \le \left|\widehat{T}_{n,m}-T_1\right|,\\
1, & \text{if } \left|\widehat{T}_{n,m}-T_0\right| > \left|\widehat{T}_{n,m}-T_1\right|.
\end{cases}
\end{equation*}
On the one hand, we have
\begin{equation*}
\left|\widehat{T}_{n,m}-T_0\right|\ge \left|\widehat{T}_{n,m}-T_0\right|\mathbbm{1}_{\{\Psi(\widehat{T}_{n,m})=1\}}.
\end{equation*}
Moreover, by the definition of $\Psi(\widehat{T}_{n,m})$, on the event $\{\Psi(\widehat{T}_{n,m})=1\}$, we have $\left|\widehat{T}_{n,m}-T_0\right|>\left|\widehat{T}_{n,m}-T_1\right|$. Therefore, by the triangle inequality,
\begin{equation*}
\left|T_1-T_0\right|\le \left|T_1-\widehat{T}_{n,m}\right|+\left|\widehat{T}_{n,m}-T_0\right|\le 2\left|\widehat{T}_{n,m}-T_0\right|.
\end{equation*}
Thus, on the event $\{\Psi(\widehat{T}_{n,m})=1\}$, we have $\left|\widehat{T}_{n,m}-T_0\right|\ge \left|T_1-T_0\right|/{2}$. Hence,
\begin{equation*}
\mathbb{E}_{\mu_0,\nu_0}\left|\widehat{T}_{n,m}-T_0\right|\ge \frac{\left|T_1-T_0\right|}{2}\mathbb{P}_{\mu_0,\nu_0}\left(\Psi(\widehat{T}_{n,m})=1\right).
\end{equation*}
Similarly, we can prove that
\begin{equation*}
\mathbb{E}_{\mu_1,\nu_1}\left|\widehat{T}_{n,m}-T_1\right|\ge \frac{\left|T_1-T_0\right|}{2}\mathbb{P}_{\mu_1,\nu_1}\left(\Psi(\widehat{T}_{n,m})=0\right).
\end{equation*}
Consequently, we have
\begin{align*}
\sup_{\mu,\nu\in \mathcal{K}_{p,r}(b)}\mathbb{E}_{\mu,\nu}\left|\widehat{T}_{n,m}-\operatorname{SW}_p^p(\mu,\nu)\right|&\ge \frac{\left|T_1-T_0\right|}{4}\left[\mathbb{P}_{\mu_0,\nu_0}\left(\Psi(\widehat{T}_{n,m})=1\right)+\mathbb{P}_{\mu_1,\nu_1}\left(\Psi(\widehat{T}_{n,m})=0\right)\right]\\
&= \frac{\left|T_1-T_0\right|}{4}\left[1-\mathbb{P}_{\mu_0,\nu_0}\left(\Psi(\widehat{T}_{n,m})=0\right)+\mathbb{P}_{\mu_1,\nu_1}\left(\Psi(\widehat{T}_{n,m})=0\right)\right]\\
&\ge \frac{\left|T_1-T_0\right|}{4}\left(1-\operatorname{TV}\left(\mu_0^{\otimes n}\otimes \nu_0^{\otimes m},\mu_1^{\otimes n}\otimes \nu_1^{\otimes m}\right)\right)\\
&\ge \frac{\left|T_1-T_0\right|}{4}\left[1-\sqrt{\frac{1}{2}\operatorname{KL}\left(\mu_0^{\otimes n}\otimes \nu_0^{\otimes m}\,\|\,\mu_1^{\otimes n}\otimes \nu_1^{\otimes m}\right)}\right]\\
& = \frac{\left|T_1-T_0\right|}{4}\left[1-\sqrt{\frac{n\operatorname{KL}(\mu_0\|\mu_1)+m\operatorname{KL}(\nu_0\|\nu_1)}{2}}\right],
\end{align*}
where the second inequality follows from the definition of total variation
(see \citet[Definition~2.4]{tsybakov2008introduction}), the third follows
from Pinsker's inequality \citep[Lemma~2.5]{tsybakov2008introduction}, and
the last equality follows from the additivity of the Kullback--Leibler
divergence over product measures.\\
Taking the infimum over all estimators $\widehat{T}_{n,m}$, we obtain
\begin{equation*}
\mathcal{R}_{n,m}^{\star}\big(\mathcal K_{p,r}(b)\big)\ge \frac{\left|T_1-T_0\right|}{4}\left(1-\sqrt{\frac{n\operatorname{KL}(\mu_0\|\mu_1)+m\operatorname{KL}(\nu_0\|\nu_1)}{2}}\right).
\end{equation*}
\subsection{Proofs for the critical plug-in lower bound}
\subsubsection{Proof of Lemma~\ref{lem:extendedlemma3.8ledoux}}\label{appendix:proofoflem:extendedlemma3.8ledoux}
\begin{proof}
Conditioning on $\widetilde N$ and using Jensen's inequality with the convex function $z\mapsto |z|$, we obtain
\begin{align*}
\mathbb E\left[ \left| \frac{\widetilde N}{n} - \frac{\widetilde M}{m}
\right| \middle| \widetilde N
\right] &\ge\left|\frac{\widetilde N}{n} -
\mathbb E\left[\frac{\widetilde M}{m}\middle|\widetilde N \right] \right|.
\end{align*}
Since $\widetilde N$ and $\widetilde M$ are independent, we have
$$ \mathbb E\left[\frac{\widetilde M}{m}\middle| \widetilde N
\right] = \mathbb E\left[\frac{\widetilde M}{m} \right]= s.
$$
Therefore,
$$
\mathbb E\left[ \left|\frac{\widetilde N}{n} - \frac{\widetilde M}{m} \right| \middle|\widetilde N\right] \ge \left| \frac{\widetilde N}{n}-s
\right|.
$$
Taking expectation with respect to $\widetilde N$, we get
$$\mathbb E\left| \frac{\widetilde N}{n}- \frac{\widetilde M}{m}
\right| \ge \mathbb E\left|\frac{\widetilde N}{n}-s \right|= \frac{1}{n}\mathbb E|\widetilde N-ns|.$$
By \citet[Lemma~3.8]{bobkov2019one}, there exists a universal constant $c_0>0$ such that
$$\mathbb E|\widetilde N-ns| \ge c_0 \min\{2ns(1-s),\sqrt{ns(1-s)}\}.
$$
Since $0<s\le1/2$, we have $1-s\ge1/2$. Moreover, since $ns\ge1$,
\begin{equation*}
2ns(1-s)\ge ns \ge \sqrt{ns},
\end{equation*}
while
\begin{equation*}
\sqrt{ns(1-s)} \ge \frac{1}{\sqrt{2}}\sqrt{ns}.
\end{equation*}
Therefore,
\begin{equation*}
\min \left\{2ns(1-s),\sqrt{ns(1-s)}\right\}\ge \frac{1}{\sqrt{2}}\sqrt{ns}.
\end{equation*}
Consequently,
$$\mathbb E\left| \frac{\widetilde N}{n}- \frac{\widetilde M}{m} \right|\ge \frac{c_0}{n\sqrt{2}}\sqrt{ns}=c\sqrt{\frac{s}{n}},
$$
for a universal constant $c=\frac{c_0}{\sqrt{2}}>0$.
\end{proof}
\subsubsection{Proof of Lemma~\ref{lem:onedimensionalcaser=2p}}\label{appendix:proofoflemmaonedimensionalcaser=2p}
As explained in the main text, the construction in
Lemma~\ref{lem:onedimensionalcaser=2p} relies on the following useful observation.
\begin{lemma}\label{lem:technicallemmaconstructionr=2p}
Let $x_0<x_1<\dots <x_L$ and let $\widehat P_n$ and $\widehat Q_m$ be probability measures supported on the grid $\{x_0,x_1,\ldots,x_L\}$. For $k=1,\dots,L$, denote $A_k=\left[x_k,+\infty\right)$, then for every $p \ge 1$, we have
\begin{equation*}
\operatorname{W}^p_p(\widehat{P}_n,\widehat{Q}_m)\ge \sum_{k=1}^{L}\left(x_k-x_{k-1}\right)^p\left|\widehat{P}_n\left(A_k\right)-\widehat{Q}_m\left(A_k\right)\right|.
\end{equation*}
\end{lemma}
\begin{proof}
For each $i=0,\dots ,L$, denote
\begin{equation*}
\phi_i = \widehat{P}_n\left(\left\{x_i\right\}\right), \quad \psi_i = \widehat{Q}_m\left(\left\{x_i\right\}\right).
\end{equation*}
Let $\pi$ be an arbitrary transport plan. Since both measures on the finite grid, $\pi$ can be represented by a nonnegative matrix 
\begin{equation*}
\pi = \left(\pi_{ij}\right)_{0 \le i,j\le L},
\end{equation*}
where $\pi_{ij}$ is the amount of mass transported from $x_i$ to $x_j$. Besides, the marginal conditions are
\begin{equation*}
\sum_{j=0}^{L} \pi_{ij}= \phi _i, \quad i=0,\dots,L.
\end{equation*}
and
\begin{equation*}
\sum_{i=0}^{L} \pi_{ij} = \psi_j,\quad j=0,\dots,L.
\end{equation*}
The transportation cost of $\pi$ is 
\begin{equation*}
C^*_p(\pi)=\sum_{i,j=0}^{L} \left|x_i-x_j\right|^p\pi_{ij}.
\end{equation*}
For each $k=1,\dots,L$, we regard the interval between $x_{k-1}$ and $x_k$ as the $k$-th cut, meaning
\begin{equation*}
\left\{x_0,\dots,x_{k-1}\right\} \big| \left\{x_k,\dots,x_L\right\}.
\end{equation*}
For fixed $i,j$, we define
\begin{equation*}
\mathbb{I}^{(k)}_{ij} = \mathbbm{1}_{\{i<k\le j \text{ or } j < k \le i\}}.
\end{equation*}
We then have
\begin{equation*}
\left|x_i-x_j\right|^p \ge \sum_{k=1}^{L} \left|x_k-x_{k-1}\right|^p\mathbb{I}^{(k)}_{ij}.
\end{equation*}
Indeed, suppose first that $i <j$. Then, we have
\begin{equation*}
x_j-x_i =\sum_{k=i+1}^{j}\left(x_k-x_{k-1}\right).
\end{equation*}
Since $p \ge 1$, and $\left(x_k-x_{k-1}\right)>0$, we have
\begin{equation*}
\left[\sum_{k=i+1}^{j}\left(x_k-x_{k-1}\right)\right]^p \ge \sum_{k=i+1}^j\left(x_k-x_{k-1}\right)^p.
\end{equation*}
Hence, 
\begin{equation*}
\left|x_i-x_j\right|^p\ge \sum_{k=i+1}^j\left(x_k-x_{k-1}\right)^p=\sum_{k=1}^L\left(x_k-x_{k-1}\right)^p\mathbb{I}^{(k)}_{ij}.
\end{equation*}
The case $i>j$ follows by the same argument, whereas if $i=j$, both sides are zero.\\
Then, multiplying by $\pi_{ij}\ge 0$, we obtain, for every fixed pair $(i,j)$,
\begin{equation*}
\left|x_i-x_j\right|^p \pi_{ij}\ge \sum_{k=1}^L \left(x_k-x_{k-1}\right)^p\mathbb{I}^{(k)}_{ij}\pi_{ij}.
\end{equation*}
Summing over all $i,j$, we get
\begin{equation*}
\sum_{i,j=0}^L\left|x_i-x_j\right|^p\pi_{ij} \ge \sum_{i,j=0}^L\sum_{k=1}^L \left(x_k-x_{k-1}\right)^p\mathbb{I}^{(k)}_{ij}\pi_{ij}=\sum_{k=1}^{L}\left(x_k-x_{k-1}\right)^p\sum_{i,j=0}^{L}\mathbb{I}^{(k)}_{ij}\pi_{ij}.
\end{equation*}
By the definition, for each fixed $k$, we  have
\begin{equation*}
\sum_{i,j=0}^L\mathbb{I}^{(k)}_{ij}\pi_{ij} = \sum_{i<k \le j} \pi_{ij} + \sum_{j <k \le i} \pi_{ij}.
\end{equation*}
Therefore,
\begin{equation*}
C^*_p(\pi) \ge \sum_{k=1}^L \left(x_k-x_{k-1}\right)^p\left(\sum_{i<k \le j} \pi_{ij} + \sum_{j <k \le i} \pi_{ij}\right).
\end{equation*}
On the other hand, since both measures are supported on $\{x_0,\dots,x_L\}$, then
\begin{equation*}
\widehat{P}_n\left(A_k\right)=\sum_{i=k}^{L}\phi_i=\sum_{i\ge k,j\ge k}\pi_{ij}+\sum_{i\ge k,j<k}\pi_{ij},
\end{equation*}
and 
\begin{equation*}
\widehat{Q}_m\left(A_k\right)=\sum_{i=k}^{L}\psi_i=\sum_{i\ge k,j\ge k}\pi_{ij}+\sum_{i< k,j\ge k}\pi_{ij}.
\end{equation*}
Consequently,
\begin{equation*}
\widehat{P}_n\left(A_k\right)-\widehat{Q}_m\left(A_k\right)= \sum_{i\ge k,j<k}\pi_{ij} - \sum_{i<k,j\ge k}\pi_{ij}.
\end{equation*}
Therefore,
\begin{equation*}
\left|\widehat{P}_n\left(A_k\right)-\widehat{Q}_m\left(A_k\right)\right| \le \sum_{i \ge k,j<k}\pi_{ij} + \sum_{i<k,j\ge k}\pi_{ij}.
\end{equation*}
Hence,
\begin{equation*}
C^*_p(\pi)\ge \sum_{k=1}^{L} \left(x_k-x_{k-1}\right)^p \left|\widehat{P}_n\left(A_k\right)-\widehat{Q}_m\left(A_k\right)\right|.
\end{equation*}
Since this inequality holds for every transport plan, taking the infimum over all such $\pi$ completes the proof.
\end{proof}
% We are now ready to prove Lemma~\ref{lem:onedimensionalcaser=2p}.

\begin{proof}[Proof of Lemma~\ref{lem:onedimensionalcaser=2p}]
Set $N_p\coloneqq\max\left\{2^{8p},1024\right\}.$ We divide the proof into two cases. We first consider $n\ge N_p$ and use
the multiscale construction described in the main text. The remaining case
$1\le n<N_p$ is treated by a one-scale construction.

\textbf{Step 1: Multiscale construction for $n\ge N_p$.}
Assume that $n\ge N_p$ and set
\begin{equation*}
L\coloneqq\left\lfloor\frac{1}{4p}\log_2 n\right\rfloor.
\end{equation*}
Since $\frac{1}{4p}\log_2 n\ge2,$ we have
\begin{equation*}
L\ge\frac{1}{8p}\log_2 n.
\end{equation*}

We choose the support points
\begin{equation*}
x_0=0,\qquad x_k=R2^k,\quad k=1,\ldots,L,
\end{equation*}
where $R>0$ will be specified below. Set $S_0=1$, and for
$k=1,\ldots,L$, define
\begin{equation*}
S_k\coloneqq\frac{B^2}{Lx_k^{2p}}.
\end{equation*}
We choose
\begin{equation*}
R\coloneqq\left( \frac{B^2n}{16L2^{2pL}}
\right)^{1/(2p)}.
\end{equation*}
With this choice,
\begin{equation*}
S_L=\frac{B^2}{L(R2^L)^{2p}}=\frac{16}{n}.
\end{equation*}
Finally, set $S_{L+1}=0$ and define
\begin{equation*}
q_k\coloneqq S_k-S_{k+1}, \qquad
k=0,\ldots,L.
\end{equation*}
We then define
\begin{equation*}
P\coloneqq\sum_{k=0}^Lq_k\delta_{x_k}.
\end{equation*}

\textbf{Step 2: Verification of the construction.}
For every $k=1,\ldots,L-1$, since $x_{k+1}=2x_k$, we have
\begin{equation*}
S_{k+1}=\frac{B^2}{Lx_{k+1}^{2p}}=2^{-2p}S_k.
\end{equation*}
Since $2^{-2p}\in(0,1)$,
\begin{equation*}
0<S_L<\cdots<S_1.
\end{equation*}
Moreover,
\begin{equation*}
S_1=S_L2^{2p(L-1)}\le \frac{16}{n}2^{2pL}.
\end{equation*}
By the definition of $L$,
\begin{equation*}
L\le\frac{1}{4p}\log_2 n,
\end{equation*}
and therefore
\begin{equation*}
2^{2pL} \le 2^{\frac12\log_2 n}=\sqrt n.
\end{equation*}
It follows that
\begin{equation*}
S_1\le\frac{16}{\sqrt n}\le\frac{1}{2},
\end{equation*}
where the last inequality follows from $n\ge1024$. Hence,
\begin{equation*}
0<S_L<\cdots<S_1\le\frac12.
\end{equation*}

Since $S_0=1>S_1$ and $S_{L+1}=0$, the sequence
$(S_k)_{k=0}^{L+1}$ is nonincreasing. Thus,
\begin{equation*}
q_k=S_k-S_{k+1}\ge0,\qquad k=0,\ldots,L,
\end{equation*}
and
\begin{equation*}
\sum_{k=0}^Lq_k=\sum_{k=0}^L(S_k-S_{k+1})=S_0-S_{L+1}=1.
\end{equation*}
Therefore, $P$ is a probability measure. Moreover, for every $k=1,\ldots,L$,
\begin{equation*}
P([x_k,+\infty))=\sum_{j=k}^Lq_j=\sum_{j=k}^L(S_j-S_{j+1})=S_k.
\end{equation*}

We also have, for every $k=1,\ldots,L$,
\begin{equation*}
nS_k\ge nS_L=16\ge1.
\end{equation*}
Together with $0<S_k\le1/2$, this shows that all the conditions required
to apply Lemma~\ref{lem:extendedlemma3.8ledoux} with $s=S_k$ are satisfied.

It remains to verify the moment constraint. By the definition of $P$,
\begin{equation*}
\int_{\mathbb R}|x|^{2p}d(P(x))=\sum_{k=0}^Lx_k^{2p}q_k=\sum_{k=0}^Lx_k^{2p}(S_k-S_{k+1})\le\sum_{k=0}^Lx_k^{2p}S_k=\sum_{k=1}^Lx_k^{2p}S_k,
\end{equation*}
where the last equality follows from $x_0=0$. Since
\begin{equation*}
x_k^{2p}S_k=\frac{B^2}{L},\qquad k=1,\ldots,L,
\end{equation*}
we obtain
\begin{equation*}
\int_{\mathbb R}|x|^{2p}d\left(P(x)\right)\le\sum_{k=1}^L\frac{B^2}{L}= B^2.
\end{equation*}
Hence,
\begin{equation*}
\left(\int_{\mathbb R}|x|^{2p}d\left(P(x)\right)\right)^{1/2}\le B.
\end{equation*}

\textbf{Step 3: Lower bound for
$\mathbb E\operatorname{W}_p^p(\widehat P_n,\widehat Q_m)$.}
For $k=1,\ldots,L$, let
\begin{equation*}
A_k=[x_k,+\infty).
\end{equation*}
Since $\widehat P_n$ and $\widehat Q_m$ are supported on
$\{x_0,x_1,\ldots,x_L\}$, Lemma~\ref{lem:technicallemmaconstructionr=2p} gives
\begin{equation*}
\operatorname{W}_p^p(\widehat P_n,\widehat Q_m)\ge\sum_{k=1}^L(x_k-x_{k-1})^p
\left|\widehat P_n(A_k)-\widehat Q_m(A_k)\right|.
\end{equation*}
Taking expectations yields
\begin{equation*}
\mathbb E\operatorname{W}_p^p(\widehat P_n,\widehat Q_m)\ge\sum_{k=1}^L(x_k-x_{k-1})^p\mathbb E\left|\widehat P_n(A_k)-\widehat Q_m(A_k)\right|.
\end{equation*}

For each $k=1,\ldots,L$, define
\begin{equation*}
N_k\coloneqq n\widehat P_n(A_k), \qquad M_k\coloneqq m\widehat Q_m(A_k).
\end{equation*}
Since $P(A_k)=S_k$, we have
\begin{equation*}
N_k\sim\operatorname{Binomial}(n,S_k),\qquad M_k\sim\operatorname{Binomial}(m,S_k),
\end{equation*}
and $N_k$ and $M_k$ are independent. Therefore, by
Lemma~\ref{lem:extendedlemma3.8ledoux},
\begin{equation*}
\mathbb E\left|\widehat P_n(A_k)-\widehat Q_m(A_k)\right|=\mathbb E\left|\frac{N_k}{n}-\frac{M_k}{m}\right|\ge c\sqrt{\frac{S_k}{n}},
\end{equation*}
where $c>0$ is a universal constant. Consequently,
\begin{equation*}
\mathbb E\operatorname{W}_p^p(\widehat P_n,\widehat Q_m)\ge c\sum_{k=1}^L
(x_k-x_{k-1})^p\sqrt{\frac{S_k}{n}}.
\end{equation*}

For every $k=1,\ldots,L$, the construction gives
\begin{equation*}
x_k-x_{k-1}\ge\frac{x_k}{2}.
\end{equation*}
Indeed, this is immediate for $k=1$ since $x_0=0$, while for
$k\ge2$ we have $x_{k-1}=x_k/2$. Therefore,
\begin{equation*}
(x_k-x_{k-1})^p\sqrt{\frac{S_k}{n}}\ge2^{-p}x_k^p\sqrt{\frac{B^2}{nLx_k^{2p}}}=2^{-p}\frac{B}{\sqrt{Ln}}.
\end{equation*}
Summing over $k=1,\ldots,L$, we obtain
\begin{equation*}\mathbb E\operatorname{W}_p^p(\widehat P_n,\widehat Q_m)
\gtrsim_pB\sqrt{\frac{L}{n}},
\end{equation*}

For $n\ge2$,
\begin{equation*}\log_2n=\frac{\log n}{\log2}\ge \frac{\log(n+1)}{2\log2}.
\end{equation*}
Combining this with
\begin{equation*}L\ge\frac{1}{8p}\log_2n,
\end{equation*}
we conclude that
\begin{equation*}
\mathbb E\operatorname{W}_p^p(\widehat P_n,\widehat Q_m)\gtrsim_p B \sqrt{\frac{\log(n+1)}{n}}
\end{equation*}
for every $n\ge N_p$.

\textbf{Step 4: The remaining values of $n$.}
We now consider $1\le n<N_p$. Let
\begin{equation*}
P\coloneqq \frac{1}{2}\delta_0+\frac12\delta_a,
\end{equation*}
where $a>0$ is chosen such that
\begin{equation*}
\frac{1}{2}a^{2p}=B^2.
\end{equation*}
Then
\begin{equation*}
\int_{\mathbb R}|x|^{2p}d\left(P(x)\right)=\frac12a^{2p}= B^2,
\end{equation*}
so the moment constraint is satisfied.

Let $\widehat P_n$ and $\widehat Q_m$ be the empirical measures associated
with two independent samples from $P$, and define
\begin{equation*}
N\coloneqq n\widehat P_n(\{a\}),\qquad M\coloneqq m\widehat Q_m(\{a\}).
\end{equation*}
Then $N$ and $M$ are independent and
\begin{equation*}
N\sim\operatorname{Binomial}\left(n,\frac12\right),\qquad
M\sim\operatorname{Binomial}\left(m,\frac12\right).
\end{equation*}
Moreover,
\begin{equation*}
\widehat P_n=\left(1-\frac{N}{n}\right)\delta_0+
\frac{N}{n}\delta_a,\qquad\widehat Q_m=\left(1-\frac{M}{m}\right)\delta_0+\frac{M}{m}\delta_a,
\end{equation*}
and hence
\begin{equation*}
\operatorname{W}_p^p(\widehat P_n,\widehat Q_m)=a^p\left|\frac{N}{n}-\frac{M}{m}\right|.
\end{equation*}

Assume first that $2\le n<N_p$. Applying
Lemma~\ref{lem:extendedlemma3.8ledoux} with $s=1/2$, we obtain
\begin{equation*}
\mathbb E\left|\frac{N}{n}-\frac{M}{m}\right|\ge c\sqrt{\frac{1}{2n}}.
\end{equation*}
Since $a^p=\sqrt{2}B$, it follows that
\begin{equation*}
\mathbb E\operatorname{W}_p^p(\widehat P_n,\widehat Q_m)\ge cBn^{-1/2}.
\end{equation*}
Since $n<N_p$,
\begin{equation*}
\log(n+1)\le\log(N_p+1),
\end{equation*}
and therefore
\begin{equation*}
n^{-1/2}\ge\frac{1}{\sqrt{\log(N_p+1)}}\sqrt{\frac{\log(n+1)}{n}}.
\end{equation*}
Since $N_p$ depends only on $p$, after adjusting the constant $c_p$,
\begin{equation*}
\mathbb E\operatorname{W}_p^p(\widehat P_n,\widehat Q_m)
\gtrsim_p B\sqrt{\frac{\log(n+1)}{n}}
\end{equation*}
for every $2\le n<N_p$.

It remains to consider $n=1$. In this case, $N$ is a Bernoulli random
variable with parameter $1/2$. Conditioning on $N$ and applying Jensen's
inequality,
\begin{align*}
\mathbb E\left|N-\frac{M}{m}\right|=\mathbb E\left[\mathbb E\left(\left|N-\frac{M}{m}\right|\middle|N\right)\right]\ge\mathbb E\left|
N-\mathbb E\left[\frac{M}{m}\middle|N\right]\right|.
\end{align*}
Since $N$ and $M$ are independent,
\begin{equation*}
\mathbb E\left[\frac{M}{m}\middle|N\right]=\frac{1}{2}.
\end{equation*}
Thus,
\begin{equation*}
\mathbb E\left|N-\frac{M}{m}\right|\ge\mathbb E\left|N-\frac12\right|= \frac{1}{2}.
\end{equation*}
Consequently,
\begin{equation*}
\mathbb E\operatorname{W}_p^p(\widehat P_1,\widehat Q_m)\ge\frac{a^p}{2}=\frac{B}{\sqrt2}.
\end{equation*}
Hence,
\begin{equation*}
\mathbb E\operatorname{W}_p^p(\widehat P_1,\widehat Q_m)\gtrsim_pB\sqrt{\log2}\gtrsim_pB\sqrt{\frac{\log(1+1)}{1}}.
\end{equation*}

Combining the two cases we obtain
\begin{equation*}
\mathbb E\operatorname{W}_p^p(\widehat P_n,\widehat Q_m)\gtrsim_pB\sqrt{\frac{\log(n+1)}{n}}
\end{equation*}
for every $1\le n\le m$. This completes the proof.
\end{proof}
\subsection{Proofs for the critical minimax lower bound}
\subsubsection{Proof of Lemma~\ref{lem:twopointmixtureversion}}\label{appendix:proofoflem:twopointmixtureversion}
\begin{proof}
Let $\widehat T = \widehat T(X_1,\ldots,X_n,Y_1,\ldots,Y_m)$ be an arbitrary estimator. 

Since $(\mu_0,\nu_0)$ and
$(\mu_\alpha,\nu_0)$ belong to $\mathcal K_{p,2p}(b)\times\mathcal K_{p,2p}(b)$ for every
$\alpha\in\mathcal A$, we have
\begin{align*}
\sup_{\mu,\nu\in\mathcal K_{p,2p}(b)}\mathbb E_{\mu,\nu} \left|
\widehat T-\operatorname{SW}_p^p(\mu,\nu) \right|
&\ge \max\left\{ \mathbb E_{\mu_0,\nu_0}|\widehat T|,
\frac{1}{|\mathcal A|} \sum_{\alpha\in\mathcal A}
\mathbb E_{\mu_\alpha,\nu_0}
\left| \widehat T-\operatorname{SW}_p^p(\mu_\alpha,\nu_0)\right|\right\}
\\
&= \max\left\{
\mathbb E_{\mathsf P_0\otimes\nu_0^{\otimes m}}
|\widehat T|, \mathbb E_{\mathsf P_1\otimes\nu_0^{\otimes m}}
\left| \widehat T-\operatorname{SW}_p^p(\mu_{\alpha_\star},\nu_0) \right|
\right\},
\end{align*}
where the equality follows from the fact that
$\operatorname{SW}_p^p(\mu_\alpha,\nu_0)$ has the same value for every
$\alpha\in\mathcal A$.

Applying the same two-point argument as in
Appendix~\ref{appendix:proofofeqlecam}, we obtain
\begin{align*}
\sup_{\mu,\nu\in\mathcal K_{p,2p}(b)} \mathbb E_{\mu,\nu} \left| \widehat T -\operatorname{SW}_p^p(\mu,\nu)\right|
&\ge  \frac{\operatorname{SW}_p^p(\mu_{\alpha_\star},\nu_0)}{4} \left[ 1- \operatorname{TV}\left( \mathsf P_0\otimes\nu_0^{\otimes m}, \mathsf P_1\otimes\nu_0^{\otimes m} \right) \right]
\\
&=
\frac{\operatorname{SW}_p^p(\mu_{\alpha_\star},\nu_0)}{4} \left[
1-\operatorname{TV}(\mathsf P_0,\mathsf P_1) \right].
\end{align*}
Since the inequality holds for every estimator $\widehat T$, taking the
infimum over all estimators completes the proof.
\end{proof}
\subsubsection{Proof of Lemma~\ref{lem:rademacher-block-construction}}\label{appendix:proofoflem:rademacher-block-construction}
\begin{proof}
We first verify that $\mu_0$ and $\mu_{\varepsilon}$ are probability measures. Since $0 \le \gamma_k \le \frac{\omega_k}{2}$ for all $k=1,\dots,L$, all the coefficients appearing in the definition of $\mu_0$ and $\mu_{\varepsilon}$ are nonnegative. Moreover, we also have
\begin{equation*}
\mu_0 \left(\mathbb{R}^d\right)=1-\Omega + \sum_{k=1}^L \omega_k=1,
\end{equation*}
and, for every $\varepsilon \in \{-1,1\}^L$,
\begin{equation*}
\mu_{\varepsilon}\left(\mathbb{R}^d\right)=1-\Omega+\sum_{k=1}^L\left(\frac{\omega_k}{2}+\varepsilon_k\gamma_k\right)+\sum_{k=1}^{L}\left(\frac{\omega_k}{2}-\varepsilon_k\gamma_k\right)=1.
\end{equation*}
Hence $\mu_0$ and $\mu_{\varepsilon}$ are probability measures.

We next compute their powered Sliced Wasserstein functional. For each $k$, both $\mu_0$ and $\mu_{\varepsilon}$ assign the same total mass $\omega_k$ to the pair $\left\{a_ke_1,b_ke_1\right\}$. Since these pairs are disjoint and ordered along the $e_1$-axis, the one-dimensional optimal coupling matches the mass in each pair with the mass in the same pair, so that no mass is transported between different pairs. Therefore, for any $\theta \in \mathbb{S}^{d-1}$,
\begin{equation*}
\operatorname{W}^p_p\left(\Pi^{\theta}_{\#}\mu_{\varepsilon},\Pi^{\theta}_{\#}\mu_{0}\right)=\left|\langle e_1,\theta\rangle \right|^p\sum_{k=1}^{L}\gamma_k\left(b_k-a_k\right)^p.
\end{equation*}
Therefore,
\begin{equation*}
\operatorname{SW}^p_p\left(\mu_{\varepsilon},\mu_0\right)=\left(\int_{\mathbb{S}^{d-1}}\left|\langle e_1,\theta\rangle\right|^pd\sigma(\theta)\right)\sum_{k=1}^{L}\gamma_k\left(b_k-a_k\right)^p=\kappa_{p,d}\sum_{k=1}^{L}\gamma_k\left(b_k-a_k\right)^p.
\end{equation*}

Now, define the likelihood ratio
\begin{equation*}
L_n \coloneqq\frac{d\mathsf P_1} {d\mathsf P_0}.
\end{equation*}
Applying Scheffé's theorem (see, for example \citet[Lemma~2.1]{tsybakov2008introduction}),
and Cauchy-Schwarz, we have
\begin{equation*}
\operatorname{TV}\left(\mathsf P_0,\mathsf P_1\right)=\frac{1}{2} \mathbb{E}_{\mathsf P_0}\left|L_n-1\right| \le \frac{1}{2}\left[\mathbb{E}_{\mathsf P_0}\left(L_n-1\right)^2\right]^{1/2}=\frac{1}{2}\left[\mathbb{E}_{\mathsf P_0}L^2_n-1\right]^{1/2}.
\end{equation*}
Therefore, it remains to derive an upper bound for the second moment $\mathbb E_{\mathsf P_0}L_n^2$. 

Let $X_1,\ldots,X_n$ be independent random variables with common distribution $\mu_0$. For each $\varepsilon\in\{-1,1\}^L$, define the one-sample likelihood ratio
\begin{equation*}
\ell _{\varepsilon}\coloneqq \frac{d\mu_{\varepsilon}}{d\mu_0}.
\end{equation*}
Since $\mathsf P_1$ is the uniform mixture of the product measures $\mu_\varepsilon^{\otimes n}$, we have
\begin{equation*}
L_n\left(X_1,\dots,X_n\right)=2^{-L}\sum_{\varepsilon \in \{-1,1\}^L}\prod_{j=1}^n\ell_{\varepsilon}\left(X_j\right)=2^{-L}\sum_{\varepsilon \in \{-1,1\}^L}\prod_{j=1}^n\frac{\mu_{\varepsilon}\left(\{X_j\}\right)}{\mu_0\left(\{X_j\}\right)}.
\end{equation*}
Hence,
\begin{equation*}
L^2_n\left(X_1,\dots,X_n\right)=2^{-2L}\sum_{\varepsilon,\varepsilon'\in \{-1,1\}^L}\prod_{j=1}^n \frac{\mu_{\varepsilon}\left(\{X_j\}\right)\mu_{\varepsilon'}\left(\{X_j\}\right)}{\mu_0\left(\{X_j\}\right)^2}.
\end{equation*}
Since $X_1,\ldots,X_n$ are i.i.d. with common distribution $\mu_0$, taking expectations under $\mathsf P_0=\mu_0^{\otimes n}$ yields
\begin{align*}
\mathbb{E}_{\mathsf P_0}[L^2_n] &= 2^{-2L}\sum_{\varepsilon,\varepsilon'\in \{-1,1\}^L}\prod_{j=1}^{n}\mathbb{E}_{X_j\sim \mu_0}\left(\frac{\mu_{\varepsilon}\left(\{X_j\}\right)\mu_{\varepsilon'}\left(\{X_j\}\right)}{\mu_0\left(\{X_j\}\right)^2}\right)\\
&= 2^{-2L}\sum_{\varepsilon,\varepsilon' \in \{-1,1\}^L}\left[\mathbb{E}_{X\sim \mu_0}\left(\frac{\mu_{\varepsilon}\left(\{X\}\right)\mu_{\varepsilon'}\left(\{X\}\right)}{\mu_0\left(\{X\}\right)^2}\right)\right]^n\\
&=2^{-2L}\sum_{\varepsilon,\varepsilon' \in \{-1,1\}^L}\left[\sum_{z \in \text{supp} \left(\mu_0\right)}\frac{\mu_{\varepsilon}\left(\{z\}\right)\mu_{\varepsilon'}\left(\{z\}\right)}{\mu_0\left(\{z\}\right)}\right]^n.
\end{align*}
By construction, we deduce
\begin{align*}
\sum_{z\in\operatorname{supp}(\mu_0)} \frac{ \mu_\varepsilon(\{z\}) \mu_{\varepsilon'}(\{z\})
}{\mu_0(\{z\}) }
&=1-\Omega\\
&\quad + \sum_{k=1}^L
\frac{ \left(\frac{\omega_k}{2}+\varepsilon_k\gamma_k\right) \left(\frac{\omega_k}{2}+\varepsilon_k'\gamma_k\right)}{\omega_k/2}\\
&\quad+
\sum_{k=1}^L \frac{ \left(\frac{\omega_k}{2} -\varepsilon_k\gamma_k\right)\left(\frac{\omega_k}{2}-\varepsilon_k'\gamma_k\right)
}{\omega_k/2}\\
&=1-\Omega+\sum_{k=1}^L\left(
\omega_k+4\varepsilon_k\varepsilon_k'\frac{\gamma_k^2}{\omega_k}\right)
\\
&=1+4\sum_{k=1}^L\varepsilon_k\varepsilon_k'\frac{\gamma_k^2}{\omega_k}.
\end{align*}
Therefore,
\begin{equation*}
\mathbb{E}_{\mathsf{P}_0}[L^2_n]=2^{-2L}\sum_{\varepsilon,\varepsilon'\in\{-1,1\}^L}\left(1+4\sum_{k=1}^L\varepsilon_k\varepsilon'_k\frac{\gamma^2_k}{\omega_k}\right)^n
\end{equation*}
Adapting the second-moment argument of \citet[Appendix~H]{kim2022minimax} to our construction, we obtain the following bound. Let $\varepsilon$ and $\varepsilon'$ be independent and uniformly distributed over
$\{-1,1\}^L$. Then
\begin{align*}
\mathbb{E}_{\mathsf P_0}[L_n^2]
&=\mathbb{E}_{\varepsilon,\varepsilon'}\left[ \left(1+4\sum_{k=1}^L \varepsilon_k\varepsilon_k'\frac{\gamma_k^2}{\omega_k}\right)^n\right]
\\&\le\mathbb{E}_{\varepsilon,\varepsilon'}\left[\exp\left(
4n\sum_{k=1}^L\varepsilon_k\varepsilon_k'\frac{\gamma_k^2}{\omega_k}\right)\right]\\
&=\prod_{k=1}^L\mathbb{E}_{\varepsilon_k,\varepsilon_k'}\left[\exp\left(
4n\varepsilon_k\varepsilon_k'\frac{\gamma_k^2}{\omega_k}\right)\right]
\\&=\prod_{k=1}^L\cosh\left(\frac{4n\gamma_k^2}{\omega_k}\right)
\\&\le\exp\left(8n^2\sum_{k=1}^L\frac{\gamma_k^4}{\omega_k^2}
\right),
\end{align*}
where the first inequality follows from $1+x\le e^x$, and the last
inequality follows from $\cosh(x)\le e^{x^2/2}$. Consequently,
\begin{align*}
\operatorname{TV}(\mathsf P_0,\mathsf P_1)\le\frac{1}{2}\left(
\mathbb{E}_{\mathsf P_0}[L_n^2]-1\right)^{1/2}\le\frac{1}{2}\left[\exp\left(
8n^2\sum_{k=1}^L\frac{\gamma_k^4}{\omega_k^2}\right)-1\right]^{1/2}.
\end{align*}
\end{proof}
We are now ready to complete the proof of the minimax lower bound in
Proposition~\ref{prop:r=2p} by combining
Lemmas~\ref{lem:twopointmixtureversion} and
\ref{lem:rademacher-block-construction} with a suitable choice of parameters.
\subsubsection{Proof of the minimax lower bound in
Proposition~\ref{prop:r=2p}}\label{appendix:proofoftheminimaxlowerboundr=2p}
We distinguish two cases.
\begin{proof}
\textbf{Case 1.} If $1 \le n <4^{8p}$.

By the two-point lower bound established in Section~\ref{sec:lowerboundr>2p}, 
\begin{equation*}
\mathcal{R}^{\star}_{n,m}\left(\mathcal{K}_{p,2p}(b)\right)\gtrsim_{p}\frac{b}{\sqrt{n}}.
\end{equation*}
Since $n<4^{8p}$, we have
\begin{equation*}
\left[\log\left(n+1\right)\right]^{1/4} \le \left[\log\left(4^{8p}+1\right)\right]^{1/4}.
\end{equation*}
Therefore,
\begin{equation*}
\mathcal{R}^{\star}_{n,m}\left(\mathcal{K}_{p,2p}(b)\right) \gtrsim_p\frac{1}{\left[\log\left(4^{8p}+1\right)\right]^{1/4}}b\frac{\left[\log(n+1)\right]^{1/4}}{\sqrt{n}}.
\end{equation*}

\textbf{Case 2.} If $n \ge 4^{8p}.$
Set
\begin{equation*}L \coloneqq \left\lfloor
\frac{\log n}{4p\log 4}\right\rfloor.
\end{equation*}
Since $n\ge 4^{8p}$, we have
\begin{equation*}
1 \le \frac{\log n}{8p\log 4}\le L \le \frac{\log n}{4p\log 4}.
\end{equation*}
For every $k\in\{1,\ldots,L\}$, define
\begin{equation*}
\omega_k \coloneqq\frac{4^{2p(L-k)}}{n\sqrt L}, \quad b_k \coloneqq
\left(\frac{b^2}{2\kappa_{p,d}^2Ls_k}\right)^{1/(2p)},
\end{equation*}
and
\begin{equation*}
a_k \coloneqq\frac{2}{3}b_k, \quad\gamma_k\coloneqq\frac{1}{4}\left(
\frac{\omega_k}{n\sqrt L}\right)^{1/2}.
\end{equation*}
We first verify that these parameters satisfy the assumption of Lemma~\ref{lem:rademacher-block-construction}. Since $\omega_{k+1}=4^{-2p}\omega_k$, we have
\begin{equation*}
b_{k+1}=4b_k\qquad\text{and}\quad a_{k+1}=4a_k.
\end{equation*}
Moreover,
\begin{equation*}
b_k=\frac{3}{2}a_k<4a_k=a_{k+1},
\end{equation*}
and therefore
\begin{equation*}
0<a_1<b_1<a_2<b_2<\cdots<a_L<b_L.
\end{equation*}
Besides, we also have
\begin{equation*}
\Omega \coloneqq\sum_{k=1}^L\omega_k=\frac{1}{n\sqrt L}\sum_{j=0}^{L-1}4^{2pj}\le  \frac{4^{2pL}}{(4^{2p}-1)n\sqrt L}.
\end{equation*}
Since $L \le \log n /(4p\log 4)$, then
\begin{equation*}
4^{2pL} \le 4^{2p\log n/(4p\log4)} \le \exp\left(2p\log 4\frac{\log n}{4p\log 4} \right)=\sqrt n.
\end{equation*}
It follows that
\begin{equation*}
\Omega \le\frac{1} {(4^{2p}-1)\sqrt{nL}} \le 1.
\end{equation*}
Furthermore, since $\omega_k\ge \omega_L=\frac{1}{n\sqrt L},$ we obtain
\begin{equation*}
\frac{\gamma_k}{\omega_k} = \frac{1}{4\sqrt{n\sqrt L\omega_k}} \le \frac{1}{4}.
\end{equation*}
Thus,
\begin{equation*}
0\le\gamma_k\le\frac{\omega_k}{2},\quad k=1,\ldots,L.
\end{equation*}

On the other hand, by the definition of $b_k$,
\begin{equation*}
\omega_kb_k^{2p} = \frac{b^2}{2\kappa_{p,d}^2L}.
\end{equation*}
Consequently,
$$\kappa_{p,d} \left( \sum_{k=1}^L\omega_kb_k^{2p}
\right)^{1/2}= \kappa_{p,d}\left( \frac{b^2}{2\kappa_{p,d}^2}
\right)^{1/2}= \frac{b}{\sqrt2}\le b.
$$
For $\varepsilon \in \{-1,1\}^L$, we have
\begin{equation*}
\int_{\mathbb{S}^{d-1}} \left(\mathbb{E}_{\mu_0}\left|\langle X,\theta \rangle \right|^{2p}\right)^{1/2}d\sigma(\theta) =\kappa_{p,d} \left[\sum_{k=1}^{L} \frac{\omega_k}{2}\left(a^{2p}_k+b^{2p}_k\right)\right]^{1/2} \le \kappa_{p,d}\left(\sum_{k=1}^L\omega_kb^{2p}_k\right)^{1/2} \le b,
\end{equation*}
and
\begin{align*}
\int_{\mathbb{S}^{d-1}} \left(\mathbb{E}_{\mu_\varepsilon}\left|\langle X,\theta \rangle \right|^{2p}\right)^{1/2}d\sigma(\theta)& \le \kappa_{p,d}\left[\sum_{k=1}^{L}\left(\frac{\omega_k}{2}+\varepsilon_k\gamma_k\right)a^{2p}_k+\sum_{k=1}^{L}\left(\frac{\omega_k}{2}-\varepsilon_k\gamma_k\right)b^{2p}_k\right]^{1/2}\\
&\le \kappa_{p,d}\left[\sum_{k=1}^{L}\left(\frac{\omega_k}{2}+\gamma_k\right)a^{2p}_k+\sum_{k=1}^{L}\left(\frac{\omega_k}{2}+\gamma_k\right)b^{2p}_k\right]^{1/2}\\
& \le \kappa_{p,d}\left[\sum_{k=1}^{L}\omega_k\left(a^{2p}_k+b^{2p}_k\right)\right]^{1/2}\\
&\le \kappa_{p,d}\left(2\sum_{k=1}^{L} \omega_kb^{2p}_k\right)^{1/2}\\
&=b.
\end{align*}
Therefore,
\begin{equation*}
\mu_0,\mu_{\varepsilon} \in \mathcal K_{p,2p}(b) \quad \text{for every }
\varepsilon\in\{-1,1\}^L.
\end{equation*}
Besides, since $b_k-a_k=\frac{b_k}{3}$, we have
$$
\operatorname{SW}_p^p(\mu_{\varepsilon},\mu_0)=\kappa_{p,d} \sum_{k=1}^L
\gamma_k(b_k-a_k)^p= \frac{\kappa_{p,d}}
{4\cdot 3^p\sqrt nL^{1/4}} \sum_{k=1}^L \sqrt{\omega_k}b_k^p.
$$
For every $k\in\{1,\ldots,L\}$,
\begin{equation*}
\sqrt{\omega_k}b_k^p =\left(\omega_kb_k^{2p}\right)^{1/2}= \frac{b}{\sqrt2\kappa_{p,d}\sqrt L}.
\end{equation*}
Hence,
\begin{equation*}
\operatorname{SW}_p^p(\mu_\varepsilon,\mu_0)=\frac{b}{4\sqrt2\,3^p} \frac{L^{1/4}}{\sqrt n}.
\end{equation*}
Moreover,
\begin{equation*}
\frac{\gamma_k^4}{\omega_k^2} = \frac{1}{256n^2L},
\end{equation*}
and therefore
\begin{equation*}
8n^2 \sum_{k=1}^L\frac{\gamma_k^4}{\omega_k^2} = \frac{1}{32}.
\end{equation*}
The total variation bound in
Lemma~\ref{lem:rademacher-block-construction} then yields
\begin{equation*}
\operatorname{TV}(\mathsf P_0,\mathsf P_1) \le \frac{1}{2}
\left(e^{1/32}-1\right)^{1/2} < \frac{1}{4},
\end{equation*}
where $\mathsf P_0 \coloneqq \mu_0^{\otimes n}$ and $\mathsf P_1 \coloneqq 2^{-L} \sum_{\varepsilon\in\{-1,1\}^L} \mu_\varepsilon^{\otimes n}$.

Finally, applying Lemma~\ref{lem:twopointmixtureversion} with
$\mathcal A=\{-1,1\}^L$ and $\nu_0=\mu_0$, we have
$$\mathcal R_{n,m}^{\star}\left(\mathcal K_{p,2p}(b)\right) \ge
\frac{1}{4} \operatorname{SW}_p^p(\mu_{\varepsilon},\mu_0) \left[ 1-\operatorname{TV}(\mathsf P_0,\mathsf P_1) \right] \ge  \frac{3b}{64\sqrt23^p} \frac{L^{1/4}}{\sqrt n}.$$
Since $L \ge \frac{\log n}{8p\log 4}$, we conclude that
\begin{equation*}
\mathcal R_{n,m}^{\star}\bigl(\mathcal K_{p,2p}(b)\bigr)
\ge \frac{3}{64\sqrt23^p(8p\log 4)^{1/4}}b\frac{(\log n)^{1/4}}{\sqrt n}.
\end{equation*}
Finally, since $n\ge 4^{8p}\ge 2$,
\begin{equation*}
\log(n+1) \le \log(2n) = \log 2+\log n\le 2\log n.
\end{equation*}
Therefore,
\begin{equation*}
\mathcal R_{n,m}^{\star}\left(\mathcal K_{p,2p}(b)\right)
\ge \frac{3}{64\sqrt23^p(16p\log 4)^{1/4}}b\frac{\left[\log(n+1)\right]^{1/4}}{\sqrt n}.
\end{equation*}
Combining the two cases, we conclude that
\begin{equation*}
\mathcal R_{n,m}^{\star}\left(\mathcal K_{p,2p}(b)\right) \gtrsim_p b \frac{\left[\log (n+1)\right]^{1/4}}{\sqrt{n}},
\end{equation*}
\end{proof}
\section{Comparison with the approach of \cite{manole2022minimax}}
\label{appendix:comparisonmanole}
The main technical difficulty, both in the approach of
\citet{manole2022minimax} and in ours, is the control of weighted empirical
quantile errors. For a fixed $\theta\in\mathbb S^{d-1}$, let $F_\theta$ and $G_\theta$ denote the distribution functions of $\Pi^\theta_\#\mu$ and
$\Pi^\theta_\#\nu$, respectively, and let $F_{n,\theta}$ and
$G_{m,\theta}$ denote the empirical distribution functions associated
with $\Pi^\theta_\#\widehat\mu_n$ and
$\Pi^\theta_\#\widehat\nu_m$, respectively.
% Throughout this appendix, for any probability measure
% $\Delta $ on $\mathbb R^d$, we write
% \begin{equation*}
% M_{\Delta,r}(\theta)\coloneqq\int_{\mathbb R^d}|\langle x,\theta\rangle|^rd\Delta(x),\qquad \theta\in\mathbb S^{d-1}.
% \end{equation*}

To illustrate the comparison, consider a representative term of the form
\begin{equation*}
T_{n,\theta}\coloneqq\int_0^1\left|G_\theta^{-1}(u)\right|^{p-1}\left|F_{n,\theta}^{-1}(u)-F_\theta^{-1}(u)\right|\,du.
\end{equation*}

The argument of \citet{manole2022minimax} first uses the projected
$r$-th moments to control simultaneously the population and empirical
quantiles. In our notation, their bound takes the form
\begin{equation*}
\max\left\{\left|F_\theta^{-1}(u)\right|,
\left|G_\theta^{-1}(u)\right|,\left|F_{n,\theta}^{-1}(u)\right|,
\left|G_{m,\theta}^{-1}(u)\right|\right\}\lesssim_{r}\left(\frac{
M_{\mu,r}(\theta)+M_{\nu,r}(\theta)+M_{\widehat\mu_n,r}(\theta)+M_{\widehat\nu_m,r}(\theta)
}{u(1-u)
}\right)^{1/r};
\end{equation*}
see \citet[Appendix~D, Step~0]{manole2022minimax}.
They then use a peeling argument near the two endpoints to control this
singular factor; see
\citet[Appendix~D, Steps~1--2]{manole2022minimax}.
When $r>2p$, the layerwise contributions obtained in their argument form
a summable series.

The remaining difficulty is to control
$\left|F_{n,\theta}^{-1}-F_\theta^{-1}\right|$ on these layers. For this
purpose, \citet{manole2022minimax} use a relative VC inequality. For
$\varepsilon\in(0,1)$, write
\begin{equation*}
\rho_{\varepsilon,n}\coloneqq\sqrt{\frac{16}{n}\left[\log\left(\frac{16}{\varepsilon}\right)+\log(2n+1)\right]},
\end{equation*}
for the quantity denoted by $\nu_{\varepsilon,n}$ in
\citet{manole2022minimax}. Let $\gamma_{\varepsilon,n}$ and
$\eta_{\varepsilon,n}$ be the two deterministic functions defined in
\citet[Eq.~(44)]{manole2022minimax}; these functions depend on
$\varepsilon$ and $n$ through $\rho_{\varepsilon,n}$. For each fixed
$\theta\in\mathbb S^{d-1}$, define
\begin{equation*}
\mathcal A_{\varepsilon,n}(\theta)\coloneqq\left\{
F_\theta^{-1}\bigl(\gamma_{\varepsilon,n}(u)\bigr)\le
F_{n,\theta}^{-1}(u)\le F_\theta^{-1}\bigl(\eta_{\varepsilon,n}(u)\bigr),
\quad\text{for all }u\in(0,1)
\right\}.
\end{equation*}
The relative VC inequality gives
\begin{equation*}
\mathbb P\bigl(\mathcal A_{\varepsilon,n}(\theta)\bigr)\ge 1-\varepsilon;
\end{equation*}
see \citet[Appendix~D, Step~0]{manole2022minimax} and
\citet[Appendix~G.6]{manole2022minimax}.

Combining this uniform quantile control with the preceding peeling
argument, their argument yields, for
$\rho_{\varepsilon,n}\le1$, a bound of the form
\begin{equation*}
T_{n,\theta}\lesssim_{p,r}\rho_{\varepsilon,n}\left(M_{\mu,r}(\theta)+M_{\nu,r}(\theta)+M_{\widehat\mu_n,r}(\theta)+M_{\widehat\nu_m,r}(\theta)
\right)^{p/r}
\end{equation*}
on $\mathcal A_{\varepsilon,n}(\theta)$; see
\citet[Appendix~D, Steps~2--4, in particular Eq.~(54)]{manole2022minimax}.

They subsequently control the empirical moment contribution and convert
this high-probability estimate into an expectation bound; see
\citet[Appendix~D, Step~4]{manole2022minimax}. Importantly, although the
peeling argument makes the endpoint contributions summable when $r>2p$,
it does not remove the prefactor $\rho_{\varepsilon,n}$. The logarithmic
dependence on $n$ is already present in this quantity through the term
$\log(2n+1)$ inherited from the relative VC inequality. Consequently, the
final expectation estimate retains the factor $\log n / n$ see \citet[Appendix~D, Step~5]{manole2022minimax}.

Our approach starts from the same reduction to weighted empirical
quantile errors, but differs in how the quantile error itself is
controlled. Rather than using a uniform high-probability bound for the
empirical quantile process, we work directly in expectation. More
precisely, Proposition~\ref{prop:peelingtool} establishes the localized
estimate
\begin{equation*}
\mathbb E\int_0^\alpha\left|F_{n,\theta}^{-1}(u)-F_\theta^{-1}(u)\right|du\lesssim_r\frac{M_{\mu,r}(\theta)^{1/r}}{\sqrt n}
\alpha^{1/2-1/r},\qquad \alpha\in[0,1],
\end{equation*}
together with the analogous upper-tail bound (see Remark~\ref{rem:peelingtool}).
\section{Comparison with \citet{staudt2025convergence}}\label{appendix:Staudt}
In this appendix, we aim to show how the results of \citet{staudt2025convergence} specialize to our setting and clarify the rates they yield for our considered problems. Moreover, as mentioned in Section~\ref{sec:relatedwork}, we provide a direct verification showing that, although their alternative approach removes the independence assumption between two samples, its assumption do not cover the cost $c(x,y)=|x-y|^p$ considered in our setting.
\subsection{Specialization of their main bound}\label{appendix:SpecializationStaudt}
We first recall the one-dimensional version of the result of \citet{staudt2025convergence}, which will be used below to derive the rates yielded by their approach in our setting
\begin{theorem}[One-dimensional version of Corollary~3.1 in \citet{staudt2025convergence}]\label{theorem:staudt}
Let $p \ge 1$, $p<q\le 2p$, $\epsilon>0$, and $1 \le n\le m$. Let $\eta_1$, $\eta_2 \in \mathcal{P}_{q+\epsilon}\left(\mathbb{R}\right)$, and let $\widehat{\eta}_{1,n}$ and $\widehat{\eta}_{2,m}$ denote the empirical measures associated with two independent samples from $\eta_1$ and $\eta_2$ respectively. 
\begin{equation*}
\mathbb{E}\left|\operatorname{W}^p_p\left(\widehat{\eta}_{1,n},\widehat{\eta}_{2,m}\right)-\operatorname{W}^p_p\left(\eta_1,\eta_2\right)\right|\lesssim n^{-1/2}+n^{-(q-p)/q},
\end{equation*}
where the implicit constant depends on $p,q,\epsilon$ and the $(q+\epsilon)$- th moments of $\eta_1$ and $\eta_2$.
\end{theorem}
Back to our setting, for $\mu,\nu \in \mathcal{K}_{p,r}(b)$, in Section~\ref{sec:proofsupperbound}, we obtained
\begin{equation*}
\mathbb{E}\left|\operatorname{SW}^p_p\left(\widehat{\mu}_n,\widehat{\nu}_m\right)-\operatorname{SW}^p_p\left(\mu,\nu\right)\right| \le \int_{\mathcal{G}}\mathbb{E}\left|\operatorname{W^p_p\left(\Pi^{\theta}_{\#}\widehat{\mu}_n,\Pi^{\theta}_{\#}\widehat{\nu}_m\right)-\operatorname{W}^p_p\left(\Pi^{\theta}_{\#}\mu,\Pi^{\theta}_{\#}\nu\right)}\right|d\sigma(\theta),
\end{equation*}
where $\mathcal{G}=\left\{\theta \in \mathbb{S}^{d-1}:M_{\mu,r}(\theta)<\infty \text{ and } M_{\nu,r}(\theta)<\infty\right\}$.

We note that if $\mathbb{E}_{X\sim \mu}\left|\langle X,\theta\rangle \right|^r=\mathbb{E}_{Y\sim \nu}\left|\langle Y,\theta\rangle \right|^r=0$, then $\Pi^{\theta}_{\#}\mu=\Pi^{\theta}_{\#}\nu =\delta_0$. Consequently, in this case $\mathbb{E}\left|\operatorname{W}^p_p\left(\Pi^{\theta}_{\#}\widehat{\mu}_n,\Pi^{\theta}_{\#}\widehat{\nu}_m\right)-\operatorname{W}^p_p\left(\Pi^{\theta}_{\#}\mu,\Pi^{\theta}_{\#}\nu\right)\right|=0$. As a result, we have
\begin{equation*}
\mathbb{E}\left|\operatorname{SW}^p_p\left(\widehat{\mu}_n,\widehat{\nu}_m\right)-\operatorname{SW}^p_p\left(\mu,\nu\right)\right|\le\int_{\mathcal{G}^*}\mathbb{E}\left|\operatorname{W}^p_p\left(\Pi^{\theta}_{\#}\widehat{\mu}_n,\Pi^{\theta}_{\#}\widehat{\nu}_m\right)-\operatorname{W}^p_p\left(\Pi^{\theta}_{\#}\mu,\Pi^{\theta}_{\#}\nu\right)\right|d\sigma(\theta),
\end{equation*}
where $\mathcal{G}^*\coloneqq \left\{\theta \in \mathbb{S}^{d-1}:0<M_{\mu,r}(\theta)+M_{\nu,r}(\theta)<\infty\right\}$.

Now, we fix $\theta \in \mathcal{G}^*$ and set $a_{\theta}\coloneqq \left(M_{\mu,r}(\theta)+M_{\nu,r}(\theta)\right)^{1/r}$. We denote by $\mu_{new}$ and $\nu_{new}$ the distributions obtained by rescaling $\mu$ and $\nu$ by the factor $a^{-1}_{\theta}$ respectively.

% $$\mu_{\text{new}}\coloneqq \frac{\mu}{a_{\theta}} \quad \nu_{\text{new}}\coloneqq \frac{\nu}{a_{\theta}}$$

Therefore, obviously, $M_{\mu_{\text{new}},r}(\theta)$ and $M_{\mu_{\text{new}},r}(\theta)$ are finite. We have
\begin{itemize}
    \item If $r>2p$, applying Theorem~\ref{theorem:staudt} (with $p=p$, $\epsilon=r-2p>0$, $q=2p$, $\eta_1=\Pi^{\theta}_{\#}\mu_{\text{new}}$, and $\eta_2=\Pi^{\theta}_{\#}\nu_{\text{new}}$), we have
    \begin{equation*}
\mathbb{E}\left|\operatorname{W}^p_p\left(\Pi^{\theta}_{\#}\widehat{\mu}_{\text{new},n},\Pi^{\theta}_{\#}\widehat{\nu}_{\text{new},m}\right)-\operatorname{W}^p_p\left(\Pi^{\theta}_{\#}\mu_{\text{new}},\Pi^{\theta}_{\#}\nu_{\text{new}}\right)\right|\lesssim n^{-1/2}+n^{-1/2},
    \end{equation*}
where $\widehat{\mu}_{\text{new},n}$ and $\widehat{\nu}_{\text{new},m}$ are empirical measures associated with two independent samples from $\mu_{\text{new}}$ and $\nu_{\text{new}}$ respectively. We then deduce that
\begin{align*}
\mathbb{E}\left|\operatorname{W}^p_p\left(\Pi^{\theta}_{\#}\widehat{\mu}_{n},\Pi^{\theta}_{\#}\widehat{\nu}_{m}\right)-\operatorname{W}^p_p\left(\Pi^{\theta}_{\#}\mu,\Pi^{\theta}_{\#}\nu\right)\right|&\lesssim n^{-1/2}\left(M_{\mu,r}\left(\theta\right)+M_{\nu,r}\left(\theta\right)\right)^{p/r}\\
& \lesssim n^{-1/2}\left[\left(M_{\mu,r}(\theta)\right)^{p/r}+(M_{\nu,r}(\theta))^{p/r}\right].
    \end{align*}
Finally, integrating over $\theta \in \mathcal{G}^*$ and using the definition of $\mathcal{K}_{p,r}(b)$ yields the parametric rate of order $n^{-1/2}$.
    \item If $p<r\le 2p$, since only an $r$-th moment is available, we take $q=r-\epsilon$ with $0<\epsilon <r-p$. Applying Theorem~\ref{theorem:staudt} (with $p=p$, $q=r-\epsilon$, $\eta_1=\Pi^{\theta}_{\#}\mu_{\text{new}}$, and $\eta_2=\Pi^{\theta}_{\#}\nu_{\text{new}}$), we obtain
    \begin{equation*}
\mathbb{E}\left|\operatorname{W}^p_p\left(\Pi^{\theta}_{\#}\widehat{\mu}_{\text{new},n},\Pi^{\theta}_{\#}\widehat{\nu}_{\text{new},m}\right)-\operatorname{W}^p_p\left(\Pi^{\theta}_{\#}\mu_{\text{new}},\Pi^{\theta}_{\#}\nu_{\text{new}}\right)\right|\lesssim n^{-1/2}+n^{-\frac{r-p-\epsilon}{r-\epsilon}}.
    \end{equation*}
Since $r-\epsilon<2p$, we have $\frac{r-p-\epsilon}{r-\epsilon}<\frac{1}{2}$. As a consequence, we obtain
\begin{equation*}
\mathbb{E}\left|\operatorname{W}^p_p\left(\Pi^{\theta}_{\#}\widehat{\mu}_{\text{new},n},\Pi^{\theta}_{\#}\widehat{\nu}_{\text{new},m}\right)-\operatorname{W}^p_p\left(\Pi^{\theta}_{\#}\mu_{\text{new}},\Pi^{\theta}_{\#}\nu_{\text{new}}\right)\right|\lesssim n^{-\frac{r-p-\epsilon}{r-\epsilon}}.
\end{equation*}
Proceeding as in the previous case yields a rate of order $n^{-\frac{r-p-\epsilon}{r-\epsilon}}$ which is strictly slower than the rate obtained in our work.
\end{itemize}
\subsection{Why their Appendix A assumptions do not cover $c(x,y)=|x-y|^p$.}\label{appendix:verificationStaudt}
We now focus on the alternative approach developed in \citet[Appendix~A]{staudt2025convergence}. Instead of stating the full conditions, we specialize the relavant part of their Assumption to the one-dimensional powered Wasserstein cost $c(x,y)=\left|x-y\right|^p$ for $x,y \in \mathbb{R}$. We first recall some useful notations. Let $c_{\mathcal{X}},c_{\mathcal{Y}}:\mathbb{R} \to \mathbb{R}_{+}$ be marginal cost functions satisfying
\begin{equation*}
c(x,y) \le c_{\mathcal{X}}(x)+c_{\mathcal{Y}}(y) \quad \forall x,y \in \mathbb{R},
\end{equation*}
and, for $R\ge 0$, define the corresponding cost balls by
\begin{equation*}
B_{\mathcal{X}}(R)\coloneqq \left\{x \in \mathbb{R}:c_{\mathcal{X}}(x) \le R\right\}, \quad B_{\mathcal{Y}}(R) \coloneqq \left\{y \in \mathbb{R}:c_{\mathcal{Y}}(y) \le R\right\}.
\end{equation*}
For our purpose, it suffices to consider the following two requirements of Assumption DBC in \citet[Appendix~A]{staudt2025convergence}. First, there must exist classes $\mathcal{F}$ and $\mathcal{G}$ of real-valued Borel functions and a set $\Phi_{c} \subset \mathcal{F} \times \mathcal{G}$ such that
\begin{equation*}
\inf_{\pi \in \Gamma \left(\rho, \xi\right)}\int_{\mathbb{R} \times \mathbb{R}}\left|x-y\right|^pd\pi(x,y)=\sup_{(f,g)\in \Phi_{c}}\left\{\int_{\mathbb{R}}fd\rho +\int_{\mathbb{R}}gd \xi\right\}
\end{equation*}
for every $\rho,\xi \in \mathcal{P}\left(\mathbb{R}\right)$ such that $\int c_{\mathcal{X}} d\rho+\int c_{\mathcal{Y}}d\xi <\infty$. This condition is condition $A.2$ in \cite{staudt2025convergence}. 

Second, there must exist a constant $\kappa>0$ and a pinning point $x_0 \in \mathbb{R}$ such that
\begin{equation*}
\sup_{f \in \mathcal{F}}\sup_{x \in B_{\mathcal{X}}(R)}\left|f(x)-f(x_0)\right| \le \kappa R, \quad \forall R\ge 1.
\end{equation*}
This is a part of condition A.3b in \citet{staudt2025convergence}. We now show that these two requirements cannot hold simultaneously when $p>1$.

Suppose, to the contrary, that those conditions are satisfied. Let $x_0$ be the pinning point mentioned above. We fix $x_1 \coloneqq x_0+1$. We choose $R_0 \ge \max \left\{1,c_{\mathcal{X}}(x_1)\right\}$, so that $x_1 \in B_{\mathcal{X}}(R_0)$. Hence, the second condition yields 
\begin{equation}\label{eq:appendixstaudtboundedconidtion}
\left|f(x_1)-f(x_0)\right| \le \kappa R_0 \quad \text{for every } f \in \mathcal{F}.
\end{equation}
Fix $t>x_1$. Obviously, we have
\begin{equation*}
\inf_{\pi \in \Gamma \left(\delta_{x_0},\delta_t{}\right)} \int_{\mathbb{R} \times \mathbb{R}}\left|x-y\right|^pd\pi(x,y) =\left|t-x_0\right|^p.
\end{equation*}
Hence, the first assumption yields that there exists $(f_t,g_t) \in \Phi_{c}$ such that 
\begin{equation*}
f_{t}(x_0)+g_t(t) \ge \left|t-x_0\right|^p-1.
\end{equation*}
On the other hand, by the first assumption, we also have
\begin{equation*}
\left|t-x_1\right|^p = \sup_{(f,g)\in \Phi_c}\left\{f(x_1)+g(t)\right\}.
\end{equation*}
Since $(f_t,g_t) \in \Phi_{c}$, it follows that
\begin{equation*}
f_t(x_1)+g_t(t) \le \left|t-x_1\right|^p.
\end{equation*}
Consequently,
\begin{equation*}
f_t(x_0)-f_t(x_1) \ge \left|t-x_0\right|^p-\left|t-x_1\right|^p-1=(t-x_0)^p-(t-x_0-1)^p-1.
\end{equation*}
When $t \to +\infty$, we have $f_t(x_0)-f_t(x_1) \to +\infty$, which contradicts the uniform bounded in \eqref{eq:appendixstaudtboundedconidtion}.
\end{appendix}
%%%%%%%%%%%%%%%%%%%%%%%%%%%%%%%%%%%%%%%%%%%%%%
%% Acknowledgements                         %%
%% should be provided in the                %%
%% Acknowledgements section.                %%
%%%%%%%%%%%%%%%%%%%%%%%%%%%%%%%%%%%%%%%%%%%%%%
\begin{acks}[Acknowledgments]
The author used ChatGPT to assist with mathematical reasoning and proof development, including computational checks and exploring mathematical intuition. ChatGPT was also used for writing support and improving the presentation of the manuscript. The author has reviewed the AI-assisted material and takes full responsibility for the correctness of the mathematical results and the final manuscript.

The author would like to thank Mohamed Hebiri for introducing him to the peeling technique. He is especially grateful to Nicolas Schreuder for his dedicated guidance, scientific and emotional support throughout this work. Although this paper displays only the author's name, it would not have been possible without Nicolas's support and encouragement.
\end{acks}

\bibliographystyle{plainnat}
\bibliography{references}

\end{document}